\documentclass{amsart}

\usepackage{amsmath, amscd, amssymb}
\usepackage{graphpap, color}

\usepackage[mathscr]{eucal}
\usepackage{pstricks}
\usepackage{color}
\usepackage{cancel}

\usepackage[all]{xy}
\numberwithin{equation}{section}

\def\cD{\mathcal D}
\def\cE{\mathcal E}

\def\cH{\mathcal H}

\def\cM{\mathcal M}

\def\cW{\mathcal W}
\def\cX{\mathcal X}
\def\cY{\mathcal Y}
\def\cZ{\mathcal Z}

\def\sA{\mathscr{A}}
\def\sE{\mathscr{E}}
\def\sF{\mathscr{F}}

\def\sI{\mathscr{I}}

\def\sK{{\mathscr K}}

\def\sO{{\mathscr O}}

\def\sU{\mathscr{U}}
\def\sV{\mathscr{V}}

\def\fA{\mathfrak A}
\def\fB{\mathfrak B}
\def\fC{\mathfrak C}
\def\fD{\mathfrak D}
\def\fE{\mathfrak E}

\def\fI{\mathfrak I}
\def\fP{\mathfrak P}
\def\fR{\mathfrak R}
\def\fS{\mathfrak S}
\def\fT{\mathfrak T}
\def\fW{\mathfrak W}
\def\fX{\mathfrak{X}}
\def\fY{\mathfrak{Y}}

\def\C{\mathfrak C}
\def\D{\mathfrak D}

\def\P{\mathfrak P}
\def\Q{\mathfrak Q}

\def\X{\mathfrak X}
\def\Y{\mathfrak Y}

\def\fm{\mathfrak m}

\def\fz{\mathfrak{z}}
\def\p{{\mathfrak p}}

\def\bA{\mathbb{A}}

\def\bP{\mathbb{P}}
\def\bQ{\mathbb{Q}}

\newcommand{\LL}{\mathbb{L}}

\newcommand{\QQ}{\mathbb{Q}}

\newcommand{\ZZ}{\mathbb{Z}}
\def\NN{\mathbb N}

\def\bk{{\mathbf k}}
\def\kk{\bk}

\def\Qu{{\mathfrak{Q}}}

\newtheorem{prop}{Proposition}[section]
\newtheorem{theo}[prop]{Theorem}
\newtheorem{lemm}[prop]{Lemma}
\newtheorem{coro}[prop]{Corollary}
\newtheorem{rema}[prop]{Remark}
\newtheorem{exam}[prop]{Example}
\newtheorem{defi}[prop]{Definition}

\newtheorem{defiprop}[prop]{Definition-Proposition}

\newtheorem{defi-prop}[prop]{Definition-Proposition}
\newtheorem{sect-sele}[prop]{Section-Selection}

\def\lsta{_*}

\def\llt{_{(t)}}

\def\sta{^\ast}

\def\virt{^{\mathrm{vir}}}
\def\upmo{^{-1}}
\def\shar{^\dagger}
\def\dual{^{\vee}}

\def\utf{^{\mathrm{t.f.}}}

\def\dpri{^{\prime\prime}}

\def\beq{\begin{equation}}
\def\eeq{\end{equation}}
\def\lab#1{\label{#1}[{#1}]\  }
\let\lab=\label

\def\defeq{:=}
\def\bl{\bigl(}
\def\br{\bigr)}
\def\and{\quad\text{and}\quad}
\def\lra{\longrightarrow}
\def\vsp{\vskip5pt}

\let\sub=\subset

\def\Po{{\mathbb P^1}}

\def\ti{\tilde}
\def\wti{\widetilde}

\def\mapright#1{\,\smash{\mathop{\lra}\limits^{#1}}\,}

\newcommand{\Del}{\Delta}

\def\si{\sigma}

\def\Lam{\Lambda}

\def\xm{X[m]}
\def\xn{X[n]}
\def\cm{C[m]}
\def\cn{C[n]}
\def\xnpo{X[n+1]}
\def\cnpo{C[n+1]}

\def\yn{Y[n]}

\def\yn{Y[n_-,n_+]}

\def\unn{^{\!n_-+n_+}}
\def\nn{n_-,n_+}

\def\Gm{G_{\mathrm{\!m}}}
\def\gn{\Gm^{{}n}}

\def\Ao{{{\bA}^{\!1}}}

\def\Ano{\bA^{\!n+1}}
\def\Anpo{{\bA^{\!n+1}}}
\def\Anpt{{\bA^{\!n+2}}}

\def\Amo{\bA^{\!m+1}}

\def\Aut{\mathrm{Aut}}
\def\Isom{\underline{\mathrm{Isom}}}
\def\Sets{\mathrm{Sets}}
\def\Hilb{\mathrm{Hilb}}
\def\Quot{\mathrm{Quot}}

\def\Spec{\mathrm{Spec}\,}
\def\spec{{\mathrm{Spec}\,}}
\def\coker{\mathrm{coker}}
\def\im{\mathrm{Im}}

\def\Tor{\mathrm{Tor}}

\def\ch{\mathrm{ch}}
\def\ann{\mathrm{ann}}
\def\loc{\mathrm{loc}}
\def\spl{\mathrm{spl}}

\def\st{\text{st}}
\def\RHom{R\cH om}

\def\ev{\mathrm{ev}}
\def\pnn{p}
\def\splitset{\Lambda_\beta^{\mathrm{spl}}}

\def\err{\mathrm{Err}}
\let\Err=\err
\def\exc{\mathrm {Err}}
\def\excz{\mathrm{Err}{\sF}}

\def\pfrmob{\phi\colon p^*\sV\to\sF}
\def\frmobo{\phi_1\colon p^*\sV\to\sF_1}
\def\frmobt{\phi_2\colon p^*\sV\to\sF_2}

\def\cohs{\varphi\colon\sO_{\xn_0}\to\sF}
\def\cohsrel{\varphi\colon\sO_{Y[\nn]_0}\to\sF}

\def\fQuot{\mathfrak{Quot}}

\def\Qx{{\mathfrak{Quot}_{\fX/\fC}^{\sV}}}
\def\Qxh{{\mathfrak{Quot}_{\fX/\fC}^{\sV,P}}}
\def\Qxp{{\mathfrak{Quot}_{\fX/\fC}^{\sV,P}}}

\def\Qyd{{\mathfrak{Quot}^{\sV_0}_{\fD_\pm\sub\fY}}}
\def\Qydp{{\mathfrak{Quot}^{\sV_0,P}_{\fD_\pm\sub\fY}}}

\def\Ixp{{\fI^P_{\fX/\fC}}}

\def\Ixpo{\fI^P_{\fX_0/\fC_0}}

\def\Mx{\fP_{\fX/\fC}}
\def\Mxh{\fP^P_{\fX/\fC}}
\def\Myd{\fP_{\fD_\pm\sub\fY}}
\def\Mydh{\fP^P_{\fD_\pm\sub\fY}}

\def\Idel{\text{$\fI$\hskip 0.1mm\raisebox{5pt}{$\scriptstyle\delta$}\hskip -1.6mm\raisebox{-2pt}{$\scriptstyle\fX_0^{\dagger}/\fC_0^\dagger$}}}
\def\scrIdel{\text{$\fI$\hskip 0.1mm\raisebox{3pt}{$\scriptstyle\delta$}\hskip -1.6mm\raisebox{-2pt}{$\scriptstyle\fX_0^{\dagger}/\fC_0^\dagger$}}}

\def\Cdel{\fC_0^{\dagger,\delta}}

\def\cMm{\cM^\delta_-}
\def\cMp{\cM^\delta_+}

\def\scrAdel{\text{$\fA_\diamond$\hskip -1.1mm\raisebox{3pt}{$\scriptstyle{{\delta}_{+},{\delta}_0}$}}}
\def\scrAdelm{\text{$\fA_\diamond$\hskip -1.1mm\raisebox{3pt}{$\scriptstyle{{\delta}_{-},{\delta}_0}$}}}

\def\LMdel{\text{$L$\raisebox{6pt}{$\scriptstyle\vee$}\hskip -2mm \raisebox{-3pt}{$\scriptstyle\scrIdel/\fC_0^{\dagger,\delta}$}}}

\def\Lprod{\text{$L$\raisebox{6pt}{$\scriptstyle\vee$}\hskip -2mm \raisebox{-3pt}
{$\scriptstyle\cM^\delta_+/\scrAdel\times \cM^\delta_-/\scrAdelm$}}}

\def\LMdelt{\text{$L$\raisebox{6pt}{$\scriptstyle\vee$}\hskip -2mm \raisebox{-3pt}{$\scriptstyle\scrIdel/\cM^\delta_+\times\cM^\delta_-$}}}

\title[Good degeneration of Quot-schemes and coherent systems]
{Good degeneration of Quot-schemes and \\coherent systems}

\author{Jun Li and Baosen Wu}
\address{Department of Mathematics, Stanford University, Stanford, CA 94305}
\email{jli@math.stanford.edu}
\address{Department of Mathematics, Harvard University, Cambridge, MA 02138}
\email{bwu@math.harvard.edu}

\begin{document}

\begin{abstract}
We construct good degenerations of Quot-schemes and coherent systems using the stack of expanded degenerations. We show that these good degenerations are separated and proper DM stacks of finite type. Applying to the projective threefolds, we derive  degeneration formulas for DT-invariants of ideal sheaves and PT stable pair invariants.
\end{abstract}

\maketitle

\section{Introduction}

Good degenerations are a class of degenerations suitable to study the geometry of moduli spaces via degenerations.
Successful applications include the degeneration formula of Gromov-Witten invariants \cite{Li1,Li2}.
In this paper, we will construct the good degenerations of Hilbert schemes, of Grothendieck's Quot-schemes, and of the
moduli of coherent systems introduced by Le Potier \cite{Le}. As applications, we obtain the degeneration formulas of
Donaldson-Thomas invariants of ideal sheaves, and of invariants of PT stable pairs of threefolds.

The degenerations we study in this paper are simple degenerations
$\pi: X\to C$
over pointed smooth curves $0\in C$.

\begin{defi}\label{def-simp}
We say $\pi: X\to C$ is a simple degeneration if
\begin{enumerate}
\item $X$ is smooth, $\pi$ is projective, $\pi$ has smooth fiber over $c\ne 0\in C$;
\item the central fiber $X_0$ has normal crossing singularity and the singular locus $D$ of $X_0$ is smooth;
\item let ${Y}$ be the normalization of $X_0$ and $\ti D={Y}\times_{X_0}D\sub {Y}$, then
$\ti D\to D$ is isomorphic to a union of two copies of $D$.
\end{enumerate}
\end{defi}

We denote the two copies of $\ti D\to D$ by $D_-$ and $D_+$.
We call $(Y, D_\pm)$ the relative pair associated with $X_0$.

We fix a relatively ample line bundle $H$ on $X/C$, and a polynomial $P(v)$; we form the Hilbert scheme $\Hilb_{X_c}^P$
of closed subschemes $Z\sub X_c$ with Hilbert polynomial $\chi^H_{\sO_Z}(v)\defeq\chi(\sO_Z\otimes H^{\otimes v})=P(v)$.
We will use the technique developed by the first named author in \cite{Li1} to find a good degeneration of the
relative Hilbert scheme (denoting $X\sta=X-X_0$ and $C\sta =C-0$)
$$\Hilb_{X\sta/C\sta}^P=\coprod_{c\in C\sta} \Hilb_{X_c}^P.
$$

To fill in the central fiber of this family over $0\in C$, we consider closed subschemes in $X[n]_0$ that are normal to the singular loci of $X[n]_0$; where $X[n]_0$ is by inserting a chain of $n$-copies of the ruled variety (over $D$)
$$\Del=\bP_D(1\oplus N_{D_+/Y}) \
$$
to $D$ in $X_0$, ($\xn_0$ is constructed in the next section,) and normal to the singular loci means that it is flat along the normal direction to the singular loci of $\xn_0$.

The central fiber of the good degeneration has set-theoretic description
$$\left\{Z\sub X[n]_0\, \Big|\,  \begin{array}{l}
\mbox{$n\ge 0$, $Z$ is normal to the singular loci } \\ \mbox{of
$X[n]_0$,   $\Aut_\fX(Z)$ is finite, $\chi^H_{\sO_Z}(v)=P(v)$.}
\end{array}
\right\}\Big/\cong.
$$
Here the equivalence and the automorphism group are defined in the next section.
In case $D$ is irreducible, it has a simple description: two closed subschemes $Z_1,Z_2\sub \xn_0$
are equivalent if there is an isomorphism $\sigma: \xn_0\to\xn_0$ preserving the projections $X[n]_0\to X_0$ such that
$\sigma(Z_1)=Z_2$. The self-equivalences of a $Z\sub X[n]_0$ form a group,
which we denote by $\Aut_\fX(Z)$.  We call $Z$ stable if $\Aut_\fX(Z)$ is finite.
Finally, $\chi^H_{\sO_Z}(v)=\chi(\sO_Z\otimes p\sta H^{\otimes v})$, where $p: \xn_0\to X_0$ is the projection
by contracting the fibers of $\Delta$.

Constructing the stack structure of this set-theoretic description of the central fiber,
and fitting it into the family $\Hilb_{X\sta/C\sta}^P$, is achieved by working with the stack $\fX\to \fC$ of
expanded degenerations. Using $\fX\to\fC$, we prove that the set-theoretic description of good degeneration is a
Deligne-Mumford stack. The first part of this paper is devoted to prove

\begin{theo}
Let $\pi: X\to C$ be a simple degeneration, $H$ be relative ample on $X\to C$, and $P$ be a polynomial.
Then the good degeneration described is a Deligne-Mumford stack proper and separated over $C$; it is of finite type.
\end{theo}

Similar results hold for good degenerations of Grothendieck's Quot-schemes and of coherent systems of Le Potier.
\vsp

The primary goal to construct such a good degeneration is to derive a degeneration formula of Donaldson-Thomas invariants and PT stable pair invariants of threefolds. For simplicity, we only state the degeneration formula in case $Y$ is a union of two irreducible complements $Y=Y_-\cup Y_+$, and $D$ is connected. We let $D_\pm=Y_\pm\cap \ti D$.

Let $\Lambda^{\text{spl}}_P$ be the set of splittings $\delta=(\delta_\pm,\delta_0)$ of $P$,
(i.e. $\delta_++\delta_--\delta_0=P$.) For each $\delta\in\Lambda^{\text{spl}}_P$, we construct the moduli stack
$\fI_{\fY_\pm/\fA_\diamond}^{\delta_\pm,\delta_0}$ of relative ideal sheaves of $(Y_\pm,D_\pm)$.
This moduli space is constructed using the stack $\fD_\pm\sub \fY_\pm$ of expanded pairs of $D_\pm\sub Y_\pm$.
Closed points of this moduli space consists of ideal sheaves $\sI_Z$ of $Y_\pm[n_\pm]$ relative to $D_\pm$,
meaning that $Z$ is normal to the singular loci of $Y_\pm[n_\pm]$ and to $D_\pm$.
This moduli space is also a proper and separated Deligne-Mumford stack of finite type. Furthermore, we have a natural morphisms
$$ \ev_\pm:\fI_{\fY_\pm/\fA_\diamond}^{\delta_\pm,\delta_0}\lra \Hilb_D^{\delta_0},
$$
to the Hilbert scheme of ideal sheaves on $D$ of Hilbert polynomial $\delta_0$,
defined via restricting ideal sheaves on $Y_\pm[n_\pm]$ to its relative divisor $D_\pm$.

Using the evaluation morphisms, we form the fiber product
$$\Idel= \fI_{\fY_-/\fA_\diamond}^{\delta_-,\delta_0}\times_{ \Hilb_D^{\delta_0}} \fI_{\fY_+/\fA_\diamond}^{\delta_+,\delta_0}.
$$
Each $\Idel$ is a closed substack of $\Ixp$, and is indeed a ``virtual'' Cartier divisor.

\begin{theo}\label{cycle0}
Let $\pi: X\to C$ be a simple degeneration of projective threefolds such that $X_0=Y_-\cup Y_+$ is a union of two smooth irreducible components. Let $[\Ixp]\virt\in A_*\Ixp$ be the virtual class of the good degeneration, and let $\triangle$ be the diagonal morphism $ \Hilb_D^{\delta_0}\to  \Hilb_D^{\delta_0}\times \Hilb_D^{\delta_0}$. Then
$i_c^![\Ixp]\virt=[\fI_{X_c}^P]\virt$ for $c\ne 0\in C$,
and
\beq\label{formula0}i_0^! [\Ixp]\virt=\sum _{\delta \in\Lam_P^{\mathrm{spl}} }
\triangle^!\bl [\fI_{\fY_-/\fA_\diamond}^{\delta_-,\delta_0}]\virt\times [\fI_{\fY_+/\fA_\diamond}^{\delta_+,\delta_0}]\virt\br.
\eeq
\end{theo}

\vsp
Using the Chern characters of the universal ideal sheaves, we also obtain the numerical version of the
Donaldson-Thomas invariant and its degeneration, first introduced in the work of Maulik, Nekrasov, Okounkov and Pandharipande \cite{MNOP2}. For a smooth projective threefold $X$ and a polynomial $P(v)=d\cdot v+n$,
we let $\fI_{X}^P$ ($\cong\Hilb_{X}^P$ canonically) be the moduli of ideal sheaves of curves $\sI_Z\subset \sO_{X}$
with Hilbert polynomial $P$; and let $\sI_\cZ\subset\sO_{X\times \fI_{X}^P}$ be its universal family. For any $\gamma\in H^l(X,\ZZ)$,
we define
\[ \ch_{k+2}(\gamma):H_*(\fI_{X}^P,\QQ)\lra H_{*-2k+2-l} (\fI_{X}^P,\QQ),
\]
via
\[ \ch_{k+2}(\gamma)(\xi)=\pi_{2*}(\ch_{k+2}(\sI_\cZ)\cdot \pi_1^*(\gamma)\cap\pi_2^*(\xi)),
\]
where $\pi_1$ and $\pi_2$ are the first and second projection of $X\times \fI_{X}^P$. The
Donaldson-Thomas invariants (in short DT-invariants) with descendent insertions are the degree of the following cycle class
\[\Bigl\langle\prod_{i=1}^r\tilde\tau_{k_i}(\gamma_{i})\Bigr\rangle_{X}^P=
\Big[\prod_{i=1}^r(-1)^{k_i+1}\ch_{k_i+2}(\gamma_{i})\cdot[\fI_{X}^P]\virt\Big]_0,
\]
where $\gamma_{i}$ are cohomology classes of pure degree $l_i$, and $[\cdot]_0$ is taking the dimension zero part of the term inside the bracket. The partition function is
$$Z_{d}\bigg(X;q\Big |\prod_{i=1}^r \tilde \tau_{k_i}(\gamma_{i})
\bigg)=\sum_{n\in \ZZ}\deg\Bigl\langle\prod_{i=1}^r\tilde\tau_{k_i}(\gamma_{i})\Bigr\rangle_{X}^{d\cdot v+n}q^n.
$$

The commonly used form of DT-invariants as introduced in \cite{MNOP2},
uses the moduli $I_n(X,\beta)$ of ideal sheaves of subschemes $Z\sub X$ with fixed curve class
$\beta=[Z]$. In this paper we package the DT-invariant using the moduli $\fI_{X}^P$
of ideal sheaves with fixed Hilbert polynomial.
This enables use to avoid the technical
issue of decomposing curve classes during degenerations. In explicit application, one should be able to
derive the general case after analyzing this issue in details.

Next, we let $\beta_1,\cdots,\beta_m$ be a basis of $H^*(D,\QQ)$. Let $\{C_\eta\}_{|\eta|=k}$ be a
Nakajima basis of the cohomology of $\Hilb_D^k$ (where $\eta$ is a cohomology weighted partition w.r.t. $\beta_i$).
The relative DT-invariants with descendent insertions \cite{MNOP2} are the degree of
\[\Bigl\langle\prod_{i=1}^r\tilde\tau_{k_i}(\gamma_{i}) \Big | \eta\Bigr\rangle_{\Y_\pm}^{\delta_\pm}=
\Bigl[ \prod_{i=1}^r(-1)^{k_i+1}\ch_{k_i+2}(\gamma_{i})\cap\ev_\pm^*(C_\eta)\cdot
[\fI_{\fY_\pm/\fA_\diamond}^{\delta_\pm,\delta_0}]\virt\Bigl]_0
\]
which form a partition function
\[ Z_{d_\pm,\eta}\bigg(Y_\pm,D_\pm;q\Big|\prod_{i=1}^r\tilde \tau_{k_i}
(\gamma_{i})\bigg)=\sum_{n\in \ZZ}\deg\Bigl\langle\prod_{i=1}^r\tilde\tau_{k_i}(\gamma_{i})
\Big | \eta\Bigr\rangle_{\Y_\pm}^{d_\pm\cdot v+n}q^n.
\]

Using the cycle version of the degeneration formula in Theorem \ref{cycle0}, we verify the following form of degeneration formula

\begin{theo}[{\cite{MNOP2}}]\lab{degen}
Fix a basis $\beta_1,\cdots,\beta_m$ of $H^*(D,\QQ)$. The degeneration formula of
Donaldson-Thomas invariants has the following form
\begin{align*}
\notag Z_{d}\bigg(X_c;q\Big|\prod_{i=1}^r  \tilde \tau_{k_i}(\gamma_{i})
\bigg)
= &\sum_{\substack {d_-,d_+;\,\eta\\d=d_-+d_+}} \frac{(-1)^{|\eta|-l(\eta)}\fz(\eta)}{q^{|\eta|}}\cdot\\
\cdot Z_{d_-,\eta}\bigg(Y_-,D_-;q\Big|&\prod_{i=1}^r\tilde \tau_{k_i}(\gamma_{i})
\bigg)\cdot
 Z_{d_+,\eta^\vee}
\bigg(Y_+,D_+;q\Big|\prod_{i=1}^r\tilde
\tau_{k_i}(\gamma_{i})\bigg)
\end{align*}
where $\eta$ are cohomology weighted partitions w.r.t. $\beta_i$, and $\fz(\eta)=\prod_i \eta_i|\Aut(\eta)|$.
\end{theo}

\vsp

\noindent{\bf Comments}.
In this paper, parallel results on the PT stable pairs invariants are proved. The PT stable pair invariant was introduced by Pandharipande and Thomas in \cite{PT}. Their degeneration was essentially proved in \cite{MPT},
though in a special form. For future reference, we include the statement and the necessary constructions
that lead to a proof of the degeneration of PT stable pair invariants in this paper.

The notion of relative ideal sheaves was developed through email communication
between Pandharipande and the first named author. The technical part of this paper is the proof of the properness and
boundedness of good generations of Grothendieck's Quot-schemes. The parallel results for PT stable pairs are simpler.
The part on perfect obstruction theory necessary for proving the degenerations of invariants are taken from the work \cite{MPT}.

The good degeneration of ideal sheaves for threefolds was constructed by the second named author in his
thesis \cite{Wu-th}. The properness, separatedness and the boundedness were proved there. The proofs  in
this paper for Grothendieck's Quot-schemes are new.

\vsp
\noindent{\bf Acknowledgment}: The first named author is partially supported by an NSF grant and a DARPA grant. The second named author is grateful to Professor Shing-Tung Yau for his support and encouragements.

\section{The stack of expanded degenerations}

We work with a fixed algebraically closed field $\kk$ of characteristic $0$.
We denote $\Gm=GL(1,\kk)$. Let $\pi\colon X\to C$, $0\in C$, be a
simple degeneration;
let ${Y}$ be the normalization of $X_0$; let $\ti D\sub {Y}$ be the preimage of $D\sub X_0$,
and fix $\ti D=D_-\cup D_+$, as defined in Definition \ref{def-simp}.
In this paper, we call $({Y},D_\pm)$ the relative pair associated with $X_0$.

In \cite{Li1, Li2}, the first named author proved the degeneration of Gromov-Witten invariants of a simple degeneration in case
$Y$ is a union of two irreducible components $Y=Y_-\cup Y_+$ and $D$ is connected.
Often, one needs to deal with simple degeneration $X\to C$ when $Y$ is irreducible or contains more than
two connected components, or $D$ is not connected. In this paper, we will construct good degenerations of
moduli spaces for general simple degenerations.

In this section, we review the construction of the stack of expanded degenerations and its family $\fX\to\fC$:
presented in the survey article \cite{Li-Survey}. Some formulation of the stack $\fX$ is new; however, the proofs of the results listed follow directly from the arguments in \cite{Li1}.
\beq\label{sq}
\begin{CD}
\fX @>{\p}>> X\\
@VVV @VV{\pi}V\\
\fC @>>> C
\end{CD}
\eeq

\subsection{The stack $\fC$}

We consider $\Ano$ with the group action
$$(t_1,\cdots,t_{n+1})^{\sigma}=(\sigma_1t_1,\sigma_1^{-1}\sigma_2t_2,\cdots,\sigma_{n-1}^{-1}\sigma_nt_n,\sigma_n^{-1}t_{n+1})
,\quad \sigma\in \gn.
$$
This group action generates a class of equivalence relations on $\Ano$.

We need another class of equivalences. We fix the convention on indices. We denote by $[n+1]=\{1,\ldots,n+1\}$;
for any $I\sub [n+1]$, we let $I^\circ=[n+1]-I$ be its complement. For $|I|=m+1$, we let
$$\text{ind}_I\colon[m+1]\to I\sub [n+1]\and \text{ind}_{I^\circ}\colon[n-m]\to I^\circ \sub [n+1]
$$
be the order preserving isomorphisms; let
\beq\label{UI}
\Ano_{U(I)}= \{ (t)\in\Ano \mid t_i \ne 0,\ i\in I^\circ\}\sub\Ano.
\eeq
We let
\beq\label{tauI}
\ti\tau_I: \Amo\times \Gm^{{}n-m}\mapright{\cong}  \Ano_{U(I)} 
\eeq
be defined by the rule
$$
(t'_1,\cdots,t'_{m+1};\sigma_1,\cdots,\sigma_{n-m})\mapsto (t_1,\cdots,t_{n+1}),\quad \left\{
\begin{array}{ll}
t_k=t'_l, & \text{if}\ k=\text{ind}_I(l);\\
 t_k=\sigma_{l}, & \text{if}\ k=\text{ind}_{I^\circ}(l).
 \end{array}\right.\!\!\!
$$
Restricting to $(\sigma_1,\ldots,\sigma_{n-m})=(1)$, it defines
\beq\label{tau}
\tau_I: \Amo\lra \Ano,\quad (t'_1,\ldots,t_{m+1}')\mapsto \ti\tau_I(t'_1,\ldots,t_{m+1}',1,\ldots,1).
\eeq
We call such $\tau_I$ standard embeddings. Given two $I,I'\sub [n+1]$ of same cardinalities, we define the isomorphism
\beq\label{tauII}
\ti\tau_{I,I'}=\ti\tau_{I}\circ\ti\tau_{I'}\upmo: \Ano_{U(I')}\lra \Ano_{U(I)}.
\eeq

Next, we let $\Ano\to \bA^{\! n+2}$ be the closed immersion $\tau_I$ using $I=[n+1]\sub [n+2]$. Let $\Gm^{{}n} \to \Gm^{{}n+1}$ be the homomorphism defined via $(\sigma_1,\ldots,\sigma_n)\mapsto (\sigma_1,\ldots,\sigma_n,1)$. Via this homomorphism, and viewing $\Ano$ as scheme over $\Ao$ via $(t)\mapsto t_1\ldots t_{n+1}$, the morphism
\beq\label{AA}
\tau_I: \Ano\to \bA^{\! n+2}
\eeq
is a $\Gm^{{}n}$ equivariant $\Ao$-morphism with $\Gm^{{}n}$ acting on $\Ao$ trivially.

Further, for general $I, I'\sub [n+1]$ of $|I|=|I'|$, the equivalence $\ti\tau_{I,I'}$ of $\Ano$ is the restriction of the equivalence $\ti\tau_{I,I'}$ of $\bA^{\!n+2}$, by considering $I,I'$ as subsets in $[n+2]$ via $I, I'\sub [n+1]\sub [n+2]$.

\begin{defi}
We define $\fA_{n}$ be the quotient $[\Ano/\sim]$ by the equivalences generated by the $\Gm^{{}n}$ action and by the equivalences $\ti\tau_{I,I'}$ for all pairs $I,I'\sub [n+1]$ with $|I|=|I'|$. The morphism \eqref{AA} defines
an open immersion $\fA_n\to\fA_{n+1}$. We define $\fA$ be the direct limit $\fA=\varinjlim\fA_n$.
\end{defi}

Note that the tautological $\Ano\to\fA_{n}$ is a surjective smooth chart; the collection $\{\Ano\to\fA\}_{n\ge 0}$
forms a smooth atlas of $\fA$.
\vsp

Now let $0\in C$ be the pointed smooth affine curve given. Without loss of generality, we assume there is
an \'{e}tale morphism $C\to \bA^1$ so that the inverse image of $0\in \bA^1$ is the distinguished point $0\in C$.
We define
$$\fC=C\times_{\Ao}\fA.
$$
It is clear that $\fC$ does not depend on the choice of $C\to\Ao$, and is covered by smooth charts
$$C[n]\defeq C\times_\Ao \Ano\lra \fC=C\times_\Ao\fA.
$$

Let $o_n\in \fA$ be the image of $0\in \Ano$ under the tautological $\Ano\to\fA$. By abuse of notation, we denote by the same $o_n\in \fC$ the lift of $o_n\in\fA$ and $0\in C$. By construction, $o_n$ has automorphism group $\Gm^{{}n}$; and $\overline{o_k}=\{ o_{k'}: k'\ge  k\}$.

\subsection{The stack $\fX$}

We begin with describing $\fX\times_C 0$. We keep the decomposition $\ti D=D_-\cup D_+$ specified at the beginning of this section. Let $N_{\pm}$ be the normal line bundles of $D_\pm$ in $Y$. Since $\pi$ is a simple degeneration, $N_-\otimes N_+\cong \sO_D$. (Here and later we implicitly identify $D_\pm$ with $D$ using $D_-\cup D_+=\ti D\to D$.)

We introduce the ruled variety
$$\Delta=\bP_D(N_+\oplus 1);
$$
it is a $\bP^1$-bundle over $D$ coming with two distinguished sections $D_+=\bP(1)$ and $D_-=\bP(N_+)$. For any $\sigma\in\Gm$, the $\Gm$-action on $N_+\oplus 1$ via $(a,b)^\sigma=(\sigma\cdot a, b)$ defines a $\Gm$-action
\beq\label{Auto}
\sigma: \Delta\lra\Delta,\quad [a,b]^\sigma=[\sigma a,b],
\eeq
called the tautological $\Gm$-action on $\Delta$. This action fixes $D_-$ and $D_+\sub\Delta$.

We now construct $X[n]_0$. We take $n$ copies of $\Delta$, indexed by $\Delta_1,\cdots,\Delta_n$, and form a new scheme
$\xn_0$ according to the following rule: we identify $D_-\sub Y$ with $D_+\sub \Delta_1$, ($D_-\cong D_+$ is
via the isomorphism $D_\pm\to D$;) identify $D_-\sub\Del_i$ with $D_+\sub\Del_{i+1}$, and identify
$D_-\sub \Del_n$ with $D_+\sub Y$.
We denote
\beq\label{xn0} \xn_0=Y\sqcup\Delta_1\sqcup\cdots\sqcup\Delta_n\sqcup (Y),\footnote{Here $\sqcup$ means that we identify $D_-\subset \Delta_i$ with $D_+\subset\Delta_{i+1}$, agreeing that $Y=\Del_0=\Del_{n+1}$; we put the further right $Y$ in parenthesis indicating that it is the same $Y$ appearing in the further left.}
\eeq
We denote $D_i\sub \xn_0$ be $D_-$ in $\Del_{i-1}$, which is also the $D_+\sub \Delta_i$.\footnote{Thus $D_+\sub \Del_i$ is
$\Del_{i-1}\cap \Del_{i}$ and $D_-\sub \Del_i$ is $\Del_{i}\cap \Del_{i+1}$.}
The singular locus of $\xn_0$ is the union of $D_1,\ldots,D_{n+1}$.

\begin{figure}[h!]
\begin{center}
\setlength{\unitlength}{3mm}
\begin{picture}(30,3)(0,-1)
\linethickness{0.075mm}\tiny
\multiput(0,0)(5,0){2}{\line(1,0){5}\circle*{0.2}}
\put(10,0){\line(1,0){2}}
\put(14,-0.3){$\cdots$}
\multiput(20,0)(5,0){2}{\line(1,0){5}}
\multiput(20,0)(5,0){2}{\circle*{0.2}}
\put(20,0){\line(-1,0){2}}
\put(4.5,0.6){$D_1$}
\put(9.5,0.6){$D_2$}
\put(2,0.6){$Y$}
\put(7,0.6){$\Delta_1$}
\multiput(3.8,-1)(5,0){2}{$D_-$}
\multiput(5.6,-1)(5,0){2}{$D_+$}
\put(19.4,0.6){$D_{n}$}
\put(23.7,0.6){$D_{n+1}$}
\put(22,0.6){$\Delta_n$}
\put(27,0.6){$Y$}
\multiput(18.8,-1)(5,0){2}{$D_{-}$}
\multiput(20.6,-1)(5,0){2}{$D_{+}$}
\end{picture}
\end{center}
\caption{\small The two ends are the same $Y$, in the middle a chain of $n$ $\Delta$'s are inserted;
 the $D_-$ of $Y$ is glued to $D_+$ of $\Delta_1$, which is named $D_1$.}
\end{figure}
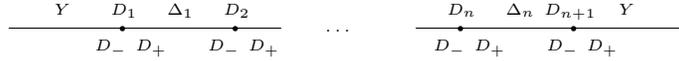

\vskip5pt

Because the inserted $\Delta_i$ intersects the remainder components along $D_i$ and $D_{i+1}\sub \Delta_i$, the tautological $\Gm$-action on $\Delta_i$ (cf. \eqref{Auto}) lifts to an automorphism of $\xn_0$ that acts trivially on all other $\Delta_{j\ne i}$. We let $\Gm^{{}n}$ acts on $\xn_0$ so that its $i$-th factor acts on $\xn_0$ via the tautological $\Gm$-action on $\Delta_i$ and trivially on $\Delta_{j\ne i}$. Let $p\colon\xn_0\lra  X_0$ be the projection contracting all inserted components $\Delta_1,\cdots,\Delta_n$; it is $\Gm^{{}n}$-equivariant with the trivial action on $X$.

\vsp
We now construct the family $\fX\to \fC$ associated with $X\to C$. Let $0\in \cn$ be the preimage of $0\in\Ano$ in $\cn$. We denote $C^*=C-0$ and let $C[m]^*=C[m]\times_C C^*$.

\begin{lemm}\label{Lem-X}
We let $\xn$ be the small resolution $\xn\to X\times_C \cn$, coupled with the projection $p\colon \xn\to X$ induced from $X\times_C\cn\to X$. It is characterized by the properties:
\begin{enumerate}
\item $\xn$ is smooth;
\item the central fiber $(\xn\times_{\cn} 0,p)$ is the $(\xn_0,p)$ constructed;
\item let $\bar\tau_I\colon\cm\to \cn$ be a morphism induced by $\tau_I: \Amo\to\Ano$ (cf. \eqref{tau}); then the induced family
$(\bar\tau_I^*\xn,\bar\tau_I\sta p)$ is isomorphic to $(\xm,p)$ as families over $\cm$, extending the identity map
$$\bar\tau_I^*\xn|_{\cm^*}=\xm|_{\cm^*}=X\times_C \cm\sta;
$$
\item let $\ell_l$ be the $l$-th coordinate line of $\Ano$; 
let $L_l=\cn\times_{\Ano}\ell_l$,
and let $\iota_l\colon L_l\to\cn$ be the inclusion; then the induced family $\iota_l^*\xn$ smooths
the  $l$-th singular divisor $D_l$ of $X[n]_0$.
\end{enumerate}
\end{lemm}

Because of (2), we will view $\xn_0$ as the central fiber $\xn\times_{\cn} 0$.

\begin{lemm}
The $\Gm^{{}n}$ action on $\cn$ with the trivial action on $X$ lifts to a unique $\Gm^{{}n}$-action on $\xn$.
The induced $\Gm^{{}n}$ action on $\xn_0$ is the action described before Lemma \ref{Lem-X}. For $I,I'\sub [n+1]$ of identical cardinalities, the equivalence $\ti\tau_{I,I'}$ in \eqref{tauII} lifts to a $C$-isomorphism
\beq\label{tauX}
\ti\tau_{I,I',X}: \xn\times_{\cn}{\cn_{U(I')}}\cong \xn\times_{\cn}{\cn_{U(I)}},
\eeq
where $\cn_{U(I)}=\cn\times_{\Ano}\Ano_{U(I)}$.
\end{lemm}

As an illustration, let $C=\Ao$, and $X/\Ao$ is a smoothing of $X_0=Y_1\sqcup Y_2$ with a single node $D$.
Then $C[1]=\bA^2$; the central fiber $X[1]_0=Y_1\sqcup \Delta\sqcup Y_2$, $\Delta=\Po$, has two singular divisors $D_1=Y_1\cap \Delta$ and $D_2=\Delta\cap Y_2$. Restricting $X[1]$ to the first coordinate line $\bA^{\!2}_1$, we obtain a family that smoothes $D_1\sub X[1]_0$ but not $D_2$; restricting to the second coordinate line $\bA^{\!2}_2$ the family smoothes $D_2$ but not $D_1$.

\begin{figure}[h!]
\begin{center}
\setlength{\unitlength}{3mm}
\begin{picture}(26,12)(0,-3)
\linethickness{0.075mm}\tiny
\put(0,6){\line(1,0){6}}
\put(0,6){\line(1,2){1}}
\put(0,6){\line(1,-3){1}}
\put(0,1){\line(1,2){1}}
\multiput(1,8)(1,0){5}{\line(1,0){0.5}}
\multiput(0,1)(1,0){6}{\line(1,0){0.5}}
\put(-0.8,2){$Y_1$}
\put(-1,4){$\Delta_1$}
\put(1.5,3){$D_1$}
\put(-1,7){$Y_2$}
\put(0,-2){\line(1,0){6}}\put(0,-2){\circle*{0.2}}
\put(3,-3.5){$\mathbb A^2_1$}
\put(21,3){\line(1,0){5}}
\put(20,6){\line(1,2){1}}
\put(20,6){\line(1,-3){1}}
\put(20,1){\line(1,2){1}}
\multiput(21,8)(1,0){5}{\line(1,0){0.5}}
\multiput(20,1)(1,0){6}{\line(1,0){0.5}}
\put(19.2,2){$Y_1$}
\put(19,4){$\Delta_1$}
\put(20.5,5.5){$D_2$}
\put(19,7){$Y_2$}
\put(20,-2){\line(1,0){6}}\put(20,-2){\circle*{0.2}}
\put(23,-3.5){$\mathbb A^2_2$}
\end{picture}
\end{center}
\caption{\small It shows that $D_1$ is smoothed over $\bA^{\!2}_1$; $D_2$ smoothed over $\bA^{\!2}_2$.}
\end{figure}
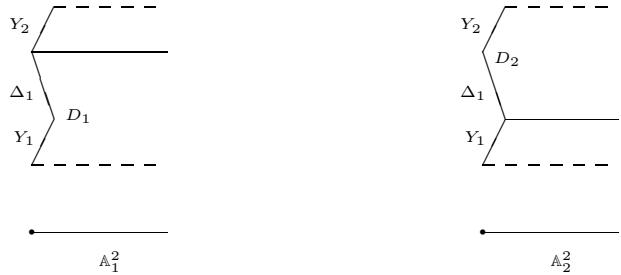

\vsp

\begin{defi}
We define $\fX_n$ be $[\xn/\!\!\sim\,]$, where $\sim$ are equivalence relations generated by the $\Gm^{{}n}$ action and the equivalences $\ti\tau_{I,I',X}$ for all $I,I'\sub [n+1]$ of $|I|=|I'|$. We let $\p_n\colon \fX_n\to X$ be the morphism induced by the tautological projection $p: \xn\to X$.
\end{defi}

The quotient is an Artin stack; it is over $C$ since the $\Gm^{{}n}$ action and the equivalence $\ti\tau_{I,I',X}$
are defined over $C$.
\vsp

Using the inclusion  $[n+1]\sub [n+2]$, the induced $\Ano\to \bA^{\! n+2}$ in \eqref{AA} and the induced $\cn\to C[n+1]$, we have tautological immersion of stacks
\beq\label{XX}
\fX_n\lra\fX_{n+1}
\eeq
that commute with the projections $\p_n$ and $\p_{n+1}$.

\begin{defi}
We define $\fX=\varinjlim \fX_n$; we define $\p\colon \fX\to X$ be the induced projection.
\end{defi}

\begin{theo}
The morphisms $\xn\to\cn$ induce a representable $C$-morphism $\fX\to\fC$. It  fits into the commutative square \eqref{sq}.
\end{theo}

We call $(\fX\to\fC,\p)$ the stack of expanded degenerations of $X\to C$. For any $C$-scheme $S$, we call $\fX\times_{\fC} S\to S$ an $S$-family of expanded degenerations.

\subsection{The stack $\fD_\pm\sub\fY$}

We now construct the stack
\beq\label{DY}
\fD_\pm\sub \fY\lra \fA_\diamond
\eeq
of expanded pairs of $(Y,D_\pm)$.

We fix the convention on indexing $\bA\unn$ and $\Gm^{{}n_-+n_+}$. In this paper, whenever we see product of $n_-+n_+$ copies, we index the individual factor by indices $-n_-,\ldots, -1,1,\ldots,n_+$. (Note that index $0$ is skipped.)
Thus the $(-n_-)$-th coordinate line of $\bA\unn$ is $(t,0,\ldots,0)$, and the $n_+$-th coordinate line is $(0,\ldots,0,t)$. The same convention applies to indexing factors of $\Gm^{{}n_-+n_+}$. We let $\Gm^{{}n_-+n_+}$ acts on $\bA\unn$ via the traditional convention
$$(t_{-n_-},\ldots, t_{-1}, t_1,\ldots, t_{n_+})^\sigma=(\sigma_{-n_-} t_{-n_-},\ldots, \sigma_{-1}t_{-1}, \sigma_1t_1,\ldots, \sigma_{n_+} t_{n_+}).
$$

We then construct
\beq\label{yn}
D[n_-]_-, \, D[n_+]_+\sub \yn\lra \bA\unn,\quad p: \yn\lra Y,
\eeq
inductively by the rule:
\begin{enumerate}
\item $(Y[0,0],D[0]_\pm)=(Y, D_\pm)$;
\item $Y[n_-,n_++1]$ is the blow-up of $Y[n_-,n_+]\times\bA^1$ along $D[n_+]_+\times 0$,
and $D[n_-]_-$ and $D[n_++1]_+$ are the proper transforms of $D[n_-]_-\times \bA^1$ and $D[n_+]_+\times \bA^1$,
respectively;
\item $Y[n_-+1,n_+]$ is the blow-up of $\Ao\times Y[n_-,n_+]$ along $0\times D[n_-]_-$,
and $D[n_-+1]_-$ and $D[n_+]_+$ are the proper transforms of $\Ao\times D[n_-]_-$ and $\Ao\times D[n_+]_+$,
respectively;
\item $p: \yn\to Y$ is the one induced by the identity $Y\to Y$.
\end{enumerate}
Following the convention, the extra copy of $\Ao$ added to the right in item (2) is the $(n_++1)$-th factor of $\bA^{n_-+(n_++1)}$; the copy $\Ao$ added to the left in item (3) is the $(-n_--1)$-th copy in $\bA^{(n_-+1)+n_+}$.

The central fiber of \eqref{yn} is easily described. We let $N_\pm$ be the normal line bundle of $D_\pm$ in $Y$; let $\Delta=\bP_D(N_+\oplus 1)$ with distinguished divisors $D_+=\bP(1)$ and $D_-=\bP(N)$. Then
$$ \yn_0=\yn\times_{\bA^{\!n_-+n_+}}0\and D[n_\pm]_{\pm,0}=D[n_\pm]_\pm\times_{\bA^{\!n_-+n_+}}0
$$
are
\beq\label{Y-comp} \yn_0=\Delta_{-n_-}\sqcup\cdots\sqcup\Delta_{-1}\sqcup Y\sqcup\Delta_1\sqcup\cdots\sqcup\Delta_{n_+}, \quad n_-, \, n_+\ge 0,
\eeq
where the square cup ``$\,\sqcup\,$'' means that we identify the divisor $D_-\sub \Delta_i$ with $D_+\sub \Delta_{i+1}$, understanding that $\Del_0=Y$, and $\Delta_i=\Del$ for $i\ne 0$; $D[n_-]_{-,0}$ is the divisor $D_+$ in $\Del_{-n_-}$, and $D[n_+]_{+,0}$ is the divisor $D_-\sub\Del_{n_+}$.

We let $p:\yn_0\to Y$ be induced by $p: \yn\to Y$ (cf. item (4)); it is by contracting all $\Del_{i\ne 0}$.
The scheme $\yn_0$ has simple normal crossing singularities when $(n_-,n_+)\ne (0,0)$.

We call
\beq\label{ynz}
(\yn_0,D[n_\pm]_{\pm, 0})\quad \text{with}\quad  p: \yn_0\to Y
\eeq
and the $\Gm^{{}n_-+n_+}$-action an expanded relative pair of $(Y,D_\pm)$.

\begin{figure}[h!]
\begin{center}
\setlength{\unitlength}{3.5mm}
\begin{picture}(30,3)(-2,-2)
\linethickness{0.075mm}\tiny
\put(0,0){\line(1,0){3}}
\multiput(0,0)(3,0){2}{\circle*{0.2}}
\put(3,0){\line(1,0){1}}
\put(6,-0.3){$\cdots$}
\put(9,0){\line(-1,0){1}}
\multiput(9,0)(3,0){4}{\circle*{0.2}}
\multiput(9,0)(3,0){3}{\line(1,0){3}}
\put(18,0){\line(1,0){1}}
\put(21,-0.3){$\cdots$}
\put(24,0){\line(-1,0){1}}
\multiput(24,0)(3,0){2}{\circle*{0.2}}
\put(24,0){\line(1,0){3}}
\put(1,0.6){$\Delta_{-n_{\!-}}$}
\put(10,0.6){$\Delta_{-\!1}$}
\put(13.5,0.6){$Y$}
\put(16.5,0.6){$\Delta_1$}
\put(25,0.6){$\Delta_{n_{\!+}}$}
\put(2,-1){$D_-$}\put(3.3,-1){$D_+$}
\multiput(11,-1)(3,0){2}{$D_-$}
\multiput(12.3,-1)(3,0){2}{$D_+$}
\put(23,-1){$D_-$}\put(24.3,-1){$D_+$}
\put(-0.3,-1){$D_+$}
\put(0,-1.8){$||$}
\put(-2,-2.5){$D[n_-]_{-,0}$}
\put(26.7,-1){$D_-$}
\put(27,-1.8){$||$}
\put(25,-2.5){$D[n_+]_{+,0}$}
\end{picture}
\end{center}
\caption{\small The $Y$, $\Delta$'s glue to form $Y[\nn]_0$; the two end divisors are the
new relative divisors of $Y[\nn]_0$.}
\end{figure}
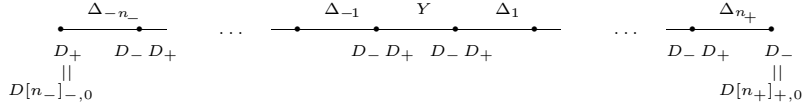
\vsp

The families $\yn\to \bA\unn$ has the following additional properties:
\begin{enumerate}
\item[(5)] let $\ell_l\to \bA\unn$ be the $l$-th coordinate line of $\bA\unn$, $-n_-\le l\le n_+$, $l\ne 0$, then the restriction $\yn\times_{\bA\unn} \ell_l$ smoothes the divisor $D_l=\Delta_{l-1}\cap\Delta_{l}$ if $l>0$, of $D_l=\Del_l\cap \Del_{l+1}$ if $l<0$.
\end{enumerate}
(Notice that $Y[\nn]_0$ has singular divisors $D_{l}$, $-n_-\le l\le n_+$ and $l\ne 0$.)

The family \eqref{yn} and the pair \eqref{ynz} are $\Gm^{{}n_-+n_+}$-equivariant. The $k$-th factor of the $\Gm$ in $\Gm^{{}n_-+n_+}$ acts trivially on all $\Del_i$ except $\Del_{k}$; on $\Del_k$ the action is the tautological $\Gm$-action of \eqref{Auto}.

\vsp

Like the stack $\fX\to\fC$, the stack \eqref{DY} we aim to construct will be the limit of the quotients of \eqref{yn} by $\Gm^{{}n_-+n_+}$ and another class of equivalences associated to subsets
\beq\label{indexI}
I\sub [-n_-, n_+]-\{0\}.
\eeq
(We define its complement $I^\circ=[-n_-,n_+]- I\cup\{0\}$.)

Given an $I$ as in \eqref{indexI}, we define $\bA\unn_{U(I)}\sub \bA\unn$ be as in \eqref{UI}. Like
\eqref{tauI}, letting $m_\pm=|I\cap \ZZ_{\pm}|$, we have an isomorphism
\beq\label{I1}
\ti\tau_{I}: \bA^{m_-+m_+}\times \Gm^{{}(n_--m_-)+(n_+-m_+)}\lra \bA\unn_{U(I)},
\eeq
and for any $I'$ as in \eqref{indexI} with
\beq\label{m}
m_\pm=|I\cap \ZZ_{\pm}|=|I'\cap \ZZ_\pm|,
\eeq
the pair $(I,I')$ defines an isomorphism
\beq\label{tauII-}
\ti\tau_{I,I'}=\ti\tau_I\circ\ti \tau_{I'}\upmo: \bA\unn_{U(I')}\lra \bA\unn_{U(I)}.
\eeq
As before, we let
\beq\label{I2}
\tau_I: \bA^{\!m_-+m_+}\lra \bA\unn
\eeq
be $\ti\tau_I$ restricting to $\bA^{m_-+m_+}\times\{1\}$, where $1\in\Gm^{(n_--m_-)+(n_+-m_+)}$ is the identity element.

Following the construction, one checks that for any $I$ as in \eqref{indexI}, we have a canonical isomorphism
$$\tau_{I,Y}: Y[m_-,m_+]\lra \tau_I\sta Y[\nn],
$$
lifting the $\tau_I$ in \eqref{I2}; for any pair $(I,I')$ of subsets in \eqref{indexI} satisfying \eqref{m},
we have a canonical isomorphism
$$\ti\tau_{I,I',Y}: Y[\nn]\times_{\bA\unn}{\bA\unn_{U(I')}}\lra  Y[\nn]\times_{\bA\unn}{\bA\unn_{U(I)}},
$$
lifting the $\ti\tau_{I,I'}$ in \eqref{tauII-}.

\begin{defi}\label{DY2}
We define $\fA_{\diamond, n_-+n_+}$ be the quotient $[\bA\unn/\!\!\sim\,]$, quotient by the equivalence relations generated by the $\Gm^{{}n_-+n_+}$-action and the equivalences $\ti\tau_{I,I'}$ for all allowable pairs $(I,I')$ in \eqref{indexI}; using \eqref{I2}, for $m_\pm\le n_\pm$, we have open immersion $\fA_{\diamond,m_-+m_+}\to\fA_{\diamond,n_-+n_+}$; we define $\fA_\diamond=\lim_{n_-,n_+}\fA_{\diamond, n_-+n_+}$.
$\fA_\diamond$ is an Artin stack.

We define $\fD_{n_\pm,\pm}\sub \fY_{n_-+n_+}$ be the quotient of $D[n_\pm]_\pm\sub Y[\nn]$ by $\Gm^{{}n_-+n_+}$ and the equivalences $\ti\tau_{I,I',Y}$ for all pairs $(I,I')$ satisfying \eqref{m}; we define $\fD_\pm\sub \fY$ be the limit of $\fD_{n_\pm,\pm}\sub \fY_{n_-+n_+}$ as $n_-,n_+\to +\infty$. We let $\p: \fY\to Y$ be the projection induced by the tautological $Y[\nn]\to Y$.
\end{defi}

\begin{theo}
The projections $\yn\to\bA\unn$ induce a representable morphism $\fD_\pm\sub \fY\to \fA_\diamond$.
\end{theo}

We call $\fD_\pm\sub \fY\to \fA_\diamond$ with $\p: \fY\to Y$ the stack of expanded relative pairs of $(Y,D_\pm)$.
Using $(\fD_\pm\sub\fY\to\fA_\diamond,\p)$, we define the collection $\fY(S)$ of expanded families of
pair $(Y,D_\pm)$ over a scheme $S$  be
$$\fD_\pm\times_{\fA_\diamond} S\sub \fY\times_{\fA_\diamond} S,\quad S\to\fA_\diamond.
$$

In case $Y=Y_-\cup Y_+$ is a union of two connected components, we use $D_\pm=\ti D\cap Y_\pm$.
We define the pair of stack
\beq\label{D-one}
\fD_+\sub \fY_+\defeq \fY\times_Y Y_+.
\eeq
Or $\fY_+$ can be defined as in Definition \ref{DY2} with $Y$ replaced by $Y_+$, $n_-=0$ and $D_-=\emptyset$.
The pair $\D_-\sub\fY_-$ is defined similarly.

\subsection{Decomposition of degenerations I}

To state the decomposition of good degenerations, we introduce the stack of node-marking objects in $\fX_0\defeq \fX\times_C 0$. This construction was first introduced in \cite{Kiem-Li-deg}.

\begin{defi}\label{node-mark}
A node-marking of $\xn_0$ is a marking of one of the singular divisor $D_k$ of $\xn_0$. A node-marking of a family $\cX\to S$ in $\fX_0(S)$ is an $S$-morphism $\eta: D\times S\to \cX$ so that for any closed $s\in S$, $\eta(D\times s)\sub\cX_s$ is a node-marking of $\cX_s$.

An arrow between two $\cX$ and $\cX'$ in $\fX_0(S)$ with node-markings $\eta$ and $\eta'$ is an arrow
$\rho:\cX\to \cX'$ in $\fX_0(S)$ so that for any closed $s\in S$, $\rho\circ\eta(D\times s)=\eta'(D\times s)$.
\end{defi}

\begin{prop}
The collection of families in $\fX_0$ with node-markings form an Artin stack, denoted by $\fX_0^\dagger$.
Forgetting the node-marking defines a morphism
$$ \fX_0^\dagger\lra \fX_0.
$$
\end{prop}

\begin{proof}
The smooth chart $\xn\to \fX$ induces a smooth chart $\xn\times_C0 \to \fX_0$. We denote $\Ano_{t_k=0}=\{(t)\in\Ano\mid t_k=0\}$. Then $\Ano\times_\Ao 0=\bigcup_{k=1}^{n+1} \Ano_{t_k=0}$. Further,
$$\xn_{t_k=0}\defeq \xn\times_{\Ano}\Ano_{t_k=0}
$$
has normal crossing singularity and its singular divisor is the image of the $X\times\Ano_{t_k=0}$-morphism
\beq\label{etak}
\eta_k: D\times \Ano_{t_k=0}\lra \xn_{t_k=0}.
\eeq
According to Definition \ref{node-mark}, one checks that \eqref{etak} is a node-marking of $\xn_{t_k=0}$; thus
\beq\label{node-mark-ex}
(\xn_{t_k=0},\eta_k)\in \fX_0^\dagger( \Ano_{t_k=0}).
\eeq
The disjoint union of \eqref{node-mark-ex} for all $1\le k\le n+1$ form a smooth atlas of $\fX_0^\dagger$.
This proves that $\fX_0^\dagger$ is an Artin stack.
\end{proof}

It will be useful to construct a stack $\fC_0^\dagger$ and an arrow $\fC_0^\dagger\to \fC$ that fits into a Cartesian product
\beq\label{XCXC}
\begin{CD}
\fX_0^\dagger @>>> \fX\\
@VVV @VVV\\
\fC_0^\dagger@>>> \fC.
\end{CD}
\eeq

We construct $\fC_0^\dagger$ as follows. For a pair of integers $1\le k\le n+1$, we let $\Gm^{{}n}$ acts on $\Ano_{t_k=0}$ via the $\Gm^{{}n}$ action on $\Ano$ and the inclusion $ \Ano_{t_k=0}\sub\Ano$. Such action generates equivalence relation on $\Ano_{t_k=0}$.

For any $I\sub [n+1]$ and $k$ an integer, we denote $I_{<k}=\{i\in I\mid i<k\}$; similarly for $I_{>k}$.
Let $k\in I\sub [n+1]$ and $k'\in I'\sub [n+1]$ such that
\beq\label{II}
|I_{<k}|=|I'_{<k}| \and |I_{>k}|=|I_{>k}'|.
\eeq
The equivalence $\ti\tau_{I,I'}$ of \eqref{tauII} restricted to $\Ano_{t_k=0}\cap \Ano_{U(I')}$ defines
\beq\label{tauIII}
\tau_{(I,k),(I',k')}: \Ano_{{k^{\prime c}}}\cap \Ano_{U(I')}\mapright{\cong} \Ano_{{k^c}}\cap\Ano_{U(I)}.
\eeq
These isomorphisms generate equivalence relations too.

We define the closed immersion
\beq\label{+1}
\tau_{+1}:  \Ano_{t_k=0}\lra \bA^{\!n+2}_{{k^c}},\quad (z)\mapsto (z,1).
\eeq

\begin{defi}
We define $\fC_{n,0}^\dagger$ be the quotient $[\coprod_{k=1}^{n+1} \Ano_{{t_k=0}}/\sim]$, where $\sim$ is the equivalence generated by the $\Gm^n$ action on $\Ano_{t_k=0}$ and by $\tau_{(I,k),(I',k')}$ for all pairs $k\in I$ and $k'\in I'$ satisfying \eqref{II}; we define open immersions $\fC_{n,0}^\dagger\to \fC_{n+1,0}^\dagger$ using \eqref{+1}; we define $\fC_0^\dagger=\varinjlim \fC_{n,0}^\dagger$.
\end{defi}

\begin{prop}
The morphisms $\xn_{t_k=0}\to \Ano_{t_k=0}$, where $\xn_{t_k=0}$ is with the node-marking \eqref{node-mark-ex}, induce a morphism $\fX_0^\dagger\to\fC_0^\dagger$ that fits into the Cartesian product \eqref{XCXC}.
\end{prop}

As $\coprod \Ano_{t_k=0}\to\fC_{n,0}^\dagger$ is a smooth chart of $\fC_{n,0}^\dagger$, and the former is the normalization of $\Ano\times_\Ao 0$, the morphism $\fC_0^\dagger \to\fC_0$ is a normalization. It is fitting to call $\fX_0^\dagger\to\fX_0$ the decomposition of locally complete intersection singularity of $\fX_0$.

The final step of the decomposition is the following isomorphism result.

\begin{prop}
There is a canonical isomorphism $\fC_0\shar\cong \fA_\diamond$ so that $\fX_0^\dagger$ is derived from $\fY$ by
identifying the stacks $\fD_-$ with $\fD_+$ via the isomorphisms $\fD_-\cong D\times \fA_\diamond\cong \fD_+$,
and declaring the identifying loci the node-marking.
\end{prop}

\begin{proof}
We define $\bA\unn\to \Ano_{t_k=0}$, $k=n_-+1,\ n=n_-+n_+$, via
$$(t_{-n_-},\ldots, t_{-1},t_1,\ldots,t_{n_+})\mapsto (t_{-1},\ldots, t_{-n_-}, 0,t_{n_+},\ldots, t_1).
$$
This is $\Gm^{{}n}$ equivariant via a homomorphism $\Gm^{{}n}\to \Gm^{{}n+1}$, and induces a morphism $\fA_\diamond\to\fC_0^\dagger$. The remainder of the proof is straightforward.
\end{proof}

\subsection{Decomposition of degenerations II}\label{subsec2.5}

This decomposition works for the case $Y=Y_-\cup Y_+$ is the union of two irreducible components; we let $D_\pm=\ti D\cap  Y_\pm$ and define $\fD_\pm\sub \fY_\pm$ as in \eqref{D-one}.

We fix an additive group $\Lambda$. Using $Y=Y_-\cup Y_+$,
we index the irreducible components of $\xn_0$ as $\Del_0=Y_-$, $\Del_{n+1}=Y_+$, and other $\Del_i$ are as usual.

\begin{defi}
A weight assignment of $\xn_0$ is a function
$$w: \{\Del_0,\ldots,\Del_{n+1}, D_1,\ldots, D_{n+1}\}\lra \Lam
$$
that assigns weights in $\Lam$ to $\Del_i$ and $D_j$ in $\xn_0$.
A weight assignment of $X_t$, $t\ne 0$, is a single value assignment $w(X_t)\in\Lambda$.
A weight assignment $w$ of $\cX\in \fX(S)$ is a collection $\{w_s\mid s\in S\}$
of weight assignments $w_s$ of $\cX_s$.
\end{defi}

We make sense of continuous weight assignments of families.
For any subchain $\Del_{[l,l']}\defeq \Delta_l\cup\ldots\cup \Delta_{l'}$ we define its weight to be
(recall $D_i=\Delta_{i-1}\cap \Delta_i$)
$$w(\Del_{[l,l']})=\sum_{l\le i\le l'} w(\Del_i)-\sum_{l<i\le l'} w(D_i).
$$
Let $s_0\in S$ be an irreducible curve, and let $w$ be a weight assignment of $\cX \in\fX(S)$. Suppose $\cX_{s_0}\cong \xn_0$ and $\cX_s\cong \xm_0$ for a general $s\in S$, Then $m\le n$, and there are
\beq\label{ki}
k_0=0< k_1< \ldots< k_{m+1}< k_{m+2}= n+2
\eeq
so that the $\Delta_i\sub \cX_s$ specializes to the chain $\Delta_{[k_i,k_{i+1}-1]}\sub \cX_{s_0}$, (i.e. the singular divisors $D_{k_i}\sub \cX_{s_0}$ are not smoothed in the family $\cX$.) The total weight of $w$ is $w(\xn_0)$.

\begin{defi}\label{rule2}
Let $s_0\in S$ be an irreducible curve, and $\cX\in\fX(S)$ be as stated. We say a weight assignment $w$ of $\cX$ is continuous at $s_0$ if the followings hold:\\
(1) In case for a general $s\in S$ we have $\cX_s=\xm_0$, letting $k_i$ be as in \eqref{ki}, then
$w_s(\Del_i)=w_{s_0}(\Delta_{[k_i,k_{i+1}-1}])$ and $w_s(D_i)=w_{s_0}(D_{k_i})$.\\
(2) In case for a general $s\in S$ we have $\cX_s=X_t$ for a $t\ne 0\in C$, then
$w_s(\cX_s)=w_{s_0}(\cX_{s_0})$.

In general, a weight assignment of $\cX\in \fX(S')$ is continuous if for any irreducible curve $s_0\in S$ and $S\to S'$, the pull back family $\cX\times_{S'} S$ with the induced weight assignment is continuous at $s_0$.
\end{defi}

\begin{exam}
Suppose $\dim X/C=1$. In case there is a locally free sheaf $\sE$ on $\cX$, assigning each $\Delta_k\sub\cX_s$
the degree of $\cE|_{\Delta_k}$ and assigning each $D_l\sub\cX_s$ zero is a continuous weight
assignment taking values in $\ZZ$.
\end{exam}

We define the stack of weighted expanded degenerations $\fX^\beta$.

\begin{defiprop}
Given a $\beta\in\Lam$, we define the groupoid $\fX^\beta(S)$ be the collections of pairs $(\cX,w)$, where $\cX \in \fX(S)$ and
$w$ is a continuous weight assignment of $\cX$ of total weights $\beta$. An arrow between $(\cX,w)$ and $(\cX',w')\in \fX^\beta(S)$ consists of an arrow $\rho: \cX\to \cX'$ in $\fX(S)$ that preserves the weights $w$ and $w'$.
The groupoid $\fX^\beta$ is an Artin stack.
\end{defiprop}

By forgetting the weights, we obtain the forgetful morphism $\fX^\beta\to \fX$. We claim that there is a weighted stack $\fC^\beta$ together with a forgetful morphism $\fC^\beta\to\fC$ so that $\fX^\beta$ is the Cartesian product
\beq\label{xcbeta}
\begin{CD}
\fX^\beta @>>> \fX\\
@VVV @VVV\\
\fC^\beta @>>> \fC.
\end{CD}
\eeq

The easiest way to do this is to define a weight assignment of a $t\in\cn$ be a weight of $\xn_t$. Or a weight of $S\to \fC$ is a weight of $\fX\times_{\fC} S$. We then define $\fC^\beta$ to be the groupoid consisting of
$(S\to \fC,w)$, where $w$ is a weight assignment of $S\to \fC$, etc.

\begin{prop}
The groupoid $\fC^\beta$ is an Artin stack, together with a tautological morphism $\fX^\beta\to\fC^\beta$;  the forgetful morphism $\fC^\beta\to\fC$ is \'etale and fits into the Cartesian square \eqref{xcbeta}.
\end{prop}

Replacing $\fX/\fC$ by $\fX_0\shar/\fC_0\shar$, we obtain a pair
$$\fX_0^{\dagger,\beta}\lra \fC_0^{\dagger,\beta},
$$
where closed points in $\fX_0^{\dagger,\beta}$ are $(\xn_0,D_k,w)$ of which $D_k\sub\xn_0$ are node-markings and $w$ are weight assignments of $\xn_0$ of total weights $\beta$. We define $\fC_0^{\dagger,\beta}$ parallelly, combining
the construction of $\fC_0\shar$ and $\fC^\beta$.

The pair $\fX_0^{\dagger,\beta}\lra \fC_0^{\dagger,\beta}$ is a disjoint union of open and closed substacks indexed by the set of splittings of $\beta$. We let
$$\Lambda_\beta^{\mathrm{spl}}=\{{\delta}=({\delta}_\pm,{\delta}_0)\mid {\delta}_-,\delta_+,{\delta}_0\in\Lambda,\
{\delta}_-+{\delta}_+-{\delta}_0=\beta\}.
$$
For each ${\delta}\in \Lambda_\beta^{\mathrm{spl}}$, we define $\fX_0^{\dagger,{\delta}}(\kk)$ be the collection of those $(\xn_0, D_k,w)\in \fX_0^{\dagger,\beta}(\kk)$ such that
$$w(\Delta_{[0,k-1]})={\delta}_-,\quad w(\Del_{[k,n+1]})={\delta}_+ \and w(D_k)={\delta}_0.
$$
It is both open and closed in $\fX_0^{\dagger,\beta}(\kk)$; thus defines an open and closed substack
$\fX_0^{\dagger,{\delta}}\lra  \fX_0^{\dagger,\beta}$.

Accordingly, we can form the stack $\fC_0^{\dagger,{\delta}}$ and a morphism $\fC_0^{\dagger,{\delta}}\to \fC_0^{\dagger,\beta}$ that fits into a Cartesian product
$$\begin{CD}
\fX_0^{\dagger,{\delta}}@>>>  \fX_0^{\dagger,\beta}\\
@VVV@VVV\\
\fC_0^{\dagger,{\delta}}@>{\Phi_\delta^\dagger}>> \fC_0^{\dagger,\beta}
\end{CD}
$$

We let
\beq\label{Phi-e}
\Phi_{\delta}: \fC_0^{\dagger,{\delta}}\lra  \fC^\beta
\eeq
be $\Phi_\delta^\dagger$ composed with the forgetful morphism $\fC_0^{\dagger,\beta}\to\fC^\beta$.
The following Proposition says that they are Cartier divisors.

\begin{prop}\label{prop2.24}
There are canonical line bundles with sections $(L_{\delta},s_{\delta})$ on $\fC^\beta$, indexed by ${\delta}\in
\Lambda_\beta^{\mathrm{spl}}$, such that
\begin{enumerate}
\item let $t\in\Gamma(\sO_\Ao)$ be the standard coordinate function and $\pi: \fC^\beta\to \Ao$ be the tautological projection, then
    $$\bigotimes_{{\delta}\in \Lambda_\beta^{\mathrm{spl}}} L_{\delta}\cong \sO_{\fC^\beta}\and     \prod_{{\delta}\in \splitset} s_{\delta}=\pi\sta t;$$
\item the morphism $\Phi_{\delta}$ factors through $s_{\delta}\upmo(0)\sub\fC^\beta$ and effects an isomorphism $\fC_0^{\dagger,{\delta}}\cong s_{\delta}\upmo(0)$.
\end{enumerate}
\end{prop}

The proof of this decomposition is essentially given in \cite{Li2}. Note that this Proposition states that $\fC_0^\beta\sub\fC^\beta$ is a complete intersection substack, and the disjoint union of $\fC_0^{\dagger,{\delta}}$ is its normalization.

\vsp

We complete the weighted decomposition by introducing the stack of weighted relative pairs. We define a weight assignment of $(Y_+[n],D_+[n])$ be a function $w$ that assigns values in $\Lam$ to the irreducible components of $Y_+[n]$, of its $D_k$'s, and of $D_+[n]$. We define the continuous weight assignments of $(\cY_+,\cD_+)\in \fY_+(S)$ parallel to Definition \ref{rule2}.

For a ${\delta}\in \splitset$, we define the stack $\fY_+^{{\delta}_+,{\delta}_0}$ so that $\fY_+^{{\delta}_+,{\delta}_0}(S)$ consists of data $(\cY_+,\cD_+, w)$, where $(\cY_+,\cD_+)\in \fY_+(S)$ and $w$ are weight assignments of $(\cY_+,\cD_+)$, so that for any closed $s\in S$, $w_{s}(\cD_{+,s})={\delta}_0$ and the total weights $w_{s}(\cY_{+,s})={\delta}_+$. The case for $(Y_-,D_-)$ and similar objects are defined with ``$+$'' replaced by ``$-$''.

We let $\fA_\diamond^{{\delta}_\pm,{\delta}_0}$ be the stack defined similarly so that we have Cartesian product
$$\begin{CD}\fY_\pm^{{\delta}_\pm,{\delta}_0} @>>> \fY_\pm\\
@VVV@VVV\\
\fA^{{\delta}_\pm,{\delta}_0}_\diamond @>>> \fA_\diamond
\end{CD}
$$

By gluing the two relative divisors $\cD_-$ and $\cD_+$ of $(\cY_\pm,\cD_\pm, w_\pm)\in \fY_\pm^{{\delta}_\pm,{\delta}_0}(S)$ and combining the weights $w_-$ and $w_+$, we obtain the following commutative square of morphisms
$$\begin{CD}
\fY_-^{{\delta}_-,{\delta}_0} \times \fY_+^{{\delta}_+,{\delta}_0} @>>> \fX^{\dagger,{\delta}}_0 \\
@VVV@VVV\\
\fA^{{\delta}_-,{\delta}_0}_\diamond \times \fA^{{\delta}_+,{\delta}_0}_\diamond@>{\Psi_{\delta}}>> \fC_0^{\dagger,{\delta}}
\end{CD}
$$

\begin{prop}\label{isom-stacks}
The morphism $\Psi_{\delta}$ is an isomorphism.
\end{prop}

\section{Admissible coherent sheaves}

We develop necessary technical results on admissible coherent sheaves on singular schemes. In this paper, we adopt the convention that for any closed or open $V\sub W$and $\sF$ a sheaf of $\sO_W$-modules, we denote $\sF|_V=\sF\otimes_{\sO_W}\sO_V$.

\subsection{Coherent sheaves normal to a closed subscheme}

Let $W$ be a noetherian scheme and $D\sub W$ be a closed subscheme.

\begin{defi}\lab{normal}
Let $\sF$ be a coherent sheaf on $W$. We say $\sF$ is normal to $D$ if $\Tor^{\sO_W}_1(\sF,\sO_D)=0$.
\end{defi}

In this paper, we are interested in two situations. One is when $D\sub W$ is a Cartier divisor; the other is when $W=W_1\cup W_2$ is a  union of subschemes $W_1$ and $W_2\sub W$ that intersect transversally along a Cartier divisor $D=W_1\cap W_2$.

\vsp

To study flat families of coherent sheaves, we quote the following known fact.

\begin{lemm} \lab{tech-1}
Let $(A,\fm)$ be a noetherian local ring with residue field $\bk$,
and $B$ a finitely generated $A$-algebra, flat over $A$. Let $M$ be a finitely generated $B$-module. Then $\Tor^B_1(M,B/\fm B)=0$ if and only if $M$ is flat over $A$.
\end{lemm}

\begin{proof}
Since $M$ is a finitely generated $B$-module, it fits into an exact sequence
$$0\lra M'\lra B^{\oplus n}\lra M\lra 0.
$$
Tensoring with $B/\fm B$, we know $\Tor^B_1(M,B/\fm B)=0$ if and only if $M'/\fm M'=M'\otimes_A\bk\to (B/\fm B)^{\oplus n}$ is injective. On the other hand, applying $\otimes_A \bk$ to the above exact sequence, we obtain
$$\Tor_1^A(B^{\oplus n},\bk) \lra \Tor_1^A(M,\bk) \lra M'\otimes_A \bk\lra
B^{\oplus n}\otimes_A \bk=(B/\fm B)^{\oplus n}.
$$
Since $B$ is $A$-flat, $\Tor_1^A(B^{\oplus n},\bk)=0$. Thus the last arrow is injective if and only if $\Tor_1^A(M,\bk)=0$.
By local criterion of flatness \cite[Theorem 49]{Mat}, this is equivalent to $M$ being $A$-flat. This proves the Lemma.
\end{proof}

\vsp

For the case where $D\sub W$ is a Cartier divisor in a smooth $W$, a coherent sheaf $\sF$ on $W$ normal to $D$ is equivalent to that $\sF$ is flat along the ``normal direction'' of $D\sub W$. To make this precise, we assume $W$ is affine and pick a regular $z\in\Gamma(\sO_W)$ so that $D=(z=0)$. We define $\tau: W\to \Ao=\Spec \bk[u]$ via $\tau\sta(u)=z$. For any scheme $S$, we denote by $\pi_S\colon W\times S\to S$ the projection and view $W\times S$ as a family over $\Ao\times S$ via
\beq\lab{split}
(\tau,\pi_S): W\times S\lra \Ao\times S.
\eeq

\begin{prop}\lab{flat}
Let $D\sub W$, $S$ and \eqref{split} be as stated. Suppose $\sF$ an $S$-flat family of coherent sheaves on $W\times S$, and $s\in S$ is a closed point so that $\sF_s=\sF\otimes_{\sO_S} \bk(s)$ is normal to $D$. Then there is an open subset $(0,s)\in U\sub \Ao\times S$ so that the sheaf $\sF|_U$ is flat over $U$.

Conversely, let $U\sub \Ao\times S$ be an open subset such that $\sF$ is flat over $U$, then for $(0,s)\in U$, $\sF_s$ is normal to $D$.
\end{prop}

\begin{proof}
We let
$$U=\{ x\in \Ao\times S \mid \sF\otimes_{\sO_{\Ao\times S}}\sO_{\Ao\times S,x} \ \text{is $ \sO_{\Ao\times S,x}$-flat}\}.
$$
By \cite[Theorem 53]{Mat}, $U$ is an open subset of $\Ao\times S$ (possibly empty) and $\sF|_U$ is flat over $U$.

To prove the Proposition, we only need to show that $(0,s)\in U$. But this is a direct application of Lemma \ref{tech-1}. We let
$$A=\sO_{\Ao\times S,(0,s)},\quad B=\Gamma(\sO_{W\times S}\otimes_{\sO_{\Ao\times S}}A),\quad
M=\Gamma(\sF\otimes_{\sO_{\Ao\times S}}A).
$$
Since the assumption that $\sF_s$ is normal to $D$ implies that $\Tor^B_1(M,B/\fm B)=0$,
Lemma \ref{tech-1} implies that $M$ is flat over $A$, that is, $(0,s)\in U$.

For the converse, given $(0,s)\in U$, by the base change property of flatness, $\sF_s=\sF|_{W\times s}$ is flat over $U_s=U\cap (\Ao\times s)$.
Since $(0,s)\in U$, we have $0\in  U_s$. By Lemma \ref{tech-1}, $\Tor_1^{\sO_W}(\sF_s, \sO_D)=0$; by Definition \ref{normal}, $\sF_s$ is normal to $D$.
\end{proof}

\begin{coro}\lab{flat-2}
Let the situation be as in Proposition \ref{flat} and let $\sF$ be an $S$-flat family of coherent sheaves on $W\times S$. Then the set
$V=\{s\in S\mid \sF_s \ \text{is normal to } D\}$ is open in $S$,
and $\sF|_{D\times V}$ is a $V$-flat family of coherent sheaves on $D\times V$.
\end{coro}

\begin{proof}
Let $U$ be the open subset introduced in the proof of Proposition \ref{flat}. Then $U\cap (0\times S)\sub S$ is exactly the locus where $\sF_s$ is normal to $D$.

By Proposition \ref{flat}, we know that there exists an open subset $U\sub \Ao\times S$, so that $0\times V\sub U$ and
$\sF|_U$ is flat over $U$. Thus, by the base change property of flatness, $\sF|_{D\times V}$ is $V$-flat. This proves the second part of the Corollary.
\end{proof}

Now we move to the second case where $W=W_1\cup W_2$ is a union of two smooth schemes $W_1$ and $W_2$ intersecting transversally along a Cartier divisor $D=W_1\cap W_2$ (in $W_1$ and $W_2$). Assume $W$ is affine; we find $z_i\in\Gamma(\sO_W)$ so that $W_1=(z_2=0)$ and $W_2=(z_1=0)$, thus $D=(z_1=z_2=0)$. We let
$$T=\Spec \bk[u_1,u_2]/(u_1u_2),
$$
and let $\xi\colon W\to T$ be defined by $\xi\sta(u_i)=z_i$. As before, since the fiber of $W\to T$ over $0\in T$ is $D$, which is smooth, by shrinking $W$ if necessary, we can assume that $\xi$ is smooth.

Now let $S$ be any scheme, $\pi_S\colon W\times S\to S$ be the projection. We will view $W\times S$ as a
family over $T\times S$ via
\beq\label{split-2}
(\xi,\pi_S): W\times S\lra T\times S.
\eeq
By our choice, it is smooth.

\begin{prop}\lab{flat-4}
Proposition \ref{flat} and Corollary \ref{flat-2} hold with the family \eqref{split} replaced by the family \eqref{split-2}.
\end{prop}

\begin{proof} The proof is exactly the same.
\end{proof}

\begin{prop}\lab{flat-5}
Let the situation be as in \eqref{split-2}. Let $\sF$ be an $S$-flat family of coherent sheaves on $W\times S$. Suppose for any $s\in S$ the sheaf $\sF_s$ is normal to $D$. Then $\sF_i=\sF|_{W_i\times S}$ is an $S$-flat family of coherent sheaves each of its members normal to $D$.
\end{prop}

\begin{proof}
We prove the case $i=1$. Since this is a local problem, we assume $W$ is affine. We pick the morphism in \eqref{split-2}. Applying Proposition \ref{flat-4}, we can find an open $D\times S\sub U\sub W\times S$
so that $\sF|_U$ is flat over $T\times S$. By the base change property of flatness, $\sF|_{U\cap W_1\times S}$ is flat over $T_1\times S$, where $T_1=(u_2=0)$. Since $D\times S\sub U$, $\sF_1=\sF|_{W_1\times S}$ is flat over $S$ near $D\sub W_1\times S$. Since $W_1-D$ is open in $W$ and $\sF|_{W_1-D}=\sF_1|_{W_1-D}$, $\sF_1$ is flat over $S$.

Finally, because $\sF_s$ is normal to $D$, $\sF_1|_{W_1\times s}$ is normal to $D$ as well.
This proves the Proposition for $i=1$. The case $i=2$ is the same.
\end{proof}

We also have the converse.

\begin{lemm}\label{2-flat}
Let $\sF$ be a sheaf on $W$ in the situation \eqref{split-2}. Then $\sF$ is normal to $D\sub W$ if and
only if both $\sF|_{W_i}$, $i=1,2$, are normal to $D\sub W_i$.
\end{lemm}

\begin{proof}
Let $T_1=(u_2=0)$ and $T_2=(u_1=0)\sub T$. It is proved in Proposition \ref{flat-5} that $\sF$ normal to $D$ implies that both $\sF|_{W_i}$ are normal to $D$. Suppose both $\sF|_{W_i}$ are normal to $D$. Then both $\sF|_{W_i}$ are flat over $T_i$ near $0\in T_i$. We prove that $\sF$ is flat over $T$ near $0\in T$. Since $\hat\sO_{T,0}=\kk[[u_1,u_2]]/(u_1u_2)$, each ideal $I\sub\hat\sO_{T,0}$ is either principal or has the form $I=(u_1^{a_1}, u_2^{a_2})$. We show that $I\otimes_{\hat\sO_{T,0}} \sF\to I\sF$ is injective. Assume that $I=(u_1^{a_1}, u_2^{a_2})$; (for $I$ principal, the argument is the same.) Let $\alpha_i\in\sF$ so that
$$ u_1^{a_1}\otimes\alpha_1+ u_2^{a_2}\otimes\alpha_2\mapsto 0\in\sF.
$$
Since $\sO_T\to\sO_W$ is defined by $u_i\mapsto z_i$. Using $z_1z_2=0$, we get
$z_1^{a_1+1}\alpha_1=0$. Because $\sF|_{W_1}$ is flat over $0\in T_1$, this is possible only if $\alpha_1=z_2\beta$ for some $\beta\in\sF$. Then $u_1^{a_1}\otimes\alpha_1=u_1^{a_1}u_2\otimes\beta=0$. For the same reason, $u_2^{a_2}\otimes\alpha_2=0$. Hence $I\otimes_{\hat\sO_{T,0}} \sF\to I\sF$ is injective. This proves that $\sF$ is flat over $T$ near $0$.
\end{proof}

We have a parallel result.

\begin{lemm}
Let $D_1,D_2\sub X$ be smooth divisors intersecting transversally in a smooth variety.
Suppose a sheaf $\sF$ is normal to $D_1$ and $D_2$, then it is normal to the union $D_1\cup D_2$.
\end{lemm}

\begin{proof}
The proof is similar, and will be omitted.
\end{proof}

\subsection{Admissible coherent sheaves}

We shall study coherent sheaves on a simple degeneration $\pi\colon X\to C$.

\begin{defi}\lab{admissible}
We call a coherent sheaf $\sF$ on $\xn_0$ admissible if it is normal to all $D_i\sub \xn_0$.
Let $(\cX,p)$ be an $S$-family of expanded degenerations. Let $\sF$ be an $S$-flat family of coherent sheaves on ${\cX}$. We say $\sF$ is an $S$-family of admissible coherent sheaves if $\sF_s\defeq\sF|_{\cX_s}$ is admissible for every closed $s\in S$.
\end{defi}

We agree that any coherent sheaf on a smooth  $X_t$ is admissible.

\begin{prop}\lab{open-2}
Let $\sF$ be an $S$-flat family of coherent sheaves on $\cX$. Then the set $\{s\in S\mid \sF_s \ \text{is admissible} \}$ is open in $S$.
\end{prop}

\begin{proof}
This follows directly from Proposition \ref{flat-4}.
\end{proof}

Similarly, we have the relative version. We agree that for $(\cY,\cD_\pm)$ an $S$-family of relative pairs
and $s\in S$ a closed point, we denote $\cY_s=\cY\times_S s$ and $\cD_{\pm,s}=\cD_\pm\times_S s$.

\begin{defi}\lab{relative}
We call a coherent sheaf $\sF$ on $Y[\nn]_0$ relative to $D[n_\pm]_{\pm,0}$
if it is normal to all $D_i\in Y[\nn]$ and is normal to the distinguished divisor $D[n_\pm]_{\pm,0}$.
Let $(\cY,\cD_\pm)$ be an $S$-family of relative pairs. We say an $S$-flat sheaf $\sF$ on $\cY$ relative to $\cD_\pm$ if for every closed $s\in S$, $\sF_s$ is a sheaf on $\cY_s$ relative to $\cD_{\pm, s}$.
\end{defi}

\begin{prop}\lab{open-3}
Let $\sF$ be an $S$-flat family of coherent sheaves on $\cY$. Then the set $\{s\in S\mid \sF_s \ \text{is relative to }\cD_{\pm, s} \}$ is open in $S$.
\end{prop}

\begin{proof}
This follows directly from Corollary \ref{flat-2} and Proposition \ref{flat-4}.
\end{proof}

\vsp
For later study, we show that the failure of admissible property of a class of $\Gm$-equivariant quotient
sheaves are constant in $t$. Since this is a local study, we work with modules.
We let $B$ be an integral $\kk$-algebra of finite type; let $A$ be the $\Gm$ $\kk$-algebra
\beq\label{action}
A= B[z_1,z_2,t]/(z_1z_2);\quad z_1^\si=\si^a z_1,\ z_2^\si=z_2,\ t^\si=\si^b t; \quad  a\in \ZZ_+,\ b\in \ZZ_-.
\eeq
We let $R=A^{\oplus m}$ be an $A$-module with the $\Gm$-action acting on individual factors as in \eqref{action}.

Given an $A$-module $M$, for $f\in M$ we denote by $\ann(f)\sub A$ the annihilator of $f$: $\ann(f)=\{a\in A\mid af=0\}$. Let
$$I=(z_1,z_2)\sub A
$$
be the ideal generated by $z_1$ and $z_2$. We define $M_{I}=\{f\in M\mid \ann(f)\supset I^k\text{ for some $k\in \NN$}\}$. Namely, $M_{I}$ consists of elements annihilated by $I^k$ for some $k$.

We use the $\Gm$-spectral decomposition to study $\Gm$-sheaves. Given a $\Gm$-module $M$, we let
$M_{[\ell]}=\{v\in M\mid v^\si=\si^\ell v\}$. Since $\Gm$ is reductive and commutative, we have direct sum decomposition $M=\oplus_{\ell\in \ZZ} M_{[\ell]}$. We call an element $v\in M$ of weight $\ell$ if $v\in M_{[\ell]}$. By the weight assignments of $z_i$ and $t$, we see that for an element $f\in A$ of weight $\ell\geq 0$ and is divisible by $t$, then $f$ is divisible by $z_1$.

Let $A_0=A/(t)$ be the quotient ring. For any $R$-module $M$, we denote $M_0=M\otimes_A A_0$.
Let  $R_0=A_0^{\oplus m}=R\otimes_A A_0$, and $I_0=(z_1,z_2)\sub A_0$.

\begin{lemm}\label{A}
Let $\varphi: R\to M$ be a $\Gm$-equivariant quotient $A$-module. Suppose $M$ is $\kk[t]$-flat,
then the natural homomorphism $M_I\otimes_A A_0\to (M_0)_{I_0}$ is an isomorphism.
\end{lemm}

We next study the failure of the flatness of $M$ over $T=\kk[z_1,z_2]/(z_1z_2)$. We let $A^-=A/(z_2)$, let $M^-=M\otimes_A A^-$, $R^-=R\otimes_A A^-$, and define $K^-=\ker\{R^-\to M^-\}$. We consider the localization
$K^-_{(t)}$ of $K^-$ by the ideal $(t)$;\footnote{Since $(K^-)\llt=(K_{(t)})^-$, there is no confusion using
$K^-_{(t)}$ to denote either.} consider its further localization by $(z_1)$, its intersecting with $R^-\llt$,
and the quotient:
\beq\label{Cgen} \Bigl( (K^-\llt)_{(z_1)}\cap R^-\llt\Bigr)\otimes_{A^-\llt} A^-\llt/(z_1)\sub B[t,t\upmo]^{\oplus m}.
\eeq
By the construction, the inclusion is $\Gm$-invariant, thus the $B[t,t\upmo]$-submodule is generated by elements in $B^{\oplus m}$. In other words, there is a $B$-submodule
$C_{\text{gen}}\sub B^{\oplus m}$ such that as submodules of $B[t,t\upmo]^{\oplus m}$,
$$
C_{\mathrm{gen}}\otimes_B B[t,t\upmo]=\Bigl( (K^-\llt)_{(z_1)}\cap R^-\llt\Bigr)\otimes_{A^-\llt} A^-\llt/(z_1).
$$

Applying the same construction to the module $K_0^-=\ker\{ R_0^-\to M^-_0\}$, where $R^-_0=R^0\otimes_{A_0}A_0^-$, where $A_0^-=A_0/(z_2)$, and same for $M^-_0$, we obtain a submodule $C_0\sub B^{\oplus m}$
such that as submodules of $B^{\oplus m}$,
$$C_0=\Bigl((K_0^-)_{(z_1)}\cap R^-_0\Bigr)\otimes_{A_0^-} A_0^-/(z_1).
$$

\begin{lemm}\label{B}
Let the situation be as in Lemma \ref{A}. Then as $B$-modules, $C_0\sub B^{\oplus m}$ coincide
with $C_{\mathrm{gen}}\sub B^{\oplus m}$.
\end{lemm}

The proofs will be given in Appendix.

\subsection{Numerical criterion}

We introduce numerical criterion to measure the failure of a coherent sheaf normal to a closed subscheme.
This will be used to prove the properness of moduli spaces.

Let $I_l\sub \sO_{\xn_0}$ be the ideal sheaf of $D_l\sub \xn_0$. For a sheaf $\sF$ on $\xn_0$, as in the previous subsection, we define
$$\sF_{I_l}=\{ v\in\sF\mid \ann(v)\subset I_l^k\ \text{for some } k\in \ZZ_+\}.
$$
We define
$$ \sF\utf=\sF/(\oplus_{l=1}^{n+1}\sF_{I_l}).
$$
It is the sheaf $\sF$ quotient out its subsheaf supported on a sufficiently thickening of the singular loci of $\xn_0$.
We then denote $(\sF\utf)_l=\sF\utf|_{\Delta_l}$, and form
$$(\sF\utf)_{l,I_l}:=\bl(\sF\utf)_l\br_{I_l}\and
(\sF\utf)_{l,I_{l+1}}:=\bl(\sF\utf)_l\br_{I_{l+1}};
$$
they are subsheaves of $(\sF\utf)_l$ supported along $D_l$ and $D_{l+1}$ respectively.

\vsp
\noindent
{\textbf{Example}}: We give an example of non-admissible quotient sheaf of $\sO_X$. For simplicity, we consider the affine case where $Y=\Delta_1\cap \Delta_2\sub \bA^4$ is defined via $\Delta_1=\{(z_i)|z_2=0\}$ and $\Delta_2=\{(z_i)|z_1=0\}$. We let
$$\sF_1=\sO_{\Delta_1}/(z_4,z_3^3,z_3^2z_1),\and \sF_2=\sO_{\Delta_2}/(z_3,z_4^3,z_4^2z_2).
$$
Let $\iota_i:\Delta_i\to Y$ be the inclusion. We define $\sF=\ker\{\iota_{1*}\sF_1\oplus\iota_{2*}\sF\to k(0)\}$, where $k(0)$ is the structure sheaf of the origin $0\in \bA^4$. Then $\sF_1\utf=\sO_{\Delta_1}/(z_4,z_3^2)$, and $\sF_2\utf=\sO_{\Delta_2}/(z_3,z_4^2)$ (c.f. \eqref{77} below); further
$$\sF\utf=\ker\{\iota_{1*}\sF_1\utf\oplus \iota_{2*}\sF_2\utf\to k(0)\},\quad \mathrm{length}(\sF/\sF\utf)=2,
$$
and $\sF\utf|_{\Delta_1}$ has a dimension zero element support at $0$.

\vsp

For an integer $v$, we continue to denote by $\sF(v)=\sF\otimes p\sta H^{\otimes v}$, where $p: \xn_0\to X$ is the
projection.

\begin{defi}
We define the $l$-th error of $\sF$ be
\beq\label{err}
\err_l\sF=\chi(\sF_{I_l}(v))+\frac{1}{2}\chi\bl(\sF\utf)_{l,I_l}(v)\br+\frac{1}{2}\chi\bl(\sF\utf)_{l-1,I_{l}}(v)\br;
\eeq
we define the total error of $\sF$ be $\displaystyle \err\sF=\sum_{l=0}^{n+1} \err_l \sF$.
\end{defi}

\begin{lemm}
A sheaf $\sF$ on $\xn_0$ is admissible along $D_l$ if and only if all $\sF_{I_l}$, $(\sF\utf)_{l,I_l}$
and $(\sF\utf)_{l-1,I_{l}}$ are zero.
\end{lemm}

\begin{proof}
This is a local problem. We pick an affine open $W\sub \xn_0$ so that $W\sub \Delta_{l-1}\cup\Delta_l-D_{l-1}\cup D_{l+1}$. We let $W_1=W\cap \Delta_{l-1}$ and $W_2=W\cap \Delta_l$. We form $\xi:W\to T$ as in \eqref{split-2} so that
for $T_i\sub T$ the lines $\Ao\cong T_i\sub T$, we have $W_i=W\times_T T_i$; thus $\xi\upmo(0)=D_l\cap W$.

By Proposition \ref{flat} and Lemma \ref{2-flat}, $\sF|_W$ is admissible if and only if $\sF|_{W_i}$ are flat over $T_i$ near $0$. Let $J$ (resp. $J_i$) be the ideal sheaf of $W_1\cap W_2\sub W$ (resp. $W_1\cap W_2\sub W_i$);
let $(\sF|_W)_J$ be the torsion subsheaf of $\sF|_W$ supported along $W_1\cap W_2$, and let $\sF|_W\utf=(\sF|_W)/(\sF|_W)_J$. By the flatness criterion, this is true if and only if $(\sF|_W)_J=0$ and $\bl (\sF|_W\utf)|_{W_i})_{J_i}=0$ for $i=1,\, 2$. This proves that $\sF|_W$ is admissible if and only if all $\sF_{I_l}|_W$, $(\sF\utf)_{l,I_l}|_W$ and $(\sF\utf)_{l-1,I_{l}}|_W$ are zero. Going over a covering of $D_l\sub \xn_0$, the lemma follows.
\end{proof}

There is a useful identity expressing $\chi(\sF(v))$ in terms of $\err\sF$ and the Hilbert polynomial of
\beq\label{77}
\sF_l\utf\defeq \sF|_{\Del_l}/(\sF_{l,I_l}\oplus \sF_{l,I_{l+1}}).
\eeq
(It is $\sF|_{\Del_l}$ quotient out its subsheaf support along $D_l\cup D_{l+1}\sub \Del_l$.)

\begin{lemm} \label{delta}Let
$$\delta_{l,i}=\chi(\sF\utf_l(v))+\chi\bl (\sF\utf)_{l,I_{l+i}}(v)\br-\chi\bl \sF\utf|_{D_{l+i}}(v)\br,\ i=0, \, 1.
$$
Then we have the identity
\beq\label{the-id}
\chi(\sF(v))= \exc\sF+\frac{1}{2}\sum_{l=0}^{n+1}\bl\delta_{l,0}+\delta_{l,1}\br .
\eeq
\end{lemm}

\begin{proof}
Since $\sF\utf=\sF/\sF_I$,
\beq\label{eq-1}
\chi(\sF(v))=\chi(\sF\utf(v))+\chi(\sF_I(v)).
\eeq
For $\sF\utf$, we have the exact sequence
\beq\label{eq-2}
0\lra \sF\utf\lra \bigoplus_{l=0}^{n+1}\sF\utf|_{\Del_l}\lra \bigoplus_{l=1}^{n+1}\sF\utf|_{D_l}\lra 0.
\eeq
(Here we view both $\sF\utf|_{\Del_l}$ and $\sF\utf|_{D_l}$ as sheaves of $\sO_{\xn_0}$-modules.) Using
$$\chi(\sF\utf|_{\Del_l}(v))=\chi(\sF\utf_l(v))+\chi((\sF\utf)_{l,I_l}(v))+ \chi((\sF\utf)_{l,I_{l+1}}(v))
$$
and \eqref{eq-2}, after regrouping, we conclude
$$\chi(\sF(v))=\Bigl(\chi(\sF_I(v))+\frac{1}{2}\sum_{l=0}^{n+1}\sum_{i=0}^1\chi\bl (\sF\utf)_{l,I_{l+i}}(v)\br\Bigr)+
\frac{1}{2}\sum_{l=0}^{n+1}(\delta_{l,0}+\delta_{l,1}).
$$
This proves the lemma.
\end{proof}

We have the following positivity in case $\sF$ is a quotient sheaf of $p\sta\sV$ for a locally free sheaf $\sV$
on $X$.

\begin{lemm}\label{vanishing}
Suppose $\sF$ is a quotient sheaf of $p\sta\sV$. For $1\le l\le n$, the leading coefficients of  $\delta_{l,0}$ and $\delta_{l,1}$
are non-negative; if one of $\delta_{l,i}$ is zero, then the other is also zero, which happens when
$p\sta\sV|_{\Delta_l}\to \sF\utf_l$ is a pull back of a quotient $p\sta\sV|_{D_l}\to \sE$ of sheaves on $D_l$ via the projection
$\pi_l:\Delta_l\to D_l$.
\end{lemm}

\begin{proof}
Let $1\le l\le n$. The quotient $p\sta\sV\to \sF$ induces a quotient homomorphism $p\sta\sV|_{\Del_l}\to \sF\utf_l$. We let $\sK$ be its kernel, which fits into the exact sequence
$$0\lra \sK\lra p\sta\sV|_{\Delta_l}\lra \sF\utf_l\lra 0.
$$
Let $\pi_l: \Del_l\to D_l$ be the projection. We claim $R^1\pi_{l\ast}\sF\utf_l=0$. Indeed, since $\pi_l$ is a $\Po$-bundle, $R^{\ge 2}\pi_{l\ast}\sK=0$. By base change, $R^1\pi_{l\ast}(p\sta\sV|_{\Delta_l})=0$ since for all closed $x\in D_l$, $H^1(\pi_l\upmo(x), p\sta\sV|_{\pi_l\upmo(x)})=0$. Applying $\pi_{l\ast}$ to the above exact sequence, by the induced long exact sequence, we conclude that $R^1\pi_{l\ast}\sF\utf_l=0$. Therefore, since $p\sta H|_{D_l}$ is ample, for large $v$,
$$\chi(\sF\utf(v))=\chi(\pi_{l\ast}\sF\utf(v))=\chi((\pi_{l\ast}\sF\utf)(v)).
$$

On the other hand, the surjective homomorphisms $p\sta\sV|_{\Delta_l}\to\sF\utf\to \sF\utf|_{D_l}$
induces a surjective $\pi_{l\ast}\sF\utf\to  \sF\utf|_{D_l}$. This implies that the leading coefficient of
$\chi((\pi_{l\ast}\sF\utf)(v))-\chi(\sF\utf|_{D_l}(v))$ is non-negative; and is zero if and only if
$\pi_{l\ast}\sF\utf=\sF\utf|_{D_l}$.

Finally, we suppose $\delta_{l,0}=0$. Then $\pi_{l\ast}\sF\utf=\sF\utf|_{D_l}$. Using $\pi_l\sta \pi_{l\ast}\sF\utf\to \sF\utf$, we obtain a homomorphism $\pi_l\sta (\sF\utf|_{D_l})\to \sF\utf$. As this homomorphism is an isomorphism when restricted to $D_l$, it is injective. Suppose it has non-trivial cokernel, then $\chi(\sF\utf(v))\ne \chi(\pi_l\sta \sF\utf|_{D_l}(v))= \chi(\sF\utf|_{D_l}(v))$, a contradiction. This proves the lemma.
\end{proof}

A parallel result holds for coherent sheaves on $Y[\nn]_0$. For the singular divisor $D_l\sub Y[\nn]_0$, we define
$\err_l\sF$ be as in \eqref{err}. For the relative divisor $D[n_\pm]_{\pm,0}$, we let $I_\pm$ be the ideal sheaf of
$D[n_\pm]_{\pm,0}\sub Y[\nn]_0$, and define $\err_\pm\sF=\chi(\sF_{I_\pm}(v))$. We define
\beq\label{err-re}
\err\sF=\sum_{-n_-\le l\le n_+} \err_l\sF+\err_-\sF+\err_+\sF.
\eeq

\begin{lemm}\label{criteria-4}
A coherent sheaf $\sF$ on $Y[\nn]_0$ is relative to $D[n_\pm]_{\pm,0}$ if and only if $\excz=0$.
\end{lemm}

\section{Degeneration of Quot schemes and coherent systems}

We construct good degenerations of Quot schemes and moduli spaces of certain types of coherent systems. We shall focus on the case of Quot schemes. For coherent systems, we will comment on
the modification needed at the end of the section.

\subsection{Stable admissible quotients}
We let $\pi: X\to C$ be a simple degeneration. We fix a relative ample
line bundle $H$ on $X/C$, and fix a locally free sheaf $\sV$ on $X$.

We begin with admissible quotients on $\xn_0$. Let $\pnn:\xn_0\to X$ be the projection.

\begin{defi}\lab{admissible-quot}
We call a quotient (sheaf) $\pfrmob$ on $\xn_0$ admissible if $\sF$ is admissible.
\end{defi}

For two quotients $\frmobo$ and $\frmobt$ on $\xn_0$, an equivalence between them consists of a pair $(\sigma,\psi)$, where $\sigma\colon\xn_0\to\xn_0$ is an automorphism induced from the canonical $\Gm^{{}n}$ action on $\xn_0$, and $\psi\colon\sF_1\cong\sigma\sta\sF_2$ is an isomorphism, so that the following square is commutative:
$$\begin{CD}
p\sta\sV @>{\lambda_1}>> \sF_1\\
@V{\sigma^{\natural}}VV @V{\psi}VV\\
p\sta\sV\cong \sigma\sta p\sta\sV @>{\sigma\sta\lambda_2}>> \sigma\sta\sF_2.
\end{CD}
$$
Here the isomorphism $p\sta\sV\cong \sigma\sta p\sta\sV$ is the (unique) one
whose restriction to $\Delta_0\cup \Delta_{n+1}$ is the identity map.

Suppose $(\sigma,\psi_1)$ and $(\sigma,\psi_2)$ are autoequivalences of a quotient $\pfrmob$, then $\psi_2^{-1}\circ\psi_1$ is an automorphism of $\pfrmob$, which is identity. Therefore $\psi_1=\psi_2$. It follows that the group $\Aut_{\X}\phi$ of autoequivalences of $\pfrmob$ is a subgroup of $\gn$.

\begin{defi}\lab{stable-quot}
We say a quotient $\pfrmob$ on $\xn_0$ is stable if it is admissible and $\Aut_{\X}\phi$ is finite.
\end{defi}

Let $(\cX,p)\in\fX(S)$ be an $S$-family of expanded degenerations, let $\sF$ be a coherent sheaf on ${\cX}$ and $\pfrmob$ be a quotient. We call $\pfrmob$ an $S$-flat family of stable quotients if $\sF$ is flat over $S$, and for every closed point $s\in S$ the restriction $\phi_s\colon p^*\sV|_{\cX_s}\to\sF|_{\cX_s}$ (of $\phi$ to $\cX_s$) is stable.
\begin{lemm}\lab{stable-open}
Let $\pfrmob$ be an $S$-flat family of quotients on $(\cX,p)\in\fX(S)$.
Then the set $\{s\in S\mid \phi_s\colon p^*\sV|_{\cX_s}\to\sF|_{\cX_s}\ \text{is stable}\}$ is an open subset of $S$.
\end{lemm}

\begin{proof}
Because automorphism groups being finite is an open condition, the Lemma follows from Proposition \ref{open-2}.
\end{proof}

We define the category $\Qx$ of families of stable quotients. For any scheme $S$ over $C$, we define $\Qx(S)$ be the set of all $(\phi;\cX,p)$ so that $(\cX,p)\in \X(S)$ and $\pfrmob$ is an $S$-flat family of stable quotients on $\cX$. An arrow between $(\phi_1;\cX_1,p)$ and $(\phi_2;\cX_2,p)$ in $\Qx(S)$ is an arrow $\sigma\colon\cX_1\to\cX_2$ in $\X(S)$ so that $\phi_1\cong \sigma\sta\phi_2$. For $\rho: S\to T$, the map $\Qx(\rho): \Qx(T)\to \Qx(S)$ is defined by pull back.

Sending $(\phi; \cX,p)\in \Qx$ to the base scheme of $\cX$ defines $\Qx$ as a groupoid over $C$.

\begin{prop}\label{DM stack}
$\Qx$ is a Deligne-Mumford stack locally of finite type.
\end{prop}

\begin{proof}
First we show that $\Qx$ is a stack. We let $\text{Sch}_C$ be the category of schemes over $C$.
For any $S$ in $\text{Sch}_C$ and two families $\phi_1,\phi_2$ in $\Qx(S)$, we define a functor
\[ \Isom_S(\phi_1,\phi_2):\text{Sch}_C\to (\Sets)
\]
that associates to any morphism $\rho:S'\to S$ the set of isomorphisms in $\Qx(S')$ between $\rho^*\phi_1$ and $\rho^*\phi_2$. Since stable quotients have finite automorphism groups, by
a standard argument, $\Isom_S(\phi_1,\phi_2)$ is represented by a finite group scheme over $S$.
An application of descent theory shows that $\Qx$ is a stack.

Now we show that $\Qx$ admits an \'{e}tale cover by a Deligne-Mumford stack locally of finite type.
Let  $\pnn:\xn\to X$ be the projection; let $\Quot^{\pnn\sta \sV}_{\xn/\cn}$ be the Quot scheme on $\xn/\cn$ of $\pnn\sta\sV$.
We form the subset $\Quot^{\pnn\sta \sV,\st}_{\xn/\cn}\sub \Quot^{\pnn\sta \sV}_{\xn/\cn}$ of stable quotients as defined in
Definition \ref{stable-quot}. By Lemma \ref{stable-open}, it is open in $\Quot^{\pnn\sta \sV}_{\xn/\cn}$. Since $\gn$ acts on $\xn/\cn$,
it acts on $\Quot^{\pnn\sta \sV}_{\xn/\cn}$, and then on $\Quot^{\pnn\sta \sV,\st}_{\xn/\cn}$.
By the stable assumption, $\Gm^n$ acts with finite stabilizers on $\Quot^{\pnn\sta \sV,\st}_{\xn/\cn}$,
thus the quotient stack $[\Quot^{\pnn\sta \sV,\st}_{\xn/\cn}/\gn]$ is a Deligne-Mumford stack.

Let
\[ F_n: [\Quot^{\pnn\sta \sV,\st}_{\xn/\cn}/\gn]\lra \Qx
\]
be the morphism induced by the universal family over $\Quot^{\pnn\sta \sV,\st}_{\xn/\cn}$.
By construction $F_n$ is \'etale. Hence, the induced
\[  \coprod_{n\ge 0} F_n:\coprod_{n\ge 0} [\Quot^{\pnn\sta \sV,\st}_{\xn/\cn}/\gn]\lra \Qx
\]
is \'etale and surjective. This proves the Proposition.
\end{proof}

We define relative stable quotients on an expanded pair in the same way by replacing $\xn_0$ with
$(Y[\nn]_0,D[n_\pm]_{\pm,0})$. Let
$$\sV_0=\sV\otimes_{\sO_X}\sO_{Y},
$$
where $Y\to X$ is induced by the normalization $Y\to X_0\sub X$. Let
$$p\colon (Y[\nn]_0,D[n_\pm]_{\pm,0})\to (Y,D_\pm)
$$
be the projection. For any quotient $\phi: p\sta \sV_0\to\sF$ on $Y$, the group $\Aut_{\Y}\phi$ is defined in the same way as that of $\Aut_\fX\phi$, which is a subgroup of $\gn$.

\begin{defi}\lab{relative-quot}
Let $(Y[\nn]_0,D[n_\pm]_{\pm,0})$ be a relative pair. A relative quotient $\phi: \pnn\sta \sV_0\to\sF$ on
$(Y[\nn]_0,D[n_\pm]_{\pm,0})$ is a quotient so that $\sF$ is admissible and is normal to $D[n_\pm]_{\pm,0}$.
We call $\phi: \pnn\sta \sV_0\to\sF$ stable if in addition $\Aut_{\Y}\phi$ is finite.
\end{defi}

We define families of relative quotients on $(\cY,\cD_\pm,p)\in (\fD_\pm\sub\fY)(S)$ similarly. We have

\begin{prop}
Let $\phi: p\sta\sV_0\to\sF$ be an $S$-flat family of relative quotients on $(\cY,\cD_\pm)$. Then the restriction
$\phi_{\cD_\pm}\colon p\sta\sV_0|_{\cD_\pm}\to\sF|_{\cD_\pm}$ is an $S$-flat family of quotients on $\cD_\pm$.
\end{prop}

\begin{proof}
This follows from Corollary \ref{flat-2}.
\end{proof}

We remark that Lemma \ref{stable-open} still holds after replacing families $\cX/S$ by families $\cY/S$.
We define the category $\Qyd$ of families of relative stable quotients accordingly.

\begin{prop}
$\Qyd$ is a Deligne-Mumford stack locally of finite type.
\end{prop}

\begin{proof}
The proof is parallel to that of Proposition \ref{DM stack}.
\end{proof}

\subsection{Coherent systems}

Coherent systems we will consider are sheaf homomorphisms
$$\varphi\colon\sO_{\xn_0}\to\sF
$$
(or on $Y[\nn]_0$) so that $\sF$ is pure of dimension one and $\varphi$ has finite cokernel. Since an automorphism of $\varphi\colon\sO_{\xn_0}\to\sF$ is a sheaf isomorphism $\sigma\colon\sF\cong\sF$ so that $\sigma\circ\varphi=\varphi$, that $\sF$ is pure of dimension one and $\coker\varphi$ is finite implies that $\sigma$ is the identity map. We define the group $\Aut_{\X}\varphi$ be the collection of pairs $(\si,\xi)$ so that $\si\in\Gm^n$ and $\xi$ is an isomorphism of $\varphi: \sO_{\xn_0}\to \sF$ with $\si\sta\varphi: \sO_{\xn_0}=\si\sta \sO_{\xn_0}\to \si\sta \sF$; it is a subgroup of $\gn$.

\begin{defi}
We say a coherent system $\cohs$ admissible if both $\sF$ and $\coker\varphi$ are admissible.
We say it is stable if it is admissible and $\Aut_{\fX}\varphi$ is finite.
\end{defi}

Since $\coker\varphi$ has dimension zero and $\sF$ is pure, $\varphi$ is admissible implies that
$\coker\varphi$ is away from the singular locus of $\xn_0$.
We adopt the convention that any coherent system on a smooth $X_t$ is admissible and stable. We define families of stable coherent systems in the same way as families of stable quotients. We have

\begin{prop}
Let $\varphi\colon \sO_{\cX}\to\sF$ be an $S$-flat family of coherent systems on an expanded degeneration $(\cX,p)
\in\fX(S)$. Then the set $\{s\in S\mid \varphi_s\colon \sO_{\cX_s}\to\sF_s\ \text{is stable}\}$ is an open subset of $S$.
\end{prop}

We form the category $\P_{\X/\C}$ of families of stable coherent systems. We have

\begin{prop}\lab{coh-stack}
$\P_{\X/\C}$ is a Deligne-Mumford stack locally of finite type.
\end{prop}

Accordingly, we have the following relative version.

\begin{defi}
We say a coherent system $\cohsrel$ relative if both $\sF$ and $\coker\varphi$ are admissible, and $\coker\varphi$ is normal to $D_\pm[n_\pm]_0$. We say it is stable if it is admissible and $\Aut_{\Y}\varphi$ is finite.
\end{defi}

\begin{prop}
Let $\varphi\colon \sO_{\cY}\to\sF$ be an $S$-flat family of relative coherent systems on $(\cY,\cD_\pm)$.
Then the restriction $\varphi_{\cD_+}\colon \sO_{\cD_+}\to\sF|_{\cD_+}$ and $\varphi_{\cD_-}\colon \sO_{\cD_-}\to\sF|_{\cD_-}$ are $S$-flat families of quotient
sheaves on $\cD_+$ and $\cD_-$.
\end{prop}

\begin{proof}
This is because for a family of relative coherent systems $\varphi\colon \sO_{\cY}\to\sF$, $\coker\varphi$ is away from $\cD_+$ and $\cD_-$. Therefore, the restrictions $\varphi_{\cD_+}\colon \sO_{\cD_+}\to\sF|_{\cD_+}$ and $\varphi_{\cD_-}\colon \sO_{\cD_-}\to\sF|_{\cD_-}$ are surjective. The flatness follows from Corollary \ref{flat-2}.
\end{proof}

We form the stack $\P_{\fD_\pm\sub\fY}$ of families of relative coherent systems. Analogue to Proposition \ref{coh-stack}, we have

\begin{prop}
$\P_{\fD_\pm\sub\fY}$ is a Deligne-Mumford stack locally of finite type.
\end{prop}

\subsection{Components of the moduli stack}

The moduli stacks $\Qx$ and $\Mx$ can be decomposed into disjoint pieces according to the
topological invariants of the sheaves. We will discuss the case for Quot scheme; it is the same for the moduli of coherent systems.

We use Hilbert polynomials to keep track of the topological data of quotients. For any coherent sheaf $\sF$ on an  $(\cX,p)\in\fX(S)$, 
and for a closed $s\in S$, denote $\sF_s=\sF|_{\cX_s}$ and define
$$
\chi^H_{{\sF_s}}(v)= \chi(\sF_s\otimes p\sta H^{\otimes v}), \quad p:\cX_s\lra X, \quad v\in \mathbb Z.
$$

Let $P(v)$ be a fixed polynomial. We define $\Qxh(\kk)\sub\Qx(\kk)$ be the subset consisting of $[\cohs]\in \Qx(\kk)$ so that $\chi^H_{{\sF}}=P$. Since the Hilbert polynomials of a flat family of sheaves are locally constant in their parameter space, $\Qxh(\kk)\sub\Qx(\kk)$ is both open and closed. Thus it defines an open and closed substack
$\Qxh\sub\Qx$.

Similarly, we let $q:Y\to X$ and $p:Y[\nn]_0\to Y$ be the projections; for a sheaf $\sF$ on $Y[\nn]_0$,
we denote $\chi^H_{\sF}(v)=\chi(\sF\otimes p\sta q\sta H^{\otimes v})$.
We define the open and closed substack $\Qydp\sub \Qyd$ be so that $\Qydp(\bk)$ consists of relative stable quotients
$\pfrmob$ such that $\chi^H_{{\sF}}=P$.

For moduli of coherent systems, following the same procedure, we have open and closed substacks $\Mxh$ of $\Mx$ and
$\Mydh$ of $\Myd$.

We state the main theorems of the first part of this paper whose proofs will occupy the next section.

\begin{theo}\label{properA}
The Deligne-Mumford stacks $\Qxh$ and $\Mxh$ are separated, proper over $C$, and of finite type.
\end{theo}

\begin{theo}\label{properB}
The Deligne-Mumford stacks $\Qydp$ and $\Mydh$ are separated, proper and of finite type.
\end{theo}

\section{Properness of the moduli stacks}

We apply the valuative criterion to prove Theorem \ref{properA} and \ref{properB}. We let $S$ be an affine scheme such that $\Gamma(\sO_S)$ is a discrete valuation $\kk$-algebra; let $\eta$ and $\eta_0\in S$ be its generic and closed point. We will often denote by $S'\to S$ a finite base change; in this case we denote by $\eta'$ and $\eta'_0$ its generic and closed points. 

For any quotient homomorphism $\phi: p\sta\sV\to\sF$ on $(\cX,p)\in\fX(S)$, we denote by
$\phi_\eta$ and $\phi_{\eta_0}$ the restriction of $\phi$ to $\cX_\eta=\cX\times_S \eta$ and $\cX_{\eta_0}$,
respectively.

\begin{prop}\lab{existence}
Let $(S, \eta, \eta_0)$ be as stated. Given any $(\phi_\eta,\cX_\eta)\in\Qxh(\eta)$, we can find a finite base change $S'\to S$ so that $(\phi_\eta,\cX_\eta)\times_\eta\eta' \in  \Qxh(\eta')$ extends to a family in $ \Qxh(S')$.
Further, the same conclusion holds for $\Qydp$.
\end{prop}

\begin{prop}\lab{unique}
Let $(S, \eta, \eta_0)$ be as stated. Given $(\phi_1,\cX_1)$, $(\phi_2,\cX_2)\in \Qxh(S)$, any isomorphism $(\phi_1,\cX_1)\times_S \eta\cong (\phi_2,\cX_2)\times_S \eta $ in $\Qxh(\eta)$ extends to an isomorphism $(\phi_1,\cX_1)\cong(\phi_2,\cX_2)$ in $\Qxh(S)$. Further, the same conclusion holds for $\Qydp$.
\end{prop}

We need an ordering on a set of polynomials.

\begin{defi}\lab{ordering}
We let $\sA\sta \sub \bQ[k]$ be the set of polynomials whose leading terms are of the form $a_r\frac{k^r}{r!}$ with $a_r\in\ZZ_+$;  let $\sA=\sA\sta\cup\{0\}$. For any $f(k)=a_r\frac{k^r}{r!}+\cdots$ and $g(k)=b_s\frac{k^s}{s!}+\cdots$ in $\sA\sta$, we say $f(k)\prec g(k)$ if either $r<s$, or $r=s$ and $a_r<b_s$; we say $f(k)\approx g(k)$ if $r=s$ and $a_r=b_s$. We agree that $0$ is $\prec$ to all other elements.
\end{defi}

For convenience, we use $\preccurlyeq$ to denote $\prec$ or $\approx$.

\begin{lemm}\lab{chain}
The set $\sA$ satisfies the descending chain condition.
\end{lemm}
\begin{proof}
For any sequence $f_1(k)\succcurlyeq f_2(k)\succcurlyeq\cdots$, since $0$ is the minimal element in $\sA$, we can assume $f_i(k)\ne 0$ for all $i$. By Definition \ref{ordering}, we know the pairs $(r,a_r)$ of the degrees and the leading coefficients of polynomials $f_i(k)$ decrease according to the lexicographic order. Since the pairs consist of non-negative integers, we can find an integer $n$, so that $f_n(k)\approx f_{n+1}(k)\approx\cdots$.
\end{proof}

\subsection{The completeness I}\label{com-I}
\def\Gme{G_{m,\eta}}

Let $(S,\eta,\eta_0)$ be as stated in the beginning of this section, and $S\to C$ be a scheme over $C$;
let $(\phi_\eta:p_\eta\sta\sV\to\sF_\eta)\in \Qxh(\eta)$ be a quotient on $(\cX_\eta,p_\eta)\in\fX(\eta)$. In this subsection, we assume $\cX_\eta$ is smooth. Since the case where $S\to C$ sends $\eta_0$ to $C-0$ is trivially true, we assume it sends $\eta_0$ to $0\in C$.

\begin{lemm}\label{1.1}
We can extend $\phi_\eta$ to a family of $S$-flat quotient $\phi:p\sta\sV\to\sF$ on an $(\cX,p)\in\fX(S)$
such that $\Aut_\fX \phi_{\eta_0}$ is finite.
\end{lemm}

\begin{proof}
Since $\cX_\eta$ is smooth, $S\to C$ sends $\eta\in S$ to a point in $C-0$. Using that $S$ is a $C$-scheme,
we define $\cX=X\times_C S$, and denote $p:\cX\to X$ the projection. Because Grothendieck's quot-scheme is
proper, the quotient $\phi_\eta$ on $\cX_\eta$ extends to a quotient $\phi: p\sta\sV\to\sF$, flat over $S$.
Since $\cX_{\eta_0}$ has no added $\Delta_l$, $\Aut_\fX \phi_{\eta_0}$ is $\{e\}$.
\end{proof}

We will show that by varying the extensions $(\cX,p)\in\fX(S)$ of $\cX_{\eta}$, we can decrease $\Err \sF_{\eta_0}$
while keeping $\Aut_\fX\phi_{\eta_0}$ finite. By the descending chain condition, this implies that we can find an
extension with stable quotient at special fiber.

\begin{lemm}\label{1.2}
Let $\phi_{\eta}: p_\eta\sta\sV\to\sF_\eta$ be a quotient as in Lemma \ref{1.1},
and let $\phi: p\sta\sV\to\sF$ be an $S$-flat quotient that extends $\phi_{\eta}$ with $\Aut_\fX\phi_{\eta_0}$ finite. Suppose $\err\sF_{\eta_0}\ne 0$, then we can find a finite base change $S'\to S$, an $S'$-flat quotients $\phi'\colon p'^*\sV\to\sF'$ on $(\cX',p')\in\fX(S')$ such that
\begin{enumerate}
\item $\cX'_{\eta'}\cong \cX_\eta\times_{\eta}\eta'\in \fX(\eta')$, and under this isomorphism $\phi'_{\eta'}=\phi_\eta\times_\eta\eta'$;
\item $\Aut_{\X}(\phi^{\,\prime}_{\eta_0'})$ is finite, and
\item $\exc{\sF'_{\eta'_{0}}}\prec\exc{\sF_{\eta_0}}$.
\end{enumerate}
\end{lemm}

We prove the Lemma by proving a sequence of lemmas. Since $S$ is local,
$$\cX=\xn\times_{\cn}S\quad \text{for a} \ \  \xi: S\to \cn
$$
such that $\xi(\eta_0)=0\in \cn$. We let $u$ be a uniformizing parameter of $S$. Denoting by
$\pi_n: \cn\to\Ano$ the projection, we express
\beq\label{zeta}
\pi_n\circ\xi=\bl c_1u^{e_1},\ldots, c_{n+1}u^{e_{n+1}}\br, \quad c_i\in \Gamma(\sO_S)\sta.
\eeq
($\Gamma(\sO_S)\sta$ are the invertible elements in $\Gamma(\sO_S)$.) Since $\xi(\eta_0)=0$, all $e_i
\ge 0$. Since $\err\sF_{\eta_0}\ne 0$, we pick an $1\le l\le n$ so that
$$\deg \err_l \sF_{\eta_0}=\deg\err\sF_{\eta_0}.
$$
We let
$$\tau_{l}:\cn\times\Gm\to\cnpo
$$
be induced from the $\Anpo\times\Gm\to\Anpt$:
\beq\label{GA}
(t_1,\cdots,t_{n+1},\sigma)\mapsto(t_1,\cdots,t_{{l}-1},\sigma^{e_l},\sigma^{-e_l}t_{l},t_{{l}+1},\cdots,t_{n+1}).
\eeq
We then introduce
\[ \xi_{l}= \tau_{l}\circ( \xi, \text{id}): S\times\Gm\stackrel{}\lra\cn\times\Gm\stackrel{}\lra\cnpo,
\]
and let $\cX':=\xi_{l}\sta\xnpo$ over $S\times\Gm$ be the pull back family. Because of the canonical isomorphism $\tau_{l}\sta\xnpo\cong\xn\times\Gm$ as families over $\cn\times \Gm$,
$$\cX'\cong\xi^*\xn\times\Gm=\cX\times\Gm.
$$
We let $p':\cX'\to X$ and $\pi_1:\cX'\to\cX$ be the projections.

We let $\phi'=\pi_1\sta\phi:p^{\prime\ast}\sV\to\sF'$ be the pullback quotient
sheaves (of $\phi$).
Since $(\cX',p')$ is induced by $\xi_l:S\times\Gm\to\cnpo$, the family of quotients $\phi'$ induces a $\cnpo$-morphism
\beq\label{morphism}
f_l:S\times\Gm\lra\Quot^{p\sta\sV, P}_{\xnpo/\cnpo}.
\eeq
For simplicity, we abbreviate $\Qu= \Quot^{p\sta\sV, P}_{\xnpo/\cnpo}$. 

We now construct a regular $\Gm$-surface $V$ and $\Gm$-morphisms that fit into the following commutative diagram
\begin{displaymath}
\xymatrix{ V \ar[r]^{\bar\jmath\qquad\qquad} & \Qu\defeq \Quot^{p\sta\sV, P}_{\xnpo/\cnpo}\ar[d]^{\pi} \\
 S\times\Gm\ar[u]^{j}\ar[r]^{\xi_l} \ar[ur]^{f_l} & \cnpo}
\end{displaymath}
so that $\pi\circ \bar\jmath:V\to C[n+1]$ is proper.

We first loot at the composite
\beq\label{B-1}
\xi_l\circ\pi_n: S\times \Gm\lra C[n+1]\lra \Anpt;
\eeq
it is given by
$$\xi_l\circ\pi_n(u,t)=(c_1u^{e_1},\ldots, c_{l-1} u^{e_{l-1}}, t^{e_l}, c_l t^{-e_l}u^{e_l}, c_{l+1} u^{e_{l+1}},\ldots,c_{n+1}u^{e_{n+1}}).
$$
We embed $S\times\Gm\sub S\times\Ao$ via the embedding $\Gm\sub \Ao$ so that the induced $\Gm$-action on
$\Ao$ is $t^\si=\si t$. We then blow up $S\times\Ao$ at $(\eta_0, 0)\in S\times\Ao$, let $\widetilde{S}$ be the proper transform of $S\times 0$, and form
$$V'=\text{bl}_{(\eta_0,0)} S\times \Ao-\widetilde S.
$$
Note that $V'\sub S\times\Ao\times \Ao$ is defined via $u=vt$, where $v$ is the standard coordinate of the
last $\Ao$-factor.

By construction, \eqref{B-1} extends to a $V'\to \Anpt$, in the form
\beq\label{B-2}
(v,t)\mapsto (c_1u^{e_1},\ldots, c_{l-1} u^{e_{l-1}}, t^{e_l}, c_l v^{e_l}, c_{l+1} u^{e_{l+1}},\ldots,c_{n+1}u^{e_{n+1}}),
\ u=vt.
\eeq
Because $C[n+1]\to \Anpt$ is proper over a neighborhood of $0\in\Anpt$, and because all $e_i>0$,
(c.f. \eqref{zeta}), $V'\to\Anpt$ lifts to a unique
$$\xi_l': V'\lra C[n+1],
$$
extending $\xi_l: S\times\Gm\to \cnpo$.

We let $\Gm$ acts on $S\times\Ao\times \Ao$ be $(u,t,v)^\si=(u,\si t, \si\upmo v)$. It leaves $V'\sub S\times\Ao\times \Ao$ invariant, thus induces a $\Gm$-action on $V'$. We let $E\sub V'$ be the exceptional divisor of $V'\to S\times\Ao$; let $E'\sub V'$ be the proper transform of $\eta_0\times \Ao$. In coordinates, $E=(t=0)$ and $E'=(v=0)$.

By construction, $f_l$ is a morphism from $V'-E$ to $\Qu$. Since $\Qu$ is proper over $C[n+1]$, $f_l$ extends to $\ti f_l: U\to \Qu$ for an open $U\sub V'$ that contains $V'-E$ and the generic point of $E$. On the other hand, since all schemes and morphisms are $\Gm$-equivariant, $U\sub V'$ can be made $\Gm$-invariant. Therefore, either $U=V'$ or $U=V'-\{o\}$, where $\{o\}=E\cap E'$.

We now consider the case $U=V'-\{o\}$. Since $\Qu$ is proper over $C[n+1]$, after successive blowing up, say
$$b: V\lra V',
$$
we can extend $\ti f_l: V'-\{o\}\to \Qu$ to a morphism
$$\bar \jmath: V\to \Qu.
$$
Since all the relevant schemes and morphisms are $\Gm$-equivariant, we can make the blowing-up $V\to V'$ $\Gm$-equivariant and the extension $\bar\jmath$ $\Gm$-equivariant.

Since $V\to V'$ is a $\Gm$-equivariant blowing up, and since the $\Gm$-action on the tangent space of the (only) fixed point $o\in V'$ has weights $e_l$ and $-e_l$, the exceptional divisor of $V\to V'$ can be made a chain of rational curves $\Sigma_1,\ldots,\Sigma_k$. We let $\Sigma_0\sub V$ (resp. $\Sigma_{k+1}\sub V$) be the proper transform of $E'\sub V'$ (resp. $E\sub V'$); then possibly after reindexing,
$$\Sigma\defeq \Sigma_0\cup\Sigma_1\cup\ldots\cup\Sigma_k\cup\Sigma_{k+1}
$$
forms a connected chain of rational curves; namely, $\Sigma_i\cap\Sigma_{i+1}\ne\emptyset$, for $0\le i\le k$. Using the explicit expression \eqref{B-2}, we conclude that under the morphism
\beq\label{pib}
\pi_n\circ\xi_l'\circ b: V\lra \Anpt,
\eeq
$\Sigma_1,\ldots, \Sigma_k$ are mapped to $0\in \Anpt$, and $\Sigma_0$ (resp. $\Sigma_{k+1}$)
is mapped to the line $\ell_{l}=\{ t_i=0, i\ne l\}\sub \Anpt$ (resp. $\ell_{l+1}\sub\Anpt$).
(Recall $\Sigma_0$ is the proper transform of $(v=0)$ and $\Sigma_{k+1}$ of $(t=0)$.)

\vsp
The proof of Lemma \ref{1.2} will be carried out by studying the pull back of the universal family of
$\Qu$ via $\bar \jmath : V\to \Qu$. We fix our convention on this pull back family. In the remainder of
this subsection, we denote
\beq\label{tilX}
\bl \ti p: \wti\cX=X[n+1]\times_{C[n+1]}V\lra X \br\in\fX(V);
\eeq
we denote $\Phi$ the universal family on $\Qu$ and denote $\ti\phi=\bar \jmath\sta \Phi$:
\beq\label{tilphi}
\ti\phi: \ti p\sta\sV\lra  \wti\sF \quad \text{on}\ \ \ti p: \wti\cX\to X.
\eeq
For any closed subscheme $A\sub V$, we use $\ti \phi_A$ to denote the restriction of $\ti\phi$ to $\wti\cX_A\defeq \wti\cX\times_V A$:
$$\ti\phi_A: \ti p_A\sta\sV\lra\wti \sF_A \quad \text{on}\ \ \ti p_A: \wti\cX_A\to X.
$$

\begin{lemm}
The family $\ti\phi$ is $\Gm$-equivariant, where the $\Gm$-action is the one induced from the
$\Gm$-morphism $\bar\jmath$. The chain of rational curves $ \Sigma$ is $\Gm$-invariant, and the
$\Gm$-fixed points of $\Sigma_i$ are $q_i=\Sigma_i\cap \Sigma_{i-1}$ and $q_{i+1}$.
\end{lemm}

\begin{proof}
The first part follows from that $\bar\jmath$ is $\Gm$-equivariant. The second part follows from that
$V\to V'$ is a successive $\Gm$-equivariant blowing up, and that $\Gm$ acts on the tangent space $T_0V'$
with weights $e_l$ and $-e_l$.
\end{proof}

\begin{lemm}
The fiber of $\wti\cX_{\Sigma_0}$ over $a\ne q_1\in \Sigma_0$ (resp. $a=q_1$) is $\xn_0$ (resp. $X[n+1]_0$);
the family $\wti\cX_{\Sigma_0}$ is a smoothing of the divisor $D_l\sub \wti\cX_{q_1}\cong X[n+1]_0$.
The $\Gm$-action on $\wti\cX_{q_1}\cong X[n+1]_0$ leaves all $\Delta_i\sub X[n+1]_0$ except $\Delta_l$ fixed, and
acting on $\Delta_l$ with fixed loci $D_l\cup D_{l+1}$.
\end{lemm}

\begin{proof}
By the construction of $\xnpo\to \cnpo$, for the $l$-th coordinate line $\ell_l\sub \Anpt$,
$\xnpo\times_{\Anpt}\ell_l$ is a family over $\ell_l$ whose fiber over $a\ne 0\in\ell_l$ is isomorphic to $\xn_0$,
and whose fiber over $0\in\ell_l$ is isomorphic to $\xnpo_0$; the family is a smoothing of the $l$-th singular divisor
$D_l\sub \xnpo_0$.

Applying this to the Lemma, knowing that $\Sigma_0\to \Anpt$ (cf. \eqref{pib}) is mapped onto the coordinate line
$\ell_l$, the first part of the lemma follows immediately.

For the second part, we need to understand the $\Gm$-action on
$$\xnpo_{\ell_l}\defeq \xnpo\times_{\Anpt}\ell_l.
$$
Recall the $\Gm$-action on $\Anpt$ is via
$$(z)^\si=(z_1,\ldots,z_{l-1}, \si^{e_l}z_l, \si^{-e_l}z_{l+1}, z_{l+2},\ldots,z_{n+2}).
$$
By the construction of $\xnpo/\cnpo$, this $\Gm$-action on $\xnpo_0$ leaves $\Del_i\sub\xnpo_0$
except $\Delta_l$ fixed, and leaves $\Delta_l$ invariant with fixed loci $D_l\cup D_{l+1}$.
(This can be seen using explicit description of $\xnpo$; it is also apparent in case $n=0$, since then
$l=1$ and the $\Gm$-action on $\Delta_0$ can only be trivial.) This proves the second part of the lemma.
\end{proof}

We have a parallel Lemma.

\begin{lemm}
The fiber of $\wti\cX_{\Sigma_{k+1}}$ over $a\ne q_{k+1}\in \Sigma_{k+1}$ (resp. $a=q_{k+1}$) is $\xn_0$ (resp. $X[n+1]_0$); the family $\wti\cX_{\Sigma_{k+1}}$ is a smoothing of the divisor $D_{l+1}\sub \wti\cX_{q_{k+1}}\cong X[n+1]_0$. The $\Gm$-action on $\wti\cX_{q_{k+1}}$ leaves all $\Delta_i\sub \wti\cX_{q_{k+1}}$ except $\Delta_{l}$ fixed, and acting on $\Delta_l$ with fixed loci $D_l\cup D_{l+1}$.
\end{lemm}

Using that the families over $\Sigma_i$, $1\le i\le k$ are all pull backs of the central fiber $\xnpo_0$ over $0\in\cnpo$, and combined with the results proved in the previous two Lemmas, we have

\begin{lemm}
For $1\le i\le k$, $\wti\cX_{\Sigma_i}\cong \xnpo_0\times\Sigma_i$; the $\Gm$-action on $\wti\cX_{\Sigma_i}$
are the product action of the $\Gm$-action on $\Sigma_i$, and the $\Gm$-action on $\wti\cX_{q_1}$, (which is identical to that on $\wti\cX_{q_{k+1}}$).
\end{lemm}


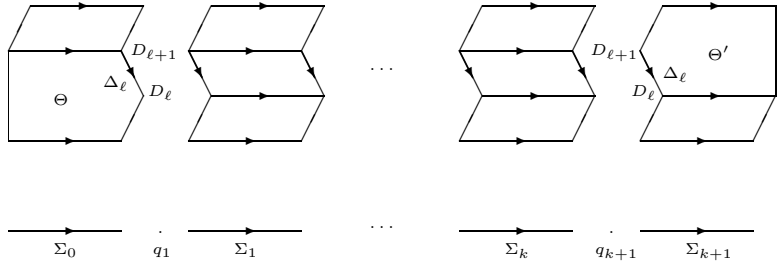
\begin{figure}[h!]
\begin{center}
\setlength{\unitlength}{0.6cm}
\begin{picture}(25,6)(1,-1)
\linethickness{0.075mm}\tiny
\put(3,2){\line(1,0){2.5}}
\put(3,4){\line(1,0){2.5}}
\put(3,4){\line(0,-1){2}}
\put(3,4){\line(1,2){0.5}}
\put(5.5,4){\line(1,-2){0.5}}
\put(5.5,2){\line(1,2){0.5}}
\put(5.5,4){\line(1,2){0.5}}
\put(6,5){\line(-1,0){2.5}}

\put(3.5,2){\vector(1,0){1}}
\put(3.5,4){\vector(1,0){1}}
\put(3.8,5){\vector(1,0){1}}
\put(5.5,4){\vector(1,-2){0.3}}

\put(6.1,3){$D_{\ell}$}
\put(5.7,3.9){$D_{\ell+1}$}
\put(5.1,3.2){$\Delta_{\ell}$}
\put(4,2.8){$\Theta$}

\put(3,0){\line(1,0){2.5}}\put(4,-0.5){$\Sigma_0$}\put(6.3,-0.1){$\cdot$}\put(6.2,-0.5){$q_1$}
\put(3,0){\vector(1,0){1.5}}
\put(7,2){\line(1,0){2.5}}
\put(7,4){\line(1,0){2.5}}
\put(7,4){\line(1,-2){0.5}}
\put(7,2){\line(1,2){0.5}}
\put(7,4){\line(1,2){0.5}}
\put(9.5,4){\line(1,-2){0.5}}
\put(9.5,2){\line(1,2){0.5}}
\put(9.5,4){\line(1,2){0.5}}
\put(10,5){\line(-1,0){2.5}}
\put(10,3){\line(-1,0){2.5}}

\put(7.5,2){\vector(1,0){1}}
\put(7.5,4){\vector(1,0){1}}
\put(7.8,3){\vector(1,0){1}}
\put(7.8,5){\vector(1,0){1}}
\put(9.5,4){\vector(1,-2){0.3}}
\put(7,4){\vector(1,-2){0.3}}

\put(7,0){\line(1,0){2.5}}\put(8,-0.5){$\Sigma_1$}
\put(7,0){\vector(1,0){1.5}}
\put(11,3.5){$\cdots$}\put(11,0){$\cdots$}

\put(13,2){\line(1,0){2.5}}
\put(13,4){\line(1,0){2.5}}
\put(13,4){\line(1,-2){0.5}}
\put(13,2){\line(1,2){0.5}}
\put(13,4){\line(1,2){0.5}}
\put(15.5,4){\line(1,-2){0.5}}
\put(15.5,2){\line(1,2){0.5}}
\put(15.5,4){\line(1,2){0.5}}
\put(16,5){\line(-1,0){2.5}}
\put(16,3){\line(-1,0){2.5}}

\put(13.5,2){\vector(1,0){1}}
\put(13.5,4){\vector(1,0){1}}
\put(13.8,3){\vector(1,0){1}}
\put(13.8,5){\vector(1,0){1}}
\put(15.5,4){\vector(1,-2){0.3}}
\put(13,4){\vector(1,-2){0.3}}

\put(13,0){\line(1,0){2.5}}\put(14,-0.5){$\Sigma_k$}\put(16.3,-0.1){$\cdot$}\put(16,-0.5){$q_{k+1}$}
\put(13,0){\vector(1,0){1.5}}
\put(17,2){\line(1,0){2.5}}
\put(17,4){\line(1,-2){0.5}}
\put(17,2){\line(1,2){0.5}}
\put(17,4){\line(1,2){0.5}}
\put(20,5){\line(0,-1){2}}
\put(19.5,2){\line(1,2){0.5}}
\put(20,5){\line(-1,0){2.5}}
\put(20,3){\line(-1,0){2.5}}

\put(17.5,2){\vector(1,0){1}}

\put(17.8,3){\vector(1,0){1}}
\put(17.8,5){\vector(1,0){1}}
\put(17,4){\vector(1,-2){0.3}}

\put(16.8,3){$D_{\ell}$}
\put(15.9,3.9){$D_{\ell+1}$}
\put(17.5,3.4){$\Delta_{\ell}$}
\put(18.5,3.8){$\Theta'$}

\put(17,0){\line(1,0){2.5}}\put(18,-0.5){$\Sigma_{k+1}$}
\put(17,0){\vector(1,0){1.5}}
\end{picture}
\end{center}
\caption{{\small In the figure, the slated lines represent $\Delta_i$; the horizontal lines represent $\Delta_i\times\Sigma_j$;
the arrows represent the $\Gm$-action; lines w/o arrows are fixed by $\Gm$.}}\label{lll}
\end{figure}


In the figure, the left column represents $\wti\cX_{\Sigma_0}$, of which only $\Delta_{l+1}\times\Sigma_0$ (the top parallelgram) and the $\Theta$ are shown. The piece $\Theta$ is the blowing up
of $\Delta_{l-1}\times\Sigma_0$ along $D_l\times q_1$, where $\Delta_{l-1}\sub\xn_0$. We endow
$\Theta$ with the $\Gm$-action induced by the product action on $\xn_0\times\Sigma_0$, where
$\Gm$-acts on $\Delta_{l-1}$ trivially, and acts on $\Sigma_0$ by that induced from the $\Gm$-action on $V$.
The family $\wti\cX_{\Sigma_0}$ is by replacing $\Del_{l-1}\times\Sigma_0\sub \xn_0\times\Sigma_0$
with $\Theta$.

The right column represents $\wti\cX_{\Sigma_{k+1}}$. The piece $\Theta'\sub \wti\cX_{\Sigma_{k+1}}$ is constructed similarly: it is the blowing up of $\Delta_l\times\Sigma_{k+1}$ along
$D_l\times q_{k+1}$; the total family $\wti\cX_{\Sigma_{k+1}}$ is by replacing $\Delta_l\times\Sigma_{k+1}$ in
$\xn_0\times\Sigma_{k+1}$ by $\Theta'$. The $\Gm$-action is the one induced from the product action
on $\xn_0\times\Sigma_{k+1}$, where the action on $\xn_0$ is via the  trivial action, and on $\Sigma_{k+1}$ is
via the one induced from that on $V$.

The next lemma explains the role of the families $\wti\cX_{\Sigma_i}$ in our proof of Lemma \ref{1.2}

\begin{lemm}\label{5.11}
For $a\in\Sigma_0-q_1$ or $a\in \Sigma_{k+1}-q_{k+1}$, $\ti\phi_a: \ti p_a\sta\sV\to \wti\sF_a$ on $\wti\cX_a$ is isomorphic to $\phi_{\eta_0}: p_{\eta_0}\sta\sV\to\sF_{\eta_0}$.
\end{lemm}

\begin{proof}
We comment that since $\cnpo=C\times_{\Ao}\Anpt$, a morphism $h:S\to \cnpo$ is given by a pair of
morphisms $h': S\to C$ and $h\dpri: S\to \Anpt$ so that their corresponding compositions $S\to C\to \Ao$
and $S\to \Anpt\to\Ao$ coincide.

We pick a morphism $\varphi_1: S\to V$ that is the lift of $S=S\times 1\mapright{\sub} S\times\Gm$. By the description of $V\to V'\to \Anpt$ (cf. \eqref{B-2}), we see that $\varphi_1(\eta_0)\in V$ lies over $(\ldots, 0,1,0,\ldots)\in\Anpt$, thus $\varphi_1(\eta_0)\in\Sigma_{k+1}-q_{k+1}$.

By the construction of $\varphi_1$, we see that the composite $\bar\jmath\circ\varphi_1: S\to V\to\Qu$ coincide with the restriction of $f_l$ (cf. \eqref{morphism}) to $S\times 1$: $\jmath\circ\varphi_1=f_l|_{S\times 1}$. Since $f_l$ is induced by the family $\phi$, we obtain
$$\phi\cong (\bar\jmath\circ\varphi_1)\sta \Phi\cong \varphi_1\sta\ti\phi,
$$
where $\Phi$ is the universal family of $\Qu$. Let $a'=\varphi_1(\eta_0)$; this proves $\ti\phi_{a'}\cong \phi_{\eta_0}$. Finally, since all points in $\Sigma_{k+1}-q_{k+1}$ form a single $\Gm$-orbit, for $a\in \Sigma_{k+1}-q_{k+1}$, $\ti\phi_a\cong \ti\phi_{a'}\cong \phi_{\eta_0}$. This prove the part of the Lemma
for the case $\Sigma_{k+1}-q_{k+1}$.

For the other case, we let $\varphi_2: S\to V$ be the lift of $(1_S,\rho):S\to S\times\Ao$,
where $\rho: S\to\Ao$ is via $\rho\sta(t)=u$. By the construction, we see that $\varphi_2(\eta_0)\in \Sigma_0-q_1$.

We let $h_i=\pi\circ\bar\jmath\circ \varphi_i: S\to \cnpo$ be the composite of $\varphi_i$ with the tautological
$V\to\cnpo$. By inspection, we see that the composites of $h_1$ and $h_2$ with $\cnpo\to C$ are identical, and their composites with $\Anpt\to\Ao$ are of the form
$$h\dpri_1(u)=(\cdots, 1, c_l u^{e_l},\cdots)\and h\dpri_2(u)=(\cdots,u^{e_l}, c_l,\cdots).
$$
Here the expressed terms are in the $l$ and $(l+1)$-th places, and the omitted terms of $h_1\dpri$ and $h_2\dpri$
are identical.

We let $\wti\cX_i\defeq\xnpo\times_{h_i}S$. Using the isomorphism $\tilde\tau_{I,I',X}$ in \eqref{tauX} with $I=[n+2]-\{l\}$ and $I'=[n+2]-\{l+1\}$, we conclude that
\begin{enumerate}
\item the generic points $h_1(\eta)$ and $h_2(\eta)$ lie in the same $\Gm^{n+1}$-orbit;
\item there is an isomorphism $\wti\cX_1\cong \wti\cX_2$ extending the isomorphism
$\wti\cX_1\times_S \eta\cong \wti\cX_2\times_S\eta$ given by the $\Gm^{n+1}$-action in (1).
\end{enumerate}

Let $\varphi_2\sta\ti\phi$ be the pull back of $\ti\phi$ via $\varphi_2: S\to V$; it is an $S$-flat family of quotient
sheaves on $\ti\cX_2$.
Since $\varphi_1(\eta)$ and $\varphi_2(\eta)$ lie in the same $\Gm$-orbit in $V$, (following
from the construction,) we have an induced isomorphism
\beq\label{iso-eta}
(\varphi_1\sta\ti\phi)_\eta\cong (\varphi_2\sta\ti\phi)_\eta.
\eeq
(Recall $(\varphi_1\sta\ti\phi)_\eta=(\varphi_1\sta\ti\phi)\times_S \eta$.) As the $\Gm^{n+1}$-action on $\Qu$
is induced by the $\Gm^{n+1}$-action on $\xnpo/\cnpo$, the isomorphism \eqref{iso-eta} is compatible with the isomorphism $\wti\cX_1\times_S\eta\cong \wti\cX_2\times_S\eta$ in (1).

Finally, using $\wti\cX_1\cong \wti\cX_2$ given by (2), we pull back the family $\phi$ on $\wti\cX_1\cong \cX$
to a quotient family $\bar\phi$ on $\wti\cX_2$; knowing that the isomorphism $\wti\cX_1\cong \wti\cX_2$
extends the isomorphism $(\wti\cX_1)_\eta\cong(\wti\cX_2)_\eta$ given by \eqref{iso-eta},
the isomorphism \eqref{iso-eta} gives an isomorphism $(\bar\phi)_\eta\cong (\varphi_2\sta\ti\phi)_\eta$.

Let $\ti p_2: \wti\cX_2\to X$ be the projection.
Since both $\bar\phi$ and $\varphi_2\sta\ti\phi$ are $S$-flat family of quotient sheaves of $\ti p_2\sta\sV$,
and are isomorphic as quotient sheaves over the generic fiber of $\wti\cX_2/S$, by that $\Qu$ is separated,
we conclude $\bar\phi\cong \varphi_2\sta\ti\phi$.
This implies
$$(\varphi_2\sta\ti\phi)_{\eta_0}\cong (\bar\phi)_{\eta_0}\cong (\varphi_1\sta\ti\phi)_{\eta_0}\cong \phi_{\eta_0}
$$
as quotient sheaves on $\xn_0$. In the end, using that $\Sigma_0-q_1$ is a single $\Gm$-orbit, $\ti\phi_a\cong \phi_{\eta_0}$ for all $a\in\Sigma_0-q_1$; the Lemma follows.
\end{proof}

\begin{lemm}\label{5.12}
The sheaf $\wti\sF_{q_1}$ (resp. $\wti\sF_{q_{k+1}}$) is normal to $D_l$ (res. $D_{l+1}$).
Let $\Delta_l\sta=\Delta_l-D_l\cup D_{l+1}$;
the restriction $\ti\phi_{q_1}|_{\Del_l\sta}$ (resp. $\ti\phi_{q_{k+1}}|_{\Del_l\sta}$) is $\Gm$-invariant.
\end{lemm}

\begin{proof}
We prove the case for $\wti\sF_{q_1}$. We consider the $\Theta\sub \wti\cX_{\Sigma_0}$ mentioned before Lemma
\ref{5.11}. Let $\Theta\sta=\Theta-\text{closure}(\wti\cX_{\Sigma_0}-\Theta)$. We let $\text{bl}: \Theta\to \Delta_{l-1}\times\Sigma_0$ be the blowing up morphism, and $g$ be the composite
$$g: \Theta\sta\mapright{\sub} \Theta\mapright{\text{bl}} \Delta_{l-1}\times\Sigma_0\mapright{\text{pr}}\Delta_{l-1}.
$$

Let $p_{l-1}: \Delta_{l-1}\to X$ be the tautological projection, let $\sF\utf_{\eta_0,l-1}$ be
$\sF_{\eta_0}|_{\Del_{l-1}}$ quotient by its subsheaf supported along $D_l\cup D_{l-1}$. By Lemma \ref{tech-1} and Proposition \ref{flat}, $\sF\utf_{\eta_0,l-1}$ is normal to both $D_l$ and $D_{l-1}$.

We consider the quotient on $\Del_{l-1}$ induced by $\phi_{\eta_0}|_{\Delta_{l-1}}$:
$$\phi_{l-1}\utf: p_{l-1}\sta\sV\lra \sF\utf_{\eta_0,l-1}.
$$
We claim
\beq\label{claim-0}
g\sta\phi\utf_{l-1}\cong \ti\phi_{\Sigma_0}|_{\Theta\sta}.
\eeq
First, we know that $\Theta$ is a blowing up of $\Del_{l-1}\times\Sigma_0$ along $D_l\times 0$, and that
$\Gm$-acts on $\Theta$ via the trivial action on $\Del_{l-1}$ and that
on $\Sigma_0$ with the only fixed point $q_1$. Second, we know that $\ti\phi_{\Sigma_0}: \ti p_{\Sigma_0}\sta\sV\to \wti\sF_{\Sigma_0}$
is $\Gm$-equivariant, and for an $a\in \Sigma_0-q_1$, $\ti\phi_a\cong \phi_{\eta_0}$. From these two, we conclude
\beq\label{iso-5}
g\sta\phi\utf_{l-1}|_{\Theta\sta-\wti\cX_{q_1}}\cong \ti\phi_{\Sigma_0}|_{\Theta\sta-\wti\cX_{q_1}}.
\eeq

To conclude the claim, we notice that the isomorphism
\beq\label{iso-6}
g\sta p\sta_{l-1}\sV|_{\Theta\sta-\wti\cX_{q_1}}\cong  \ti p\sta_{\Sigma_0}\sV|_{\Theta\sta-\wti\cX_{q_1}},
\eeq
which is part of the isomorphism \eqref{iso-5}, is the identity map of the pull back of $\sV$ via the tautological projection;
$\Theta\sta-\wti\cX_{q_1}\to X$. Thus
\eqref{iso-6} extends to a
\beq\label{iso-7}
g\sta p\sta_{l-1}\sV|_{\Theta\sta}\cong  \ti p\sta_{\Sigma_0}\sV|_{\Theta\sta}.
\eeq
On the other hand, the family $\ti p\sta_{\Sigma_0}\sV|_{\Theta\sta}\to \wti\sF_{\Sigma_0}|_{\Theta\sta}$ is flat over $\Sigma_0$.
By the uniqueness of flat completion of quotient sheaves, the claim follows if we can show that $g\sta p\sta_{l-1}\sV
\to g\sta \sF_{l-1}\utf$ is flat over $\Sigma_0$.

Since $\sF\utf_{l-1}$ is normal to $D_l$, by Proposition \ref{flat}, $\sF_{l-1}\utf$ is flat along the normal direction of $D_l\sub \Del_{l-1}$.
Thus $g\sta \sF_{l-1}\utf$ is flat along the normal direction of the exceptional divisor of $\Theta\sta\to \Del_{l-1}\times\Sigma_0$.
Applying Proposition \ref{flat}, we conclude that it is flat over $\Sigma_0$, and in addition, $g\sta \sF_{l-1}\utf|_{\Theta\sta\cap \wti\cX_{q_1}}$ is
admissible.

This proves that $\wti\sF_{q_1}$ is normal to $D_l$. It is $\Gm$-equivariant because $\ti\phi_{\Sigma_0}$ is $\Gm$-equivariant.
\end{proof}

\begin{lemm}\label{localexc}
For all $1\le i\le k$, we have
\beq \label{eq-3}\err_l\wti\sF_{q_i}+\err_{l+1}\wti\sF_{q_i}= \err_l\wti\sF_{q_{i+1}}+\err_{l+1}\wti\sF_{q_{i+1}}.
\eeq
Suppose for an $1\le i\le k$, $\err_l\wti\sF_{q_i}\prec\err_{l+1}\wti\sF_{q_i}$ and $\err_l\wti\sF_{q_{i+1}} \succcurlyeq\err_{l+1}\wti\sF_{q_{i+1}}$,
then for $a\in \Sigma_i-\{q_i,q_{i+1}\}$,
\beq \label{ineq-2}\err_l\wti\sF_{q_i}+\err_{l+1}\wti\sF_{q_i}\succ \err_l\wti\sF_{a}+\err_{l+1}\wti\sF_{a} .
\eeq
\end{lemm}

\begin{proof}
{Since $\wti\sF$ is flat over $\Sigma$, we get $\chi(\wti\sF_{q_i}(v))=\chi(\wti\sF_{q_{i+1}}(v))$ for all $1\le i\le k$. Moreover, since $\Gm$ leaves $\Delta_j$ fixed for $j\ne l$, we know the restriction of $\wti\sF$ to
$(X[n+1]_0-\Delta_l)\times\Sigma_i$ is a constant family of sheaves parameterized by $\Sigma_i$ for all $1\le i\le k$.
Therefore, for any $j\ne l,l+1$, the quantities $\exc_j\wti\sF_{a}$ are the same for all $a\in\Sigma_i$.
If we let $\delta^a_{j,i}$ be the quantities associated to sheaf $\wti\sF_a$ defined as $\delta_{l,i}$ in Lemma \ref{delta}, then for $j\ne l$, $\delta^a_{j,0}$ (resp. $\delta^a_{j,1}$) are the same for all $a\in\Sigma_i$.

Applying identity \eqref{the-id} in Lemma \ref{delta}, and subtracting these identical quantities from the right hand side of \eqref{the-id}, we conclude that
\beq\label{the-id2}
\err_l\wti\sF_{a}+\err_{l+1}\wti\sF_{a}+\frac{1}{2}(\delta^a_{l,0}+\delta^a_{l,1})
\eeq
have the same values for all $a\in\Sigma_1\cup\cdots\cup\Sigma_{k}$.

Since by Lemma 5.7 and 5.10, $q_i,q_{i+1}\in \Sigma_i$ are $\Gm$-fixed points of $\Sigma_i$,
and $\Gm$ acts linearly on $\Delta_l$ with fixed locus $D_l\cup D_{l+1}$, we know the restriction of
$\ti\phi_{q_i}$ to $\Delta_l^*$ is $\Gm$-invariant. Moreover, for $1\le l\le k$,
$$\phi_{q_i,l}\utf: p_{l}\sta\sV\lra \sF\utf_{q_i,l}
$$
is the pull back of a quotient sheaf on $D_l$ via the projection $\Delta_l\to D_l$. Applying Lemma \ref{vanishing},
$\delta^{q_i}_{l,0}=\delta^{q_i}_{l,1}=0$ for $1\le l\le k$. For the same reason,
we have $\delta^{q_i}_{l+1,0}=\delta^{q_i}_{l+1,1}=0$ for $2\le i\le k+1$.
The identity \eqref{eq-3} follows from that  \eqref{the-id2} takes same values for $a=q_1,\cdots,q_{k+1}$.

Next we prove \eqref{ineq-2}. By \eqref{the-id2} and the previous argument, for any $a\in\Sigma_i$,
\beq\label{the-id3}\err_l\wti\sF_{q_i}+\err_{l+1}\wti\sF_{q_i}=\err_l\wti\sF_{a}+\err_{l+1}\wti\sF_{a}+
\frac{1}{2}(\delta^a_{l,0}+\delta^a_{l,1})
\eeq
Applying Lemma \ref{A} and \ref{B}, for $a\in \Sigma_i-\{q_i,q_{i+1}\}$, we have
$$\err_l\wti\sF_{q_i}=\err_l\wti\sF_{a} \and \err_{l+1}\wti\sF_{q_{i+1}}=\err_{l+1}\wti\sF_{a}.
$$
Therefore, \eqref{the-id3} gives us
$$\err_{l+1}\wti\sF_{q_i}=\err_{l+1}\wti\sF_{q_{i+1}}+\frac{1}{2}(\delta^a_{l,0}+\delta^a_{l,1}), \quad a\in\Sigma_i-\{q_i,q_{i+1}\}.
$$
Applying \eqref{eq-3}, we also have
$$\err_{l}\wti\sF_{q_{i+1}}=\err_{l}\wti\sF_{q_{i}}+\frac{1}{2}(\delta^a_{l,0}+\delta^a_{l,1}).
$$
Now suppose for a $1\le i\le k$, $\err_l\wti\sF_{q_i}\prec\err_{l+1}\wti\sF_{q_i}$ and
$\err_l\wti\sF_{q_{i+1}} \succcurlyeq\err_{l+1}\wti\sF_{q_{i+1}}$.
Then  $\deg(\delta^a_{l,0}+\delta^a_{l,1})=\deg\err_{l+1}\wti\sF_{q_i}\ge\deg\err_l\wti\sF_{q_i}$.
Therefore, in the identity \eqref{the-id3}, the degree of the left hand side is equal to the degree of the last
term on the right hand side; because of the weak positivity of $\delta^a_{l,0}+\delta^a_{l,1}$ proved in
Lemma \ref{vanishing}, \eqref{ineq-2} follows.
}\end{proof}

\begin{proof}[Proof of Lemma \ref{1.2}]
For any quotient $\phi_{\eta}: p_\eta\sta\sV\to\sF_\eta$ and its extension to an $S$-flat quotient
$\phi: p\sta\sV\to\sF$ such that $\err\sF_{\eta_0}\ne 0$ as stated in the Lemma, according to our construction,
we pick $1\le l\le n$ so that
$$\deg\err_l\sF_{\eta_0}=\deg\err\sF_{\eta_0},
$$
and form a regular $\Gm$-surface $V$, together with a family $\ti p:\wti\cX\to X$ in $\X(V)$ and a
$\Gm$-equivariant quotient $\ti\phi:\ti p^*\sV\to\wti\sF$ on $\wti\cX$.

We further find a connected chain of rational curves $\Sigma=\Sigma_0\cup\cdots\cup\Sigma_{k+1}$ in $V$
so that the restriction of $\ti\phi$ to $\Sigma$ satisfies the properties stated in Lemma \ref{5.11} and \ref{5.12};

According to Lemma \ref{5.12}, we know
$$0=\err_l\wti\sF_{q_1}\prec\err_{l+1}\wti\sF_{q_1}=\err_l\sF_{\eta_0}\ne 0
$$
and $0=\err_{l+1}\wti\sF_{q_{k+1}}\prec\err_{l}\wti\sF_{q_{k+1}}=\err_l\sF_{\eta_0}\ne 0$.
By \eqref{eq-3} in Lemma \ref{localexc}, we can find a $\Sigma_i$, so that the assumptions in Lemma \ref{localexc}
$\err_l\wti\sF_{q_i}\prec\err_{l+1}\wti\sF_{q_i}$ and $\err_l\wti\sF_{q_{i+1}} \succcurlyeq\err_{l+1}\wti\sF_{q_{i+1}}$
are satisfied. For such $i$,
$$ \err_l\wti\sF_{q_i}+\err_{l+1}\wti\sF_{q_i}\succ \err_l\wti\sF_{a}+\err_{l+1}\wti\sF_{a},\quad a\in\Sigma_i-\{q_i,q_{i+1}\}.
$$
Moreover, $\wti\sF_a|_{\Delta_l^*}$ is not $\Gm$-invariant by the non-vanishing of its associated quantity
$\delta^a_{l,0}+\delta^a_{l,1}$ via Lemma 3.18.
By our choice of $l$, we conclude that $\err\wti\sF_{q_i}\succ\err\wti\sF_a$. Combined with $\err\wti\sF_{q_i}=\err\sF_{\eta_0}$,
we have the $\err\wti\sF_a\prec\err\sF_{\eta_0}$, and $\dim \Aut_{\X}(\ti\phi_a)\le \dim \Aut_{\fX}(\phi_{\eta_0})$.

Finally we find the desired curve $S'\sub V$. 
Because $V$ is smooth at the point $a$, and $b:V\to V'$ is a sequence of blow-ups whose exceptional divisor contains $a$,
we can find a smooth curve $S'\sub V'$ that contains $a$, and that the composition $S'\to V'\to S\times \bA^1\to S$ is non constant, thus branched at $\eta'_0=a\in S'$, and $S'\to S$ is finite.
Furthermore, we can take such $S'$ so that its image in $V'$ is not contained in $E'\cup E\sub V'$.
For such an $S'\sub V$, the induced family of quotients $\phi'=\ti\phi_{S'}$ on $p':\cX'=\wti\cX\times_V S'\to X$ satisfies the properties stated in Lemma \ref{1.2}.
\end{proof}

\subsection{The completeness II}
\def\lleta{_{l,\eta}}
\def\lletaz{_{l,\eta_0}}

We complete the proof of Theorem \ref{properA} and \ref{properB} by working out the remainder cases.

Let $(S, \eta, \eta_0)$ be as stated in the beginning of this section. We prove a Lemma analogous to
Lemma \ref{1.2} for $\Qydp$.

\begin{lemm}\label{relative-exist} Let $(\phi_\eta,\cY_\eta)\in\Qydp(\eta)$,
and let $\phi: p\sta \sV_0\to\sF$ be an $S$-flat extension of $\phi_\eta$
over $(\cY,p)\in\fY(S)$.
Suppose $\cY_\eta$ is smooth, $\Aut_\fY\phi_{\eta_0}$ is finite, and $\err\sF_{\eta_0}\ne 0$. Then
we can find a finite base change $S'\to S$, an $S'$-flat quotients $\phi'\colon p'^*\sV_0\to\sF'$ on $(\cY',p')\in\fY(S')$ such that
\begin{enumerate}
\item $\cY'_{\eta'}\cong \cY_\eta\times_{\eta}\eta'\in \fY(\eta')$, and under this isomorphism $\phi'_{\eta'}=\phi_\eta\times_\eta\eta'$;
\item $\Aut_{\Y}(\phi^{\,\prime}_{\eta'_0})$ is finite, and
\item $\exc{\sF'_{\eta'_{0}}}\prec\exc{\sF_{\eta_0}}$.
\end{enumerate}
\end{lemm}

\begin{proof}
We follow the same strategy used to prove Lemma \ref{1.2}.
Since $S$ is local, we can find a $\xi: S\to \bA\unn$ so that $\xi(\eta_0)=0$ and
$\cY\cong \xi\sta Y[\nn]$. Since $\err\sF_{\eta_0}\ne 0$, we pick an $l$ so that
as polynomials,
$$\deg \err_l \sF_{\eta_0}=\deg\err\sF_{\eta_0},\quad -n_--1\le l\le n_++1,\ l\ne 0
$$
Here we agree that $\err_{-n_--1}=\err_-$ and $\err_{n_++1}=\err_+$
(cf. \eqref{err-re}). Without loss of generality, we assume $l>0$.

We let $u$ be a uniformizing parameter of $S$, and express
\beq\label{zeta2}
\xi=\bl c_{-n_-}u^{e_{-n_-}},\ldots, c_{n_+}u^{e_{n_+}}\br, \quad c_i\in \Gamma(\sO_S)\sta.
\eeq
Since $\xi(\eta_0)=0$, all $e_i\ge 0$. We let
$$\tau_{l}:\bA\unn\times\Gm\to\bA^{n_-+n_+'},\quad n_+'=n_++1,
$$
be defined by
\beq\label{GA2}
(t_{-n_-},\cdots,t_{n_+};\sigma)\mapsto(t_{-n_-},\cdots,t_{l-1},\sigma^{-e_l}t_l,\sigma^{e_l},t_{{l}+1},\cdots,t_{n_+}).
\eeq
(In case $l=1$, we replace $t_{l-1}=t_0$ by $t_{-1}$.)
We then introduce
\[ \xi_{l}= \tau_{l}\circ( \xi, \text{id}): S\times\Gm\stackrel{}\lra\bA\unn\times\Gm\stackrel{}\lra\bA^{n_-+n_+'},
\]
and let $\cY':=\xi_{l}\sta Y[\nn']$ over $S\times\Gm$ be the pull back family.
By the construction of $Y[\nn]$, we have $\cY'\cong\xi^* Y[\nn]\times\Gm=\cY\times\Gm$.
We let $p':\cY'\to Y$ and $\pi_1:\cY'\to\cY$ be the projections.

We let $\phi'=\pi_1\sta\phi:p^{\prime\ast}\sV\to\sF'$ be the pullback quotient
sheaves. By the universal property of Grothendieck's Quot-scheme, the family $\phi'$
induces an $\bA^{n_-+n_+'}$-morphism
\beq\label{morphism2}
f_l:S\times\Gm\lra\Quot^{p\sta\sV_0, P}_{Y[\nn']/\bA^{n_-+n_+'}}.
\eeq

Mimic the proof of Lemma \ref{1.2},
We construct a regular $\Gm$-surface $V$ and $\Gm$-morphisms that fit into the following
commutative diagram
\begin{displaymath}
\xymatrix{ V \ar[r]^{\bar\jmath\qquad\qquad} & \Quot^{p\sta\sV_0, P}_{Y[\nn']/\bA^{n_-+n_+'}}\ar[d]^{\pi} \\
 S\times\Gm\ar[u]^{j}\ar[r]^{\xi_l} \ar[ur]^{f_l} & \bA^{n_-+n_+'}}
\end{displaymath}
so that $\pi\circ \bar\jmath:V\to \bA^{n_-+n_+'}$ is proper.

Once we have the surface the pull back family over $V$ from $\bar\jmath$, we can repeat the proof of
Lemma \ref{1.2} line by line to conclude the existence of $S'\sub S$ that satisfies the requirement of
the Lemma. Since the proof is a mere repetition, we omit the details.
This completes the proof.
\end{proof}

\begin{proof}[Proof of Proposition \ref{existence}] We first prove the Proposition for $\Qydp$.
Let $(\phi_\eta:p_\eta\sta\sV_0\to\sF_\eta)\in \Qydp(\eta)$ be a quotient on
$(\cY_\eta,p_\eta)\in\fY(\eta)$.
Then $\cY_\eta=Y[\nn]_0\times \eta$ for some $\nn\ge 0$.
Following the convention \eqref{Y-comp},
$$Y[\nn]_0=\Delta_{-n_-}\cup\ldots\cup \Delta_0\cup\ldots\cup \Delta_{n_+},
$$
where $\Delta_0=Y$.

In the remainder of this proof, we adopt the convention that
$W_l=\Delta_l$ for $-n_-\le l\le n_+$; following the rule specified after  \eqref{ynz}
we endow $W_l$ the relative divisors $E_{l,-}$ and $E_{l,+}$
by the rule: for $l>-n_-$, $E_{l,+}=\Delta_{l-1}\cap \Delta_l$; for $l<n_+$, $E_{l,-}=\Delta_{l}\cap \Delta_{l+1}$;
$E_{-n_-,+}=D[n_-]_{-,0}$ and $E_{n_+,-}=D[n_+]_{+,0}$, where $D[n_-]_{-,0}$
and $D[n_+]_{+,0}$ are the two relative divisors of $Y[\nn]_0$.

We let $W_{l,\eta}=W_l\times\eta\sub\cY_\eta$;
we let $E_{l,\pm,\eta}=E_{l,\pm}\times\eta\sub W_{l,\eta}$, let $p_{l,\eta}: W_{l,\eta}\to X$
be the tautological projection, and let $\Gme=\Gm\times\eta$. We adopt the same convention when
$\eta$ is replaced by $\eta_0$ or $S$.

We consider
\beq\label{Wl}
\phi_{l,\eta}\defeq \phi_\eta|_{W_{l,\eta}}: p\lleta\sta \sV\lra \sF\lleta\defeq \sF_\eta|_{W_{l,\eta}}
\eeq
Since $\phi_\eta$ is stable, $\sF_{l,\eta}$ is normal to
the relative divisors $E_{l,\pm,\eta}$ of $W\lleta$.
Because the Grothendieck's Quot-scheme is proper, we can extend $\phi\lleta$
to an $S$-flat quotient family $\ti\phi_l: p_{l}\sta \sV_0\to \ti\sF_l$ on $W_{l,S}=W_l\times S$.

In the ideal case where all $\ti\sF_{l,\eta_0}=\ti\sF_l|_{W_{l,\eta_0}}$ are normal to $E_{l,\pm,\eta_0}$,
then we will show that we can patch $\ti\phi_l$ to a quotient family $\ti\phi$ on $Y[\nn]_0\times S$ whose
quotient sheaf is admissible. Suppose further that its automorphism group $\Aut_{\fY}(\ti\phi|_{\eta_0})$ is finite, this family
will be the desired family that proves Proposition \ref{existence}.


In general, we divide the proof into several steps.
We first take care of the automorphism groups caused by the $\Gm$-action on $W_l$, $l\ne 0$.
Suppose at least one of $n_-$ and $n_+$ is positive. For any $n_-\le l\ne 0\le n_+$,
suppose $\ti\phi_l|_{W\lletaz}$ is not invariant under the tautological $\Gm$ action on $W\lletaz$ and
$p\lletaz\sta\sV_0$, we do nothing. Suppose it is invariant under $\Gm$.
We claim that $\ti\phi_l|_{W_{l,\eta_0}}$ is not a pull back quotient sheaf from $W\lletaz\to D\times\eta_0$.
Suppose it is a pull back quotient sheaf, then $\ti\sF_l$ is flat over $S$ implies that
$\ti \sF_l$ is a pull back sheaf from $W\times S\to D\times S$; in particular
$\ti\sF_l|_{W\lleta}=\sF\lleta$ is a pull back sheaf from $W\lleta\to D\times \eta$.
But this is impossible since $\phi_\eta$ is stable implies that $\phi\lleta$ is not invariant under the
$\Gm$-action, a contraction.

We continue to suppose $\ti\phi_l|_{W_{l,\eta_0}}$ is $\Gm$-invariant. This invariance together with
that $\ti\phi_l|_{{W}_{l,\eta_0}}$ is not a pull back sheaf from
$D\times\eta_0$ implies that $\ti\sF_l|_{{W}_{l,\eta_0}}$
is not normal to at least one of $E_{l,\pm,\eta_0}\sub W\lletaz$.
Therefore, by repeating the proof of Lemma \ref{1.2},
and possibly after a base change, we can find a $\xi_l: \eta\to\Gme$ so that under
$$\psi_l: {W}_{l,\eta}\mapright{(1,\bar \xi_l)} {W}_{l,\eta}\times_\eta \Gme\mapright{\times}
{W}_{l,\eta}
$$
where $\bar\xi: W\lleta\to\Gme$ is via $W\lleta\mapright{\text{pr}}\eta\mapright{\xi_l} \Gme$, and the second
arrow is the $\Gme$-action on $W\lleta$,
the pull back family $\psi_l\sta(\phi_{l,\eta})$ extends to a new $S$-flat family $\ti\phi_l$ (denoted by the same $\ti\phi_l$)
on ${W}_l\times S$
so that $\ti\phi_l|_{{W}_{l, \eta_0}}$ is not invariant under $\Gm$.

For the modified families $\ti\phi_l$, $n_-\le l\le n_+$, our next step is to modify them so that they are
normal to $E_{l,\pm,\eta_0}\sub W\lletaz$.
We let $\fE_{l,\pm}\sub \fW_l$ over $\fA_\diamond$ be the stack of expanded relative pairs of $E_{l,\pm}\sub W_l$.
(Like $\fD_\pm\sub \fY$ with $D_\pm\sub Y$ replaced by $E_{l,\pm}\sub W_l$.)
Then $\ti\phi_{l,\eta}\in \mathfrak{Quot}_{\fW_l/\fA_\diamond}^{p_l\sta \sV_0}(\eta)$.
In case $\ti\sF_{l,\eta_0}=\ti\sF_l|_{W_{l,\eta_0}}$ is normal to $E_{l,\pm,\eta_0}$,
which is equivalent to $\err\sF_{\eta_0}=0$ by the criteria Lemma \ref{criteria-4},
$\ti\sF_{l,\eta_0}$ is admissible and $\ti\phi_{l,\eta_0}$ is stable.
Otherwise, by Lemma \ref{relative-exist}, we can find a finite base change $S'_l\to S$ and an $S'_l$-flat family of quotients
$\phi'_l$ on $(\cW_l',p'_l)$ so that,
letting $\eta_0'$ and $\eta'$ be the closed and the generic points of $S_l'$,
\begin{enumerate}
\item
$\cW'_{l,\eta'}\cong \cW_{l,\eta}\times_{\eta}\eta'$, that under this isomorphism
$\phi'_{l,\eta'}=\phi_{l,\eta}\times_\eta\eta'$;
\item $\Aut_{\fW_l}(\phi'_{l,\eta_0'})$ is finite;
\item $\err\sF'_{\eta_0'}\prec\err\sF_{\eta_0}$.
\end{enumerate}
If $\err\sF'_{\eta_0'}$ is still nonzero,
we repeat this process. By Lemma \ref{chain} on descending chain, this process terminates at finitely many steps.
Thus we obtain an $S_l'$-family of quotient family
$\phi'_l$ satisfying (1) and (2) above together with (3) replaced by
$\err\sF'_{l,\eta_0'}=0$. Namely,
$\{\phi_l': p_l^{\prime \ast}\sV_0\lra \sF_l' \}\in \mathfrak{Quot}_{\fW_l/\fA_\diamond}^{p_l\sta \sV_0}(S_l')$.

In case $l\ne 0$, we can say more of the symmetry of $\phi_{l,\eta'_0}'$.
When $l\ne 0$, $\cW'_{l,\eta_0'}\cong \Delta\cup\ldots\cup\Delta$, is the union of a chain of,
say $n_l$ copies, of $\Delta$. We define
\beq\label{finite}
\Aut_{\fW_l,\Gm}(\phi_{l,\eta_0'}')=\{ g\in \Gm^{\times n_l}\mid g\cdot (\phi_{l,\eta'_0}')\cong \phi_{l,\eta'_0}'\}.
\eeq
Here $g\cdot (\phi_{l,\eta'_0}')$ is the pull back family of $\phi_{l,\eta'_0}'$ under the $\Gm^{\times n_l}$ action
$\cW_{l,\eta_0'}'\mapright{\cdot g} \cW_{l,\eta_0'}'$, and $g\cdot(\phi_{l,\eta'_0}')\cong \phi_{l,\eta'_0}'$ is
the isomorphism as quotient families, using that $p_l^{\prime\ast} \sV_0|_{\cW_{l,\eta_0'}}$ is
invariant under $\Gm^{\times n_l}$.
It follows from the construction of $\phi_l'$ and the proof of Lemma \ref{relative-exist}
that $\Aut_{\fW_l,\Gm}(\phi_{l,\eta_0'}')$ is finite.

By replacing each $S_l'$ by the fiber product of all $S_l'$ over $S$, we
can assume all $S_l'=S'$ for a single finite base change $S'\to S$.
Let $\eta'$ be the generic point of $S'$.
We now show that we can glue the families $\phi_l'$ to a family $\phi'\in \Qydp(S')$ that
extends $\phi_\eta\times_\eta\eta'$. Let $\cW_l$ over $S'$ be the underlying family of $\phi_l'$.
Since $\phi_l'$ is an extension of $\phi_{l,\eta}\times_\eta\eta'$,
we have $\cW_l\times_{S'}\eta' =W_{l,\eta}\times_\eta\eta'$.
We let $\cE_{l,\pm}\sub\cW_l$ be the closure of $E_{l,\pm,\eta}\times_\eta\eta'\sub W_{l,\eta}\times_\eta\eta'$;
$\cE_{l,\pm}\sub\cW_l$ is the the pair of relative divisors of $\cW_l\in \fW_l(S')$.
Thus, they are smooth divisor in $\cW_l$ and $\cE_{l,\pm}\cong E_{l,\pm}\times S'$ canonically.

We then form the union $\cup_{l=-n_-}^{n_+} \cW_l$; using the canonical isomorphism
$E_{l,-}\cong E_{l+1,+}$, we identify $\cE_{l,-}\sub \cW_l$ with $\cE_{l+1,+}\sub \cW_{l+1}$ for $n_-< l<n_+$,
resulting a family, denoted by $\cY'\to S'$. Let $p': \cY'\to Y$ be the projection induced by
$p_l': \cW_l\to Y$, which exists. In conclusion, our construction of $\fW$ (or $\fY$)
ensures that $(\cY',p')\in \fY(S')$.

We let $\iota_l: \cW_l\to\cY'$ be the tautological closed immersion. We claim that we can
find a quotient family $\phi': p^{\prime\ast} \sV_0\to \sF'$ so that
$\iota_l\sta \phi'\cong \phi'_l$. Indeed, since $p'_l: \cW_l\to Y$ is equal to $p'\circ\iota_l: \cW_l\to \cY'\to Y$,
we have canonical isomorphism $\iota_l\sta p^{\prime\ast} \sV_0\cong p_l^{\prime\ast} \sV_0$.
Hence, using the canonical $p^{\prime\ast}\sV_0\to \iota_{l\ast}\iota_l\sta p^{\prime\ast}\sV_0\cong
\iota_{l\ast} p_l^{\prime\ast} \sV_0$, we obtain quotient sheaf
$p^{\prime\ast}\sV_0\to \iota_{l\ast}\sF_l'$. We now verify that as quotient sheaves
\beq\label{LL}
\bl p^{\prime\ast}\sV_0\to \iota_{l-1\ast} \sF_{l-1}'\br\otimes_{\sO_{\cY'}} \sO_{\iota_{l-1}(\cE_{l-1,+})}
\cong \bl p^{\prime\ast}\sV_0\to \iota_{l\ast} \sF_{l}'\br\otimes_{\sO_{\cY'}} \sO_{\iota_{l}(\cE_{l,-})}.
\eeq
(Note $\iota_{l-1}(\cE_{l-1,+})=\iota_{l}(\cE_{l,-})\sub \cY'$.) First, the above two sides are canonically
isomorphic after restricting to fibers over $\eta'\in S'$; this is true because the two sides of
\eqref{LL} restricted to fiber over $\eta'$ are the quotient $\phi_\eta$ restricted to $E_{l-1,+}\times\eta=
E_{l,-}\times\eta\sub \cY_\eta$. On the other hand, since both $\phi_{l-1}'$ and $\phi_l'$ are families of stable quotients,
by Corollary \ref{flat-2}, both sides of \eqref{LL} are flat over $S'$. Therefore, by the separatedness
of Grothendieck's Quot-scheme, \eqref{LL} holds.
Consequently, the desired quotient family $\phi'$ exists.

Finally, we check that $\phi'$ is a family in $\Qydp(S')$. The fact that $\phi'$ is admissible follows from
Lemma \ref{2-flat}; that $\Aut_{\fY}(\phi'_{\eta_0'})$ is finite follows from that
$\Aut_{\fW_l}(\phi'_{l,\eta_0'})$ is finite for $l=0$ and \eqref{finite} is finite for $l\ne 0$.
This shows that $\phi'\in \Qydp(S')$. This completes the proof of Proposition \ref{existence} for
the stack $\Qydp$.

The proof for $\Qxh$ is exactly the same. In case $\phi\in \Qxh(\eta)$ has its underlying scheme $\cX_\eta$
smooth, then the existence of its extension to an $\phi'\in \Qxh(S')$ for a finite base change $S'\to S$
follows from Lemma \ref{1.1} and \ref{1.2}. In case $\cX_\eta$ is singular, then it is isomorphic to $\xn_0\times\eta$.
Like in the proof of the previous case, we split $\xn_0$ as union of smooth $\Delta_i$ and $Y$; study the extension
problem for the restriction of $\phi$ to $\Delta_i\times\eta$ and $Y\times \eta$, and glue them to form a
desired extension. The proof is exactly the same to the first part of the proof. This proves
the Proposition.
\end{proof}

\subsection{The separatedness}

We show the separatedness part in Theorem \ref{properA} and \ref{properB}. By valuative criteria, this is equivalent to show that the extension of $\phi_\eta$ to $\phi$ constructed in the previous subsections is unique.

We prove Proposition \ref{unique} for smooth generic fibers, the others are the same.
\begin{proof}[Proof of Proposition \ref{unique}]
Let $(\phi_1,\cX_1)\ \text{and}\  (\phi_2,\cX_2)\in \Qxp(S)$ be two families of quotients, where $S$ is as before, such that there is a
$\rho_\eta: \cX_{1,\eta}\to \cX_{2,\eta}$ in $\fX(\eta)$ such that $\phi_{{1,\eta}}=\rho_\eta\sta \phi_{2,\eta}$.

Suppose $\rho_\eta:\cX_{1,\eta}\to\cX_{2,\eta}$ extends to $\rho:\cX_1\to\cX_2$, then $\rho^*\phi_2$ is a family of stable quotient sheaves.
By the separatedness of the Quot-schemes, we have $\rho^*\phi_2\cong \phi_1$. Adding that $(\rho^*\phi_2)_{\eta_0}$ is stable,
we conclude that $\rho:\cX_1\to\cX_2$ is an isomorphism, and the Proposition is done.

Suppose such an extension $\rho$ does not exist. Instead, we will construct $\bar\cX_i\in\X(S)$, and morphisms
$h_i:\bar \cX_i\to\cX_i$ so that $h_{i,\eta}:\bar \cX_{i,\eta}\cong\cX_{i,\eta}$ and the arrow $h^{-1}_{2,\eta}\circ\rho_\eta\circ h_{1,\eta}:\cX_{1,\eta}\to \cX_{2,\eta}$ extends to an arrow $h:\bar\cX_1\cong\bar\cX_2$.

We express $\cX_i$ as $\xi_i\sta X[n_i]$ induced by $\xi_i:S\to C[n_i]$ with $\xi_i(\eta_0)=0$. Let $u$ be a uniformizing parameter of $S$;
we express
$$\pi_{n_i}\circ \xi_i=\bl c_{i,1} u^{e_{i,1}},\ldots, c_{i,n_i+1} u^{e_{i,n_i+1}}\br
$$
as in \eqref{zeta}.
Because $\cX_{1,\eta}=\cX_{2,\eta}\in\fX(\eta)$, we have
\beq\label{nnn}
n:=\sum_{j=1}^{n_1+1} e_{1,j}=\sum_{j=1}^{n_2+1} e_{2,j}.
\eeq
We then define
$\xi_i'$ and $\bar\xi_i: S\to C[n]$ by the rule
\beq\label{prime}\pi_{[n]}\circ \xi_i'=\bl c_{i,1} u^{e_{i,1}},1\ldots,1, c_{i,2} u^{e_{i,2}},1,\ldots,1, c_{i,n_i+1}u^{e_{i,n_i+1}},1,\ldots, 1\br,
\eeq
where after each term $c_{i,j} u^{e_{i,j}}$ we repeat $1$ exactly $e_{i,j}-1$ times, and by
\beq\label{bar}\pi_{[n]}\circ \bar\xi_i=\bl c_{i,1} u,u\ldots,u, c_{i,2} u,u,\ldots,u, c_{i,n_i+1}u,u,\ldots, u\br,
\eeq
where after each term $c_{i,j} u$ we repeat $u$ exactly $e_{i,j}-1$ times.

We let $\cX_i'=\xi_i^{\prime\ast} \xn$ and let $\bar\cX_i=\bar\xi_i\sta\xn$. We describe the relations between these families. First, since \eqref{prime} has the form of the standard embedding defined in \eqref{tau}, the families $\cX_i' \cong\cX_i\in \fX(S)$. Next, we let $\sigma_{i,\eta}: \eta\to \Gm^{{}n}$ be defined via
$$\sigma_{i,\eta}(u)=(u^{{e_{i,1}}-1},u^{{e_{i,1}}-2},\cdots,1,u^{{e_{i,2}}-1},u^{{e_{i,2}}-2},\cdots,1,u^{{e_{i,n_i+1}}-1},u^{{e_{i,n_i+1}}-2},\cdots 1),
$$
then $\xi_i'=(\bar\xi_i)^{\sigma_{i,\eta}}$. Lastly, because $c_{i,j}$ are elements in $\Gamma(\sO_S)^*$, from the expression \eqref{bar}, there is a $\sigma: S\to \Gm^{{}n}$ so that $\bar\xi_1=(\bar\xi_2)^\sigma$, which induces an isomorphism $h:\bar\cX_1\cong\bar\cX_2$.

Moreover, because in the coordinate expression of the morphism $\sigma_{i,\eta}:\eta\to\Gm^{{}n}$, all powers of $u$ are nonnegative, the isomorphisms $\bar\cX_{i,\eta}\cong \cX_{i,\eta}$ induced by $\sigma_{i,\eta}$ and the standard embedding \eqref{prime} extend to morphisms
$h_i: \bar\cX_{i}\to \cX_i$, and the restriction of $h_i$ to $\eta_0$, $h_{i,\eta_0}:  \bar\cX_{i,\eta_0}\to \cX_{i,\eta_0}$, is a contraction of all components $\Delta_j\sub\bar\cX_{i,\eta_0}$ except
$\Delta_0$, $ \Delta_{e_{i,1}},\Delta_{e_{i,1}+e_{i,2}},\ldots,\Delta_{e_{i,1}+\ldots+ e_{i,n_i+1}}$.

We now show that the isomorphism $\phi_{{1,\eta}}=\rho_\eta\sta \phi_{2,\eta}$ extends to $(\phi_1,\cX_1)\cong (\phi_2,\cX_2)$. We first prove $e_{1,j}=e_{2,j}$ for all $j$. Indeed,
using isomorphism $\bar\cX_{i,\eta}\cong \cX_{i,\eta}$ we define
$\bar\phi_{i,\eta}$ be the pull back of $\phi_{i,\eta}$ to $\bar\cX_{i,\eta}$. Let
$\bar\phi_i$ on $\bar\cX_i$ be the $S$-flat completion of $\bar\phi_{i,\eta}$.
Such completion exists since the relative Quot-scheme $\Quot_{\cX_i/S}^P$ is proper over $S$.

Since $\phi_{i,\eta_0}$ is stable, in particular admissible, one checks that the pull back of $\phi_i$ via  $h_i:\bar\cX_i\to\cX_i$ is flat over $S$.
Then by the separatedness of the relative Quot-scheme, $\bar\phi_i=h_i^*\phi_i$.

Then since $\bar\phi_{1,\eta}= h\sta \bar\phi_{2,\eta}$ under the isomorphism $h: \bar\cX_1\to\bar\cX_2$, we must have $\bar\phi_1= h\sta \bar\phi_2$. This implies
$e_{1,1}=e_{2,1}$, $e_{1,1}+e_{1,2}=e_{2,1}+e_{2,2}$, etc. Thus combined with identity \eqref{nnn}, we conclude
$n_1=n_2$ and
$e_{1,j}=e_{2,j}$ for all $j$. This implies that the arrow $\cX_{1,\eta}\cong \cX_{2,\eta}$ in $\fX(\eta)$ extends to
an arrow $\cX_1\cong \cX_2$ in $\fX(S)$. By the separatedness of the Quot-scheme, we get $(\phi_1,\cX_1)\cong
(\phi_2,\cX_2)$ in $\Qxp(S)$. This proves that $\Qxp$ is separated.
\end{proof}

\subsection{For the stable pairs}
We now investigate the properness and separatedness of $\Mxh$ and $\Mydh$. Let $S=\spec R\to C$ with $\eta_0$ and $\eta\in S$
be as in the statement of Proposition \ref{existence}. Let $\phi_\eta:\sO_{\cX_\eta}\to\sF_\eta$ be an element in $\Mxh(\eta)$.
We indicate how to find a finite base change $S'\to S$ and a
$\phi':\sO_{\cX'}\to\sF'$ in $\Mxh(S')$ so that $\phi'\times_{S'}\eta'=\phi_{\eta}\times_{\eta}\eta'$.

By definition, $\phi'\in \Mxh(S')$ if the following hold:
\begin{enumerate}
\item $\sF'$ is a flat $S'$-family of pure one-dimensional sheaves; $\coker \phi'$ has relative dimension at most zero;
\item $\coker \phi'$ is away from the singular divisor of $\cX_{\eta_0}$;
\item $\sF'$ is normal to the singular divisor of $\cX_{\eta_0}$;
\item $\Aut_{\fX}(\phi'_{\eta_0})$ is finite.
\end{enumerate}

Let $\sK_\eta=\coker(\phi_\eta)$ and let $E_\eta\sub\cY_\eta$ be its support. We first study the case where $\cX_\eta$ is smooth. In this case, following the proof in \cite{Li2},
possibly after a finite base change of $S$, which by abuse of notation we still denote by $S$,
we can find an $\cX\in \fX(S)$ that extends $\cX_\eta$ so that
\begin{itemize}
\item[(a)] the closure $\overline E_\eta$ of $E_\eta$ in $\cX$ is disjoint from the singular divisors of $\cX_{\eta_0}$;
\item[(b)] for any added $\Delta\sub \cX_{\eta_0}$, we have $\Delta\cap \overline E_\eta\ne \emptyset$.
\end{itemize}

Since the moduli of stable pairs over a projective scheme is projective (cf. \cite{Le}),
we can extend $\phi_\eta$ to a $\phi:\sO_\cY\to \sF$ that satisfies (1); because of (b), (4) holds as well.
Suppose (2) is violated for the family $\phi$, then by repeating the argument in Subsection \ref{com-I},
we conclude that by a further finite base change, which we still denote by $S$, we can find an
extension $\phi$ of $\phi_\eta$ that satisfies (1), (2) and (4). In case the extension $\phi$ does not satisfies
(3). Then because of (2), $\phi: \sO_{\cX}\to \sF$ near where $\sF$ is not normal to the singular divisor
of $\cX_{\eta_0}$ is a quotient homomorphism. Thus we can apply the result in Subsection \ref{com-I} directly to
conclude that we can find a finite base change $S'\to S$ and an extension $\phi'\in \Mxh(S')$ of $\phi_\eta$ as desired.

The general case for $\Mxh$ and $\Mydh$ is similar to the proof developed in Section 5. Since it is merely a
duplication of the previous argument, we will not repeat it here. This completes the proof of the separetedness and
the properness of Theorem \ref{properA} and \ref{properB}.

\subsection{The boundedness}

We prove the boundedness part in Theorem \ref{properA}, \ref{properB}.

\begin{prop}\label{bdd2}
The set $\fQuot_{\fX_0/\fC_0}^{\sV,P}(\bk)$ is bounded.
\end{prop}

We quote the following known result (c.f. \cite{HL}).

\begin{prop}\label{bound}
A set of isomorphism classes of coherent sheaves on a projective scheme is bounded if and only if the set of their Hilbert polynomials is finite, and there is a coherent sheaf $\sF$ so that every sheaf in this set is a quotient sheaf of $\sF$.
\end{prop}

These two Propositions imply that $\Qxh(\bk)$ is bounded.
We prove Proposition \ref{bdd2} by induction on the degree of the polynomial $P(v)$.
To carry out the induction, we need the following lemma. For simplicity, in the remainder part
of this Section, we assume that $H$ on $X\to C$ is sufficiently ample.

\begin{lemm}\lab{generic-tran}
Let $W$ be either $\xn_0$ or $Y[\nn]_0$, and let $p:W\to X_0$ be the projection.
For any coherent sheaf $\sF$ on $W$, there is an open dense subset $U\sub |p\sta H|$ such
that each divisor $V\in U$ has normal crossing singularity; is smooth away from the singular locus of $W$,
and $\sF$ is normal to $V$. Moreover, if $\sF$ is normal to the singular divisors of $W$
(resp. the distinguished divisor of $Y[\nn]_0$), so does $\sF|_V$, viewed as a sheaf on $W$.
\end{lemm}

\begin{proof}
Given $\sF$, we can find a finite length filtration
$$0\sub \sF_{\le 0}\sub \sF_{\le 1}\sub\cdots\sub\sF_{\le d}=\sF,
$$
where $\sF_{\le k}$ is the subsheaf of $\sF$ consisting of elements of dimension at most $k$.
Let $Z_k$ be the support of $\sF_{\le k}$; it is closed in $W$. Because $H$ is sufficiently ample, $|H|$ is base point free.

We then let
$U\sub |p\sta H|$ be the open subset of those divisors $V\in U$ that have normal crossing singularities;
are smooth away from the singular locus of $W$, and do not contain any irreducible component
of $Z_k$ for all $k$. By Bertini's theorem, $U$ is open and non-empty. For $V\in U$, by Definition \ref{normal}, $\sF$ is normal to
$V$ if and only if no element of $\sF_{\le k}$ is supported entirely in $V$.
Because of the construction, $U$ satisfies the requirement of the lemma.

By the same reason, if $\sF$ is normal, we can choose $U$ so that in addition to the requirement stated,
we have that for every $V\in U$, all $\sF$ and $\sF|_{D_i}$
are normal to $V$.
Therefore, $\sF|_V$ is normal to $D_i$ for all $V\in U$.
\end{proof}

\begin{rema}\lab{Del-tran}
Following the proof of the Lemma, one sees that the set $U$ in the Lemma covers every $D_i\sub \xn_0$ (of $Y[\nn]_0$)
up to finite points in that
$$\dim\bl D_i-\cup_{S\in U} S\cap D_i\br=0.
$$
\end{rema}


We state the following lemma due to Grothendieck \cite{HL}.

\begin{lemm}\label{Gro}
Let $W$ be a projective scheme with an ample line bundle $h$. Let $\sV$ be a fixed coherent sheaf on $W$. Let $\fS$ be the set of those quotients $\phi:\sV\to \sF$ so that $\sF$ is pure of dimension $d$. Suppose there is a constant $N$ so that for any $\sF\in\fS$, its Hilbert polynomial
$$ \chi^h_{\sF}(v)=a_dv^d+a_{d-1}v^{d-1}+\cdots,
$$
satisfies $|a_d|\le N$ and $a_{d-1}\le N$. Then $\fS$ is bounded.
\end{lemm}

Here we use $\chi^h_{\sF}(v)=\chi(\sF\otimes h^{\otimes v})$ to indicate the dependence
on the polarization $h$ of the Hilbert polynomial of $\sF$. Also we use $(\phi,\sF)$ to abbreviate a quotient sheaf $\phi: \sV\to\sF$
when $\sV$ is understood.

\begin{coro}\label{GroCor}
Let $W$ and $\sV$ be as in Lemma \ref{Gro}, and let $N$ and $d\ge 0$ be two integers.
Let $\fS$ be a set of quotient sheaves $\phi:\sV\to \sF$ on $W$. Suppose for any $(\phi,\sF)\in\fS$, every subsheaf of $\sF$ has dimension $\ge d$, and the Hilbert polynomial
$\chi^h_{\sF}(v)=a_mv^m+\cdots+a_0$ satisfies $|a_i|\le N$ for $i\ge d$ and $a_{d-1}\le N$. Then $\fS$ is bounded.
\end{coro}

\begin{proof}
We let $\fS_k=\{(\phi,\sF)\mid \deg \chi_\sF^h\le k,\ (\phi,\sF)\in\fS\}$. We prove that $\fS_k$ are bounded by induction on
$k$.

When $k=d$, every sheaf $\sF$ in the $\fS_d$ is of pure dimension $d$. The result then follows from Lemma \ref{Gro}.
We now suppose the statement is true for a $k\ge d$; we will show that it is true for $k+1$.

For any quotient $(\phi,\sF)\in \fS_{k+1}$, we let $\sF_{\le k}\sub \sF$ be the maximal subsheaf of dimension at most $k$.
Since $\sF$ has dimension at most $k+1$, $\sF_{>k}:=\sF/\sF_{\le k}$ is either zero or is pure of dimension $k+1$.
Also, the quotient homomorphism $\phi:\sV\to\sF$ induces a quotient $\phi_{>k}:\sV\to\sF_{>k}$;
we let $\fS'$ be the set $\{(\phi_{>k},\sF_{>k})\mid (\phi,\sF)\in\fS_{k+1}\}$, and let
$\fT=\{(\phi_{\le k}:\ker\phi_{>k}\to\sF_{\le k})\mid (\phi,\sF)\in\fS_{k+1}\}$, where $\phi_{\le k}$ is induced from
$\phi_{>k}$. Let
$$\fS_{k+1}\lra \fS'\times \fT, \quad \phi\mapsto (\phi_{>k}, \phi_{\le k}),
$$
which is injective.

Since $\sF_{\le k}$ has dimension at most $k$, its Hilbert polynomial $\chi^h_{\sF_{\le k}}(v)=b_kv^k+\cdots$ has $b_k\ge 0$. Since
$$\chi^h_{\sF_{>k}}(v)=\chi^h_{\sF}(v)-\chi^h_{\sF_{\le k}}(v)=a_{k+1}v^{k+1}+(a_k-b_k)v^k+\cdots,
$$
and by assumption $a_{k+1}$ and $a_k$ are bounded, we see that $a_k-b_k$ is bounded from above.
Applying Lemma \ref{Gro}, we see that $\fS'$ is bounded. It also implies that $\{\ker\phi_{>k}\mid (\phi,\sF)\in \fS_{k+1}\}$ is bounded.

Finally, we consider the quotients $(\phi_{\le k}: \ker\phi_{>k}\to \sF_{\le k})\in \fT$. Since $\sF_{\le k}$ has dimension at most $k$,
and since the collection $\{\ker\phi_{>k}\mid (\phi,\sF)\in\fS_{k+1}\}$ is bounded,
we can apply the induction hypothesis to obtain the boundedness of the set $\fT$.
Therefore, $\fS_{k+1}$ is bounded since $\fS_{k+1}\to \fS'\times \fT$ is injective.
\end{proof}

For any polynomial $f(v)$, we denote $[f(v)]_{>0}=f(v)-f(0)$, which is $f(v)$ taking out the constant term.

\begin{lemm}\label{bdd-tech1}
Let $D\subset W$ be a divisor in a smooth variety $W$; $h$ be an ample line bundle on $W$;
$\sU$ be a coherent sheaf on $W$, and $B$ be a finite set of polynomials in $v$. Let $\fS$ be a set of quotients
$\phi:\sU\to\sE$ so that for any $(\phi,\sE)\in\fS$, $\sE$ is normal to $D$ and $[\chi^h_\sE(v)]_{>0}\in B$.
Then $\fS_D=\{(\phi|_D,\sE|_D)\mid (\phi,\sE)\in \fS\}$ is bounded.
Further, suppose $\{\chi(\sE)\mid (\phi,\sE)\in \fS\}$ is finite, then $\fS$ is bounded.
\end{lemm}

\begin{proof}
For $(\phi,\sE)\in \fS$, we denote by $\phi_{>1}:\sU\to\sE_{>1}$ the induced quotient homomorphism. We claim that
$\fS'=\{(\phi_{>1},\sE_{>1})\mid (\phi,\sE)\in \fS\}$ is bounded. Indeed,
since $B$ is finite, there is a constant $M$ so that
for any $(\phi,\sE)\in \fS$, the coefficients of $\chi^h_{\sE}(v)=a_nv^n+\cdots+a_0$
satisfy $|a_i|\le M$ for $i\ge 1$.
Since $\sE_{\le 1}$ has dimension $\le 1$, $\chi^h_{\sE_{\le 1}}(v)=b_1v+b_0$ has $b_1\ge 0$.
Then
$$\chi^h_{\sE_{>1}}(v)=\chi^h_{\sE}(v)-\chi^h_{\sE_{\le 1}}(v)=a_nv^n+\cdots+a_2v^2+(a_1-b_1)v+(a_0-b_0)
$$
has $|a_i|\le M$ for $i\ge 2$ and $a_1-b_1\le M$.
Applying Corollary \ref{GroCor}, we conclude that $\fS'$ is bounded.
Thus, $a_1-b_1$ is bounded; thus by replace $M$ by a larger constant if necessary, we have $|b_1|\le M$.

We now study $\sE|_D$. As $(\phi,\sE)\in\sF$,
$\sE$ is normal to $D$, thus both $\sE_{\le 1}$ and $\sE_{>1}$ are normal to $D$; therefore
\beq\label{exact}
0\lra\sE_{\le 1}|_D\lra\sE|_D\lra\sE_{>1}|_D\lra 0
\eeq
is exact. Since $\fS'$ is bounded, the set $\{(\phi_{>1}|_D,\sE_{>1}|_D) \mid (\phi,\sE)\in\fS\}$ is bounded.
On the other hand, since the leading coefficients $b_1$ of $\chi^h_{\sE_{\le 1}}(v)$ for
$(\phi,\sE)\in\fS$ satisfy $b_1\le M$, using that the set of effective one-dimensional cycles in
$W$ of bounded degree is bounded, we conclude that
the restrictions $\sE_{\le 1}|_D$ form a set of zero dimensional sheaves of bounded length. Therefore, the set
$\{\sE_{\le 1}|_D\mid (\phi,\sE)\in \fS\}$ is bounded. By \eqref{exact}, together with that
$\{(\phi_{>1}|_D,\sE_{>1}|_D) \mid (\phi,\sE)\in\fS\}$ is bounded, we conclude that
$\fS_D=\{(\phi|_D,\sE|_D)\mid (\phi,\sE)\in \fS\}$ is bounded.

Finally, assuming $\{\chi(\sE)\mid (\phi,\sE)\in \fS\}$ is finite, then $B$ finite implies that $\{\chi_\sE^h(v)\mid (\phi,\sE)\in\fS\}$ is
finite. Since $h$ is ample, by Proposition \ref{bound}, we conclude that $\fS$ is bounded.
\end{proof}

Let $p:\Delta\to D$ be the ruled variety over $D$ used to construct $\xn_0$;
let $D_\pm\sub \Delta$ be its two distinguished sections. Denote $h=p\sta (H|_D)$,
where $H$ is sufficiently ample on $X$, we form $L=h(D_+)$, which is ample.
Let $\sV$ be a locally free sheaf on $X$ as before, and we denote $p\sta\sV=p\sta(\sV|_D)$.
Let $\fB$ be a bounded set of sheaves of $\sO_\Delta$-modules, and 
let $B$ be a finite set of polynomials. For $S\in |h|$, we denote by $\iota_S: S\to \Delta$ the embedding.

\begin{lemm}\label{bdd-tech2}
Let $\fR$ be a set of quotients $\phi:p^*\sV\to \sE$ on $\Delta$. Suppose every $\sE\in \fR$ is normal to
$D_+$, $\chi^h_{\sE|_{D_+}}(v)\in B$, and there is a smooth $S\in |h|$ so that $\iota_{S\ast}(\sE|_S)\in \fB$.
Then the set $\{[\chi^L_{\sE}(v)]_{>0}\mid (\phi,\sE)\in\fR\}$ is finite. Moreover, if there is an $N$ so that
$\chi(\sE)\le N$ for all $(\phi,\sE)\in\fR$, then $\fR$ is bounded.
\end{lemm}

\begin{proof}
Let $(\phi,\sE)\in \fR$. By the proof of Lemma \ref{generic-tran}, we can find a smooth $S\in |h|$ so that $\sE$ is normal to $S$. Since $\sE$ is normal to $D_+$, $\sE$ is normal to the divisor $S+D_+$. We can also require that $\sE|_S$ is normal to $D_+$.
Using that $L\cong \sO_{\Delta}(S+D_+)$, we obtain the exact sequence
$$ 0\lra \sE\otimes L^{-1}\lra\sE\lra\sE|_{S+D_+}\lra 0.
$$
It follows that
\beq\label{chi-1}
\chi^L_{\sE|_{S+D_+}}(v)=\chi^L_{\sE}(v)-\chi^L_{\sE}(v-1).
\eeq
Using the exact sequence
$$ 0\lra \sE|_{D_+}(-S\cap D_+)\lra\sE|_{S+D_+}\lra \sE|_{S}\lra 0,$$
and $\iota_{S\ast}(\sE|_S)\in \fB$, which is bounded, and $\chi^h_{\sE|_{D_+}}(v)\in B$, by the standard argument used in
Corollary \ref{GroCor}, the set of quotients $\{p^*\sV\to\sE_{S+D_+}\}$ induced from $(p^*\sV\to \sE)\in \fR$ is bounded.
Therefore, the set of polynomials $\{\chi^L_{\sE|_{S+D_+}}(v)\mid (\phi,\sE)\in\fR\}$ is finite. By \eqref{chi-1},
the set $\{[\chi^L_{\sE}(v)]_{>0}\mid (\phi,\sE)\in\fR\}$ is finite. This proves the first part of the lemma.

Moreover, when $\chi(\sE)\le N$ for all $(\phi,\sE)\in\fR$, Corollary \ref{GroCor} implies that $\fR$ is bounded.
\end{proof}

\begin{lemm}\label{bdd-tech3}
Let $\phi:p^*\sV\to\sE$ be a quotient sheaf on $\Delta$, and $\sE$ is normal to both $D_+$ and $D_-$.
Suppose there is an open subset $U\sub |h|$ such that every $V\in U$ has the following property:
$V$ is smooth; $\sE$ is normal to $V$;
$\dim\bl D_--\cup_{V\in U}D_-\cap V\br=0$, and the restriction $\phi|_V:p^*\sV|_V\to\sE|_V$ is $\Gm$-invariant. Then
$$\chi^h_{\sE|_{D_-}}(v)=\chi^h_{\sE|_{D_+}}(v).$$
\end{lemm}

\begin{proof}
As before, we let $\sE_{\le 1}\sub \sE$ be the subsheaf of elements of dimension at most $1$,
and form the quotient sheaf $\sE_{>1}=\sE/\sE_{\le 1}$. Let $\phi_{>1}:p^*\sV\to\sE_{>1}$
be the induced quotient homomorphism. We claim that the tautological $p^*p_*\sE_{>1}\to\sE_{>1}$ is an isomorphism.

Since $\sE$ is normal to $D_-$, $\sE_{>1}$ is normal to $D_-$. Thus we have
$$0\lra\sE_{>1}(-D_-)\lra\sE_{>1}\lra\sE_{>1}|_{D_-}\lra 0.
$$
Applying $p_*$, we obtain
$$0\lra p_*(\sE_{>1}(-D_-))\lra p_*\sE_{>1}\lra p_*(\sE_{>1}|_{D_-})\lra R^1p_*(\sE_{>1}(-D_-))=0.
$$
Here the last term is zero because $\sE_{>1}$ is a quotient sheaf of $p^*\sV$. We claim that $p_*(\sE_{>1}(-D_-))=0$.
Suppose not, then it is supported on a positive dimensional subset since $\Delta\to D$ has dimension one fibers.
Let $A\sub D$ be an irreducible positive dimensional component of the support of $p_*(\sE_{>1}(-D_-))$.
Because $\dim\bl D_--\cup_{V\in U}D_-\cap V\br=0$,
the union $\cup\{p^{-1}(A)\cap V\mid V\in U\}$ is dense in $p\upmo(A)$.
Therefore, for an open $S\sub D$ such that $S\cap A\ne\emptyset$, we have that
$\sE|_{p^{-1}(S)}\cong p^*(\sE|_{D_-\cap S})$. Thus $p_*(\sE_{>1}(-D_-))|_S=0$,
contradicting to $S\cap A \ne \emptyset$. This proves $p_*(\sE_{>1}(-D_-))=0$; consequently,
$p^*p_*\sE_{>1}\cong\sE_{>1}$, and
$$ \chi^h_{\sE_{>1}|_{D_-}}(v)=\chi^h_{p_*\sE_{>1}}(v)=\chi^h_{\sE_{> 1}|_{D_+}}(v).
$$

Repeating the same argument, we conclude that $\sE_{\le 1}$ is supported at finite fibers of $p:\Del\to D$.
Since $\sE$ is normal to $D_-$ and $D_+$, $\sE_{\le 1}$ is normal to $D_-$ and $D_+$ too.
Thus $\chi(\sE_{\le 1}|_{D_-})=\chi(\sE_{\le 1}|_{D_+})$. Therefore
$$\chi^h_{\sE|_{D_-}}(v)=\chi^h_{\sE_{>1}|_{D_-}}(v)+\chi(\sE_{\le 1}|_{D_-})
=\chi^h_{\sE_{>1}|_{D_+}}(v)+\chi(\sE_{\le 1}|_{D_+})=\chi^h_{\sE|_{D_+}}(v).
$$
This proves the Lemma.
\end{proof}

In the remainder of this Section, we abbreviate $\Q_P:=\fQuot_{\fX_0/\fC_0}^{\sV,P}(\bk)$.

\begin{proof}[Proof of Proposition \ref{bdd2}]
We prove that $\Q_P$ is bounded by induction on the degree of the polynomial $P(v)$.

Suppose $P(v)=c$ is a constant. Let $(\phi,\sF,\xn_0)\in\Q_P$. Then $\sF$ is a zero dimensional sheaf such that its
support is away from the singular locus of $\xn_0$ and its length is $c$.
The stability of $\sF$ implies that $\sF|_{\Delta_i}$ is nonzero for every $1\le i\le n$.
Therefore, $n\le \mathrm{length}(\sF)=c$. Applying Proposition \ref{bound}, we conclude that $\Q_P$ is bounded in this case.

Next we assume that for an integer $d$, $\Q_P$ is bounded when $\deg P(v)\le d-1$.
We show that $\Q_P$ is bounded when $P(v)$ has degree $d$.

Let $P$ be a degree $d$ polynomial, and let $(\phi, \sF,X[n]_0)\in \Q_P$.
By Lemma \ref{generic-tran}, we can find an $S\in |p^*H|$ so that it has normal crossing singularity; is smooth away from the singular locus of $\xn_0$; that $\sF$ is normal to $S$, and the restriction $\sF|_S$ is normal to the singular divisor of $S$.

Let $\sF'=\iota_{S*}(\sF|_{S})$ and $\phi':p^*\sV\to \sF'$ be the quotient homomorphism induced by $\phi$.
We have $\chi^{p\sta H}_{\sF'}(v)=P(v)-P(v-1)$. By our choice of $S$, $\sF'$ is admissible but not necessary stable.
We let $\Lambda_{\phi}\sub\{1,\cdots,n\}$ be the subset of indices $k$ so that $\phi'|_{\Delta_k}$ is {\sl not} $\Gm$-invariant;
we let $n_\phi=\#\Lambda_{\phi}\ge 0$, and let
$$I_\phi: \{1,\ldots, n_\phi\}\to \Lam_\phi
$$
be the order-preserving isomorphism.
Let $\Lambda_{\phi}^\complement$ be the complement of $\Lambda_{\phi}$. We then contract all $\Delta_i\sub\xn_0$,
$i\in\Lambda_\phi^\complement$,
to obtain $p_\phi:\xn_0\to X[n_\phi]_0$. Let $p':X[n_\phi]_0\to X_0$ be the projection.
Since $\phi'$ is admissible, and $\phi'|_{\Delta_i}$ is $\Gm$-invariant for $i\in\Lambda_{\phi}^\complement$,
there is a quotient
$$(\phi')^{\st}:p'^*\sV\lra \sF'^{\st} \quad \text{such that}\quad  \phi'=p_\phi^*(\phi')^{\st}.
$$
Then $\bl (\phi')^{\st},\sF'^{\st}, X[n_\phi]_0\br \in\Q_{P_1}$, where $P_1(v)=P(v)-P(v-1)$.
By the induction hypothesis, $\Q_{P_1}$ is bounded. Therefore, there is an $N$ depending on $P$ only so that
\beq\label{NN}
n_\phi\le N.
\eeq

To proceed, we let $p_{\Del}:\Delta\to D$ be the ruled variety used to construct $\xn_0$ with distinguished sections
$D_\pm\sub \Delta$. Let $h=p_{\Del}\sta (H|_D)$, where $H$ is sufficiently ample on $X$
(using $H^{\otimes m}$ if necessary), and form $L=h(D_+)$, which is ample.
Let $H_i=p^*H|_{\Del_i}$; and let $L_i=H_i(D_i)$, $i>0$. We fix the tautological isomorphisms
\beq\label{Del-isom}
\rho_i:\Del\cong \Del_i,\quad \text{so that }h=\rho_i^*H_i,\quad L=\rho_i^*L_i,
\eeq
for all intermediate components $\Del_1,\cdots,\Del_n$ of $\xn_0$.

\vsp

\noindent
{\bf Sublemma 1}. {\sl The set $\{\chi^{H_i}_{\sF|_{D_i}}(v) \mid (\phi,\sF,\xn_0)\in\Q_P,\,  i\le n+1\}$ is finite.
}
\begin{proof}
Let $N$ be as specified in \eqref{NN}.
We first construct a finite sequence of finite sets $B_1,B_2,\cdots,B_{N+1}$ and show that for any $(\phi, \sF,
\xn_0)\in \Q_P$, and any $1\le i\le n+1$, we have
$\chi^{H_i}_{\sF|_{D_i}}(v)\in B_k$ for some $k$.
This will prove the Sublemma.

Let $B_1=\{\chi_{\sF|_{D_1}}^{H_0}(v)\mid (\phi,\sF, \xn_0)\in \Q_P\}$.
We prove that $B_1$ is a finite set. Indeed, by induction, we can find $S\in |p^*H|$ so that
$\sF'=\iota_{S*}(\sF|_S)$ is admissible, and $\chi^{p\sta H}_{\sF'}(v)=P(v)-P(v-1)$.
Restricting to $\Del_0=Y$, since $(\sF')^{\st}|_{\Del_0}=\sF'|_{\Del_0}$, the induction hypothesis that $\Q_{P_1}$ is
bounded implies that $\{\chi^{H_0}_{\sF'|_{\Del_0}}(v)\mid ((\phi')^{st},(\sF')^{\st},X[n_\phi]_0)\in\Q_{P_1}\}$ is finite.
Therefore,
\beq\label{Del0}
\{[\chi^{H_0}_{\sF|_{\Del_0}}(v)]_{>0}\mid (\phi,\sF, \xn_0)\in\Q_P\}\quad \text{ is finite}.
\eeq
Since $H_0$ is ample on $\Delta_0$, using Lemma \ref{bdd-tech1}, we know that
$\{\sF|_{D_1}\mid (\phi,\sF, \xn_0)\in \Q_P\}$ is bounded.
Therefore, $B_1$ is finite.

We define $B_{i\ge 2}$ inductively. Suppose we have defined $B_k$. Using the isomorphisms \eqref{Del-isom}, we define
a set of quotient homomorphisms on $\Del$:
$$\fR_k=\cup_{i\ge 1}\{\rho_i^*(\phi|_{\Delta_i})\mid (\phi,\sF,\xn_0)\in\Q_P, \ \chi^h_{\rho_i^*(\sF|_{D_i})}(v)\in B_k\}.
$$
(Recall that $D_i\sub \Delta_i$ is identified with $D_+\sub \Delta$ under $\rho_i$ (cf. \eqref{xn0}).)
We apply the first assertion of Lemma \ref{bdd-tech2} to $B=B_k$ and
$\fB=\cup_{i\ge 1}\{\rho_i\sta(\phi|_{\Delta_i})\mid (\phi,\sF,\xn_0)\in\Q_{P_1}\}$
to conclude that the set
$\{[\chi^L_{\sF}(v)]_{>0}\mid (\phi, \sF)\in \fR_k\}$ is finite. Then applying Lemma \ref{bdd-tech1} to $D_-\subset \Del$,
we conclude that $\{(\phi,\sF)|_{D_-}\mid (\phi,\sF)\in\fR_k\}$ is bounded, which implies that
$B_{k+1}=\{\chi^H_{\sF|_{D_-}}(v)\mid (\phi,\sF)\in\fR_k\}$
is finite. 

For any $(\phi,\sF,\xn_0)\in\Q_P$ and $1\le i\le n+1$, we claim that $\chi^{H_i}_{\sF|_{D_i}}(v)\in B_k$ for some $k\le N+1$.
To show this, we consider the sequence of polynomials
\beq\label{seqbi} \chi^{H_1}_{\sF|_{D_1}}(v),\cdots,\chi^{H_{n+1}}_{\sF|_{D_{n+1}}}(v).
\eeq
By Lemma \ref{bdd-tech3}, for $i\in\Lambda_\phi^\complement$, $\chi^{H_i}_{\sF|_{D_i}}(v)=\chi^{H_{i+1}}_{\sF|_{D_{i+1}}}(v)$;
for $i=I_\phi(k)\in \Lambda_\phi$ for some $k$,
$\chi^{H_{i+1}}_{\sF|_{D_{i+1}}}(v)\in B_{k+1}$. Since $\# \Lam_\phi\le N$, we have
$\chi^{H_{i}}_{\sF|_{D_{i}}}(v)\in \cup_{k=1}^{N+1}B_k$. Since each $B_k$ is finite, the Sublemma follows.
\end{proof}


\noindent
{\bf Sublemma 2}. {\sl There is a constant $M>0$ so that for any $(\phi,\sF,\xn_0)\in\Q_P$, then
\begin{enumerate}
\item for $i\in\Lambda_\phi^\complement$, we have $\chi(\sF|_{\Del_i}(-D_{i}))\geq 1$;
\item $\chi(\sF|_{\Del_0}(-D_{n+1}))\geq -M$.
\item for $i=I_\phi(k)\in \Lambda_\phi$, $\chi(\sF|_{\Del_i}(-D_{i}))\geq -M$. 
\end{enumerate}}

\begin{proof}
We first prove item (1).
Let $(\phi,\sF,\xn_0)\in\Q_P$ and let $i\in\Lambda_\phi^\complement$.
We let $S\in |p\sta H|$ and $\phi': p\sta\sV\to \sF'=\iota_{S\ast}(\sF|_S)$ be as the quotient sheaf
constructed at the beginning of the this proof (of Proposition \ref{bdd2}).
By the construction of $\Lambda_\phi^\complement$,
we know that the restriction (to $\Delta_i$) $(\phi'|_{\Delta_i},\sF'|_{\Delta_i})$ is $\Gm$-invariant.
By Lemma \ref{vanishing},
$\chi^{H_i}_{\sF^{\prime}|_{\Del_i}}(v)-\chi^{H_i}_{\sF^{\prime}|_{D_i}}(v)=0$.
Since
$$\chi^{H_i}_{\sF^{\prime}|_{\Del_i}}(v)=\chi^{H_i}_{\sF|_{\Del_i}}(v)-\chi^{H_i}_{\sF|_{\Del_i}}(v-1)\and
\chi^{H_i}_{\sF^{\prime}|_{D_i}}(v)=\chi^{H_i}_{\sF|_{D_i}}(v)-\chi^{H_i}_{\sF|_{D_i}}(v-1),
$$
the polynomial $f(v)=\chi^{H_i}_{\sF|_{\Del_i}}(v)-\chi^{H_i}_{\sF|_{D_i}}(v)$
then satisfies $f(v)=f(v-1)$, which makes it a
constant equal to $\chi(\sF|_{\Del_i}(-D_i))$. Since $\sF|_{\Del_i}$ is not $\Gm$-invariant, by Lemma \ref{vanishing},
$\chi(\sF|_{\Del_i}(-D_i))\ge 1$.

We now prove item (2). Suppose the lower bound does not exist. Then there is a sequence $(\phi_k, \sF_k, X[n_k]_0)\in \Q_P$
\beq\label{toinfty1}\chi(\sF_{k}|_{\Del_{0}}(-D_{n_k+1}))\to -\infty,\quad \text{when } k\to +\infty.
\eeq
But by \eqref{Del0} and Corollary \ref{GroCor}, we know that $\{\sF_k|_{\Del_0}\}_{k\ge 1}$ is bounded;
contradicts to \eqref{toinfty1}. Thus item (2) holds.

Suppose item (3) does not hold, then there is a sequence $(\phi_k, \sF_k, X[n_k]_0)\in \Q_P$
and a sequence $1\le i_k\le n_k$ such that
\beq\label{toinfty}\chi(\sF_{k}|_{\Del_{i_k}}(-D_{i_k}))\to -\infty,\quad \text{when } k\to +\infty.
\eeq
Using isomorphisms \eqref{Del-isom}, we introduce
$\bar\sF_k=\rho_{{i_k}}^*(\sF_{k}|_{\Del_{i_k}})$. Tensoring $\bar\sF_k$ with $\sO_\Del(-D_+)$, we obtain a sequence of quotients
$\bar\phi_k:\bar\sV(-D_+)\to \bar\sF_k(-D_+)$, where $\bar\sV=p_{\Del}^*\sV|_D$, $p_{\Del}:\Del\to D$.
By construction, $\chi(\bar\sF_k(-D_+))\to-\infty$. In particular $\chi(\bar\sF_{k}(-D_{+}))$ is bounded from above.

We claim that the set of polynomials $\{[\chi^L_{\bar\sF_{k}(-D_{+})}(v)]_{>0}\}_{k\ge 1}$ is finite.
Once this is proved, then applying Corollary \ref{GroCor} we conclude
that $\{\bar\phi_k\}_{k\ge 1}$ is bounded, which contradicts to $\chi(\bar\sF_k(-D_+))\to-\infty$.

We prove the claim. By Sublemma\! 1, there is a finite set $B$ so that $\chi^{H_k}_{\sF_k|_{D_{i_k}}}(v)\in B$.
Using isomorphism \eqref{Del-isom}, we obtain $\chi^{H}_{\bar\sF_k|_{D_{+}}}(v)\in B$. Applying the first assertion of
Lemma \ref{bdd-tech2}, we conclude that  $\{[\chi^L_{\bar\sF_k}(v)]_{>0}\}_{k\ge 1}$ is finite.
Restricting to $D_+$, Lemma \ref{bdd-tech1} implies that  $\{\chi^L_{\bar\sF_k|_{D_+}}(v)\}_{k\ge 1}$ is finite. The claim
then follows from
$[\chi^L_{\bar\sF_k(-D_{+})}(v)]_{>0}=[\chi^L_{\bar\sF_k}(v)]_{>0}-[\chi^L_{\bar\sF_k|_{D_+}}(v)]_{>0}$.
\end{proof}

We now complete the proof of Proposition \ref{bdd2}. Let $(\phi, \sF, \xn_0)\in \Q_P$.
Since $\sF$ is normal to all $D_i$,
\beq\label{chi-id}\chi(\sF)=\chi(\sF|_{\Del_0}(-D_{n+1}))+\chi(\sF|_{\Del_1}(-D_1))+\cdots+\chi(\sF|_{\Del_{n}}(-D_{n})).
\eeq
For $i\in \Lambda_\phi^\complement$, we have $\chi(\sF|_{\Delta_i}(-D_i))\ge1$;
for $i\in\Lambda_\phi\cup\{0\}$, by Sublemma\! 2, we have $\chi(\sF|_{\Delta_i}(-D_i))\ge -M$ ($D_0=D_{n+1}$).
Since $n_\phi\le N$, we obtain
$\chi(\sF)\geq (N+1)(-M)+(n-\# \Lambda_\phi)$, which implies
\beq\label{nn}n\leq  \chi(\sF)+(N+1)M+N.
\eeq
The identity \eqref{chi-id} and Sublemma 2 also gives the bound,
$$\chi(\sF|_{\Delta_i}(-D_i)) \le \chi(\sF)+(N+1)M+N, \quad 0\le i\le n.
$$
Therefore, applying Lemma \ref{bdd-tech1} and \ref{bdd-tech2}, we conclude that
for each $i$, the set $\{\sF|_{\Delta_i}\mid (\phi,\sF, \xn_0)\in \Q_P\}$ is bounded.
This together with the bound \eqref{nn} implies that $\Q_P$ is bounded.
\end{proof}

By a parallel argument, we have

\begin{prop}
The set $\Qydp(\kk)$ is bounded.
\end{prop}

\subsection{The moduli of stable pairs}

We prove the boundedness of the moduli $\Mxh$ and $\Mydh$. Here $P(v)$ is a degree one polynomial.

\begin{prop}
The set $\Mxh(\kk)$ and $\Mydh(\kk)$ are bounded.
\end{prop}

\begin{proof}
We work with the case $\Mxh(\kk)$. The other is the same. Let $P(v)=av+b$. Let $(\varphi,\sF,\xn_0)\in\Mxh(\kk)$, let
$\sF_i=\sF|_{\Delta_i}$ and $H_i=p^*H|_{\Del_i}$. Then each $\chi^{H_i}_{\sF_i}(v)=a_iv+b_i$ has $a_i\ge 0$, and
\beq\label{asum} a=a_0+a_1+\cdots+a_{n}.
\eeq
Let $\Lambda_\varphi$ be the set of those $k\ge 1$ so that $\chi^{H_k}_{\sF_k}(v)$ has positive degree.
Then by \eqref{asum}, $n_\varphi=\#\Lambda_\varphi\le a$.
Let $\Lambda_\varphi^\complement=\{1,\cdots,n\}-\Lambda_\varphi$.

First, we show that for each $i\in\Lambda_\varphi^\complement$, $\chi(\sF_{i}(-D_i))\ge 1$.
Let $\varphi_i:\sO_{\Del_i}\to\sF_i$ be the restriction of $\varphi$ to $\Delta_i$. Since $\coker\varphi$ has zero dimensional support,
$\chi(\coker\varphi)\ge 0$. Hence $\chi(\sF_i(-D_i))\ge \chi(\im \varphi_i(-D_i))$.

For $\im \varphi_i$, we have the induced quotient homomorphism
$\varphi_i':\sO_{\Del_i}\to\im \varphi_i$. Applying Lemma \ref{vanishing} to $\varphi_i'$,
we get $\chi(\im \varphi_i(-D_i))\ge 0$. Since $\varphi_i$ is not $\Gm$-invariant,
either $\chi(\coker \varphi)>0$ or $\chi(\im \varphi_i(-D_i))>0$. Thus
$\chi(\sF_{i}(-D_i))\ge 1$.

Next, we let $I_\varphi:\{1,\cdots,n_\varphi\}\to\Lambda_\varphi$ be the order-preserving isomorphism.
We form
$$\Xi_k=\{\chi(\sF_{j}(-D_{j}))\mid (\varphi,\sF,\xn_0)\in\Mxh(\kk), \, j=I_\varphi(k) \}.
$$
(For $k=0$, we agree $I_{\varphi}(0)=0$ and $D_0=D_{n+1}$.)
Applying the same argument as in Sublemma~2 of the proof of Proposition \ref{bdd2} to $\varphi'_k$,
we conclude that there is an $M>0$ so that for each $k\le a$, $\inf\{\chi\in\Xi_k\}\ge -M$.
(Note that by the bound $\#\Lambda_\varphi\le a$, $\Xi_k=\emptyset$ if $k>a$.)


Lastly, since $\sF$ is normal to $D_i$, we have
\beq\label{chi-idd}\chi(\sF)=\chi(\sF_{0}(-D_{n+1}))+\chi(\sF_{1}(-D_1))+\cdots+\chi(\sF_{{n}}(-D_{n})).
\eeq
Repeating the argument following \eqref{chi-id}, we prove the
boundedness of $\Mxh(\kk)$.
\end{proof}

\subsection{Decomposition of the central fiber}

In this subsection, we assume that $Y$ is a disjoint union of two smooth components $Y_-$ and $Y_+$.
We introduce a canonical decomposition of the central fiber of the moduli stacks $\Qxp$ and $\Mxh$ over $C$.
We shall focus on $\Qxp$ and omit the details for $\Mxh$.

Let
$$\fQuot_{\fX_0/\fC_0}^{\sV,P}=\Qxp\times_C 0
$$
be the central fiber of $\Qxp$ over $C$. We denote $\fC^P$ be the weighted stack of weights in $\Lam=\QQ[m]$
(polynomials in $m$)
and of total weight $P$ (c.f. Section 2.5). For each stable quotient $\phi\colon p^*\sV\to\sF$ in $\Qxp(\kk)$,
where $\sF$ is a sheaf on ${\xn_0}$, it assigns a weight $w$ to $\xn_0$ by assigning each irreducible
$\Del_l\sub\xn_0$ (resp. divisor $D_l\sub \xn_0$) the polynomial $\chi^H_{{\sF_{\Delta_l}}}$ (resp. $\chi^H_{{\sF_{D_l}}}$).
Since $\sF$ is admissible, this rule applied to $(\phi, \cX)\in \Qxp(S)$ defines a continuous weight assignment of the family
$\cX/S$. In particular, the morphism $\Qxp\to\fC$ factors through
\beq\label{pip}
\pi_P: \Qxp\lra \fC^P.
\eeq

We now form the set of splittings of $P$: $\Lam_P^{\text{spl}}$, which is the set of triples $\delta=(\delta_\pm,\delta_0)$ in $\Lam$ so that $\delta_-+\delta_+-\delta_0=P$.
We follow the notation developed in Subsection \ref{subsec2.5}.
For any $\delta\in \Lam_P^{\text{spl}}$, we form the moduli of stable relative quotients on $\fD_\pm\sub\fY_\pm$ over $\fA_\diamond$:
for any scheme $S$, we define $\fQuot_{\fY_-/\fA_\diamond}^{\delta_-,\delta_0}(S)$ be the
collection of $(\phi; \cY,\cD)$,
where $(\cY,\cD)\in \fY_-(S)$ and $\phi\colon p^*\sV\to\sF$ is an $S$-flat family of stable relative
quotients on the pair $\cD\sub\cY$ such that for any closed $s\in S$,
$\chi^H_{{\sF_s}}=\delta_-$ and $\chi^H_{{{\sF_s}}|_{\cD_s}}=\delta_0$.
We form $\fQuot_{\fY_+/\fA_\diamond}^{\delta_+,\delta_0}$ similarly. By Theorem \ref{properB}, we have

\begin{prop}
The groupoids $\fQuot_{\fY_\pm/\fA_\diamond}^{\delta_\pm,\delta_0}$ are Deligne-Mumford stacks,
proper and separated, and of finite type.
\end{prop}

Using $\delta\in\Lam_P^{\text{spl}}$, we form the stack $\Cdel$, according to the rule specified in Section $2$. We define
$$\fQuot_{\fX_0^\dagger/\fC_0^\dagger}^{\delta}=\Qxp\times_{\fC^P}\Cdel.
$$
It parameterizes stable quotients $\phi\colon p^*\sV\to\sF$ on $\xn_0$ with a node-marking $D_k\sub \xn_0$
so that the Hilbert polynomials of $\sF$ restricted to $\cup_{i<k}\Delta_i$, to $\cup_{i\ge k}\Delta_i$ and to
$D_k$ are $\delta_-$, $\delta_+$ and $\delta_0$, respectively.

For each  $\delta\in\Lam_P^{\text{spl}}$, like the case of stable morphisms, we have the gluing morphism that factors through $\fQuot_{\fX_0^\dagger/\fA_0^\dagger}^{\delta}$ (it originally maps to $\Qxp\times_C 0$):
\beq\label{glue-H}
\Phi_{\delta}: \fQuot_{\fY_-/\fA_\diamond}^{\delta_-,\delta_0}\times_{\Quot_D^{\sV_D,\delta_0}} \fQuot_{\fY_+/\fA_\diamond}^{\delta_+,\delta_0}\lra
\fQuot_{\fX_0\shar/\fA_0\shar}^{\delta},
\eeq
where $\Quot_D^{\sV_D,\delta_0}$ is the Grothendieck's Quot-scheme of quotient sheaves $\sV_D=\sV|_D\to \sE$ with $\chi^H_{\sE}(v)=\delta_0$.

Using the collection of pairs of line bundles and sections $(L_\delta,s_\delta)$ for $\delta\in \Lam_P^{\text{spl}}$
constructed in Proposition \ref{prop2.24}, and let $\pi_P$ be as in \eqref{pip}, we have

\begin{theo}\label{3.36}
Let $(L_\delta,s_\delta)$ 
and the notation be as in Proposition \ref{prop2.24}. Then
\begin{enumerate}
\item $\otimes_{\delta\in \Lam_P^{\rm{spl}}} \, \pi_P\sta L_\delta\cong \sO_{\Qxp}$,
and $\prod_{\delta\in \Lam_P^{\rm{spl}}} \pi_P \sta s_\delta =\pi_P \sta \pi\sta t$;
\item as closed substacks, $\fQuot_{\fX_0^\dagger/\fA_0^\dagger}^\delta= (\pi_P \sta s_\delta=0)$;
\item The morphism $\Phi_\delta$ in \eqref{glue-H} is an isomorphism of Deligne-Mumford stacks.
\end{enumerate}
\end{theo}

For the case of coherent systems, like Quot-schemes, the morphism $\P^P_{\X/\C}\to\fC$ factors through
\beq\label{pip2}
\pi_P: \P^P_{\X/\C}\lra \fC^P.
\eeq
For any $\delta\in \Lam_P^{\text{spl}}$, we define the moduli of relative stable pairs on $\fD_\pm\sub\fY_\pm$ over $\fA_\diamond$:
$$\P_{\fY_-/\fA_\diamond}^{\delta_-,\delta_0}\and \P_{\fY_+/\fA_\diamond}^{\delta_+,\delta_0}.
$$
They are again Deligne-Mumford stacks, proper and separated, and of finite type; and they both admit an evaluation morphism to the Hilbert scheme $\Hilb_D^{\delta_0}$ via restriction.

Accordingly, for $\delta\in \Lam_P^{\text{spl}}$, we define
$$\P_{\fX_0^\dagger/\fC_0^\dagger}^{\delta}=\Mxh\times_{\fC^P}\Cdel.
$$
We have a glueing morphism
\beq\label{glue-P}
\Phi_{\delta}: \P_{\fY_-/\fA_\diamond}^{\delta_-,\delta_0}\times_{\Hilb_D^{\delta_0}} \P_{\fY_+/\fA_\diamond}^{\delta_+,\delta_0}\lra
\P_{\fX_0\shar/\fC_0\shar}^{\delta}.
\eeq

\begin{theo}\label{3.36pri}
Let $(L_\delta,s_\delta)$ and the notation be as in Proposition \ref{prop2.24}. Then
\begin{enumerate}
\item $\otimes_{\delta\in \Lam_P^{\rm{spl}}} \, \pi_P\sta L_\delta\cong \sO_{\Mxh}$,
and $\prod_{\delta\in \Lam_P^{\rm{spl}}} \pi_P \sta s_\delta =\pi_P \sta \pi\sta t$;
\item as closed substacks, $\P_{\fX_0^\dagger/\fC_0^\dagger}^\delta= (\pi_P \sta s_\delta=0)$;
\item The morphism $\Phi_\delta$ in \eqref{glue-P} is an isomorphism of Deligne-Mumford stacks.
\end{enumerate}
\end{theo}

\section{Virtual cycles and their degenerations}

Let $\pi\colon X\to C$ and $H$ ample on $X$ be a simple degeneration of projective threefolds.
We fix a degree one polynomial $P(v)$. 
Applying Theorem~\ref{properA}, we form the good degeneration $\Ixp:=\mathfrak{Quot}_{\X/\C}^{\sO_X,P}$ of Hilbert scheme of subschemes of
$X/C$, of Hilbert polynomial $P$.

In this section, we construct the virtual class of $\Ixp$, and use this class to prove a degeneration formula
of the Donaldson-Thomas invariants of ideal sheaves. For notational simplicity, we only treat the case where
the central fiber $X_0$ is the union of two irreducible components and their intersection $D\sub X_0$ is connected.
Our construction of perfect relative obstruction theory of $\Ixp\to\fC^P$ is based on the work of
Huybrechts-Thomas on Atiyah class \cite{HT}; our proof of degeneration formula follows the proof of a similar degeneration formula by Maulik, Pandharipande and Thomas in \cite{MPT}; the formulation of degeneration based on Chern characters
follows the work of Maulik, Nekrasov, Okounkov and Pandharipande in \cite{MNOP2}.

As $X_0$ is assumed to have two irreducible components, the normalization $q: Y\to X_0$ has two connected components
$$Y=Y_-\cup Y_+,\and D_\pm=Y_\pm\cap q\upmo(D).
$$

\subsection{Virtual cycle of the total space}

We first construct the relative obstruction theory of $\Ixp\to\fC^P$ (c.f. \eqref{xcbeta}). We let
$$\pi: \cX=\fX\times_{\fC}\Ixp\lra \Ixp, \and \sI_{\cZ}\sub\sO_{\cX}
$$
be the universal underlying family and the universal ideal sheaf of $\Ixp$. 
We form the traceless part of the derived homomorphism of sheaves of $\sO_{\cX}$-modules:
\beq\label{der-hom}
E=R\pi\lsta \RHom(\sI_{\cZ},\sI_{\cZ})_0[1].
\eeq
Since $\cX\to\Ixp$ is a family of l.c.i. schemes, and $\sI_{\cZ}$ is admissible and of rank one,
by Serre duality, locally $E$ is a two-term perfect complex concentrated at $[0,1]$.

Let
$$L_{\Ixp/\fC^P}=\tau^{\ge -1}\LL_{\Ixp/\fC^P}$$
be the truncated relative cotangent complex of $\Ixp\to\fC^P$.

\begin{prop}[{\cite[Prop 10]{MPT}}]\label{P6.1}
The Atiyah class constructed in \cite{HT} defines a perfect relative obstruction theory
\beq\label{po}\phi: E\dual \lra L_{\Ixp/\fC^P}.
\eeq
\end{prop}

We let $[\Ixp]\virt\in A_*\Ixp$ be the associated virtual class.

\begin{prop}
Let $c\ne 0\in C$, and let $i_c^!: A\lsta \Ixp\to A_{\ast-1} \fI_{X_c}^P$ be the Gysin map associate to the divisor $c\in C$. Then $i_c^! [\Ixp]\virt=[\fI_{X_c}^P]\virt$.
\end{prop}

\begin{proof}
This is because the obstruction theory of $\fI_{X_c}^P$ is the pull back of the relative obstruction theory of $\Ixp\to \fC^P$ via
$c\in C$ (c.f. \cite{BF}).
\end{proof}

Next we construct the virtual class of the relative Hilbert schemes.
In the subsequent discussion, we use that $Y=Y_-\cup Y_+$ is the union of $Y_-$ and $Y_+$.
We let $\delta=\{(\delta_+,\delta_0),(\delta_-,\delta_0)\}$ be two pairs of polynomials.

We denote by $\fI_{\fY_+/\fA_\diamond}^{\delta_+,\delta_0}$ the moduli of stable relative
ideal sheaves on $\fD_+\sub \fY_+$ of pair Hilbert polynomial $(\delta_+,\delta_0)$.
For simplicity, we abbreviate it to $\cM^\delta_+$.
Let $\text{$L$\raisebox{-3pt}{$\scriptstyle{\cM^\delta_+/\scrAdel}$}}$ 
be the truncated relative cotangent complex of
$\cM^\delta_+=\fI_{\fY_+/\fA_\diamond}^{\delta_+,\delta_0}\to\fA^{{\delta}_+,{\delta}_0}_\diamond$.
Let
$$\pi_+: \cY_+\lra \fI_{\fY_+/\fA_\diamond}^{\delta_+,\delta_0}\and \sI_{\cZ_+}\sub \sO_{\cY_+}
$$
be the universal underlying family and the universal ideal sheaf of $\fI_{\fY_+/\fA_\diamond}^{\delta_+,\delta_0}$.

\begin{prop}[{\cite{MPT}}]\label{P6.4}
The Atiyah class in \cite{HT} defines a perfect relative obstruction theory
\beq\label{relob}
\phi_{+}: E_+\dual \defeq R\pi_{+\ast} \RHom(\sI_{\cZ_{+}},\sI_{\cZ_{+}})_0[1]\dual \lra \text{$L$\raisebox{-3pt}{$\scriptstyle{\cM^\delta_+/\scrAdel}$}}.
\eeq
\end{prop}

The obstruction theory defines its virtual class
$[\fI_{\fY_+/\fA_\diamond}^{\delta_+,\delta_0}]\virt\in A\lsta \fI_{\fY_+/\fA_\diamond}^{\delta_+,\delta_0}$.
By replacing the subscript ``$+$'' with ``$-$'', we obtain a parallel theory for $\cM^\delta_-\defeq \fI_{\fY_-/\fA_\diamond}^{\delta_-,\delta_0}$.

\subsection{Decomposition of the virtual cycle}

We study the decomposition of the virtual cycles of the central fiber $\Ixpo\defeq \Ixp\times_C0$.

We let $\Lambda_P^\spl$ be the collection of triples  $\delta=(\delta_-,\delta_+,\delta_0)$ of polynomials
in $\sA$ so that $\delta_++\delta_--\delta_0=P$.
Following the notation developed in Section 5, the morphism $\Ixp\to \fC$ lifts to $\pi_P:\Ixp\to \fC^P$.
Fixing a splitting data $\delta\in \Lambda_P^\spl$, we define the closed substack
$\Idel$ via the Cartesian diagram
\beq\label{diag0}
\begin{CD}
\Idel\defeq \Ixp\times_{\fC^P}\Cdel @>>> \Ixp\\
@VVV @V{\pi_P}VV \\
\fC_0^{\dagger,\delta}@>>>\fC^P.
\end{CD}
\eeq
We denote by $(L_\delta,s_\delta)$ the pair of the line bundle and the section for $\delta\in \Lam_P^{\text{spl}}$
constructed in Proposition \ref{prop2.24}. Then $\fC_0^{\dagger,\delta}=(s_\delta=0)\sub \fC^P$; and by Theorem \ref{3.36}, $\Idel=(\pi_P^*s_\delta=0)$. We define
$$c_1^\loc(L_\delta, s_\delta): A\lsta \Ixp\lra A_{\ast-1} \Idel
$$
be the localized first Chern class of $(L_\delta, s_\delta)$.
\vsp
We define the perfect relative obstruction theory of $\Idel\to\fC_0^{\dagger,\delta}$ by pulling back
the relative obstruction theory \eqref{po} of $\Ixp\to\fC^P$ via the diagram \eqref{diag0}:
\beq\label{phi-del}\phi_\delta: E_\delta\dual\defeq
R\pi_{\delta\ast} \RHom(\sI_{\cZ_\delta},\sI_{\cZ_\delta})_0[1]\dual \lra L_{\scrIdel/\fC_0^{\dagger,\delta}},
\eeq
where
$$\pi_\delta: \cX_\delta\to \Idel\and \sI_{\cZ_\delta}\sub \sO_{\cX_\delta}
$$
is the universal family
of $\Idel$, which is also the pull back of $(\cX,\sI_{\cZ})$ to $\Idel$ via the arrow in \eqref{diag0}.

Applying \cite{BF},  we get
\beq \label{cloc}
i_0^![\Ixp]\virt=[\Idel]\virt=c_1^\loc(L_\delta, s_\delta)[\Ixp]\virt.
\eeq

\begin{prop}\label{prop-virt-deco}
Let $\iota_\delta:  \Idel\to\Ixpo$ be the inclusion. We have an identity of cycle classes
\beq\label{decomp-virt}
i_0^![\Ixp]\virt=\sum_{\delta\in \Lambda_P^\spl}\iota_{\delta\ast}[\Idel]\virt.
\eeq
\end{prop}
\begin{proof}
This follows from item (1) of Theorem \ref{3.36} and the identity \eqref{cloc}.
\end{proof}

To reinterpret the terms in the summation of \eqref{decomp-virt}, we will express them in terms of the virtual class of relative
Hilbert schemes. For this, we will use the Cartesian product (keeping the abbreviation
$\fI_{\fY_\pm/\fA_\diamond}^{\delta_\pm,\delta_0}=\cM^\delta_{\pm}$)
\beq\label{diag}
\begin{CD}
\cMm\times_{\Hilb_D^{\delta_0}}\cMp@>u>>\cMm\times\cMp\\
 @VfVV @VV(\ev_-,\ev_+)V\\
 \Hilb_D^{\delta_0}@>\triangle>>\Hilb_D^{\delta_0}\times \Hilb_D^{\delta_0},
\end{CD}
\eeq
where $\ev_\pm$ are the evaluation morphisms and $\triangle$ is the diagonal morphism,
and use the isomorphism (c.f. Theorem \ref{3.36})
\beq\label{arrow-1}
\Phi_\delta:\cMm\times_{\Hilb_D^{\delta_0}}\cMp\lra \Idel.
\eeq
Note that the relative obstruction theory of $\Idel\to\fC_0^{\dagger,\delta}$
endows $\cMm\times_{\Hilb_D^{\delta_0}}\cMp\to\fC_0^{\dagger,\delta}$ a perfect relative obstruction theory;
also
\beq\label{arrow-2}
\cMm\times\cMp
\lra\fA^{{\delta}_-,{\delta}_0}_\diamond\times \fA^{{\delta}_+,{\delta}_0}_\diamond
\eeq
has a perfect relative obstruction theory induced
from that of its factors. We will compare these two obstruction theories.

We continue to denote by $\cX_\delta\to \Idel$ with $\sI_{\delta}\sub\sO_{\cX_\delta}$
(resp. $\cY_\pm\to\cM_\pm$ with $\sI_\pm\sub \sO_{\cY_\pm}$) the universal family of $\Idel$
(resp. $\cM^\delta_\pm$).
We let
$$\ti\cY_\pm=\cY_\pm\times_{\cM^\delta_\pm} \Idel\and
\ti\sI_\pm=\sI_\pm\otimes_{\sO_{\cY_\pm}}\sO_{\ti\cY_\pm},
$$
where $\Idel\to \cM^\delta_\pm$ is the composite of $\Phi_\delta\upmo$ (cf. \eqref{arrow-1}) with the
projection; we let
$\cD_\delta\sub\cX_\delta$ be the total space of the distinguished (marked) divisor (of
$\Idel$). We have the short exact sequence
\beq\label{seq-11}
0\lra \sI_\delta\lra \ti\sI_+\oplus\ti\sI_-\mapright{(1,-1)} \ti\sI_0\lra 0,
\eeq
where $\ti\sI_0\defeq \sI_\delta\otimes_{\sO_{\cX_\delta}}\sO_{\cD_\delta}$.
Because of the admissible requirement, $\ti\sI_0$ is an ideal sheaf of $\sO_{\cD_\delta}$,
and via the $f$ in \eqref{diag}, we have isomorphism as ideal sheaves of $\sO_{\cD_\delta}$:
\beq\label{Hilb}
\ti\sI_0\cong
\sI_{\cZ_D}\otimes_{\sO_{\Hilb_D^{\delta_0}}}\sO_{\scrIdel},
\eeq
where $\cZ_D\sub D\times \Hilb_D^{\delta_0}$ is the universal family of ${\Hilb_D^{\delta_0}}$.

Let $\pi_\delta: \cX_\delta\to \Idel$,
let $\ti\pi_\pm: \ti\cY_\pm\to \Idel$ and $\ti\pi_0: \cD_\delta\to \Idel$ be the corresponding projections.
According to \cite[p.961]{MPT}, we have the following commutative diagram of derived objects
\beq\notag
\begin{CD}
\LMdel[-1]@>>>\Lprod[-1]@>>>\LMdelt \\
@VVV @VVV @VVV @.\\
R\pi_{\delta\ast}\RHom(\sI_\delta,\sI_\delta)_0@>>>\bigoplus_{-,+} R\ti\pi_{\pm\ast}\RHom(\ti\sI_\pm,\ti\sI_\pm)_0
@>>> R\ti\pi_{0\ast}\RHom(\ti\sI_0,\ti\sI_0)_0, \\
\end{CD}
\eeq
where the vertical arrows are the dual of the perfect obstruction theories, and the lower sequence is part of the
distinguished triangle induced by \eqref{seq-11}.

We claim that, under the morphism $f$ in \eqref{diag},
\beq\label{Rf}
R\ti\pi_{0\ast}\RHom(\ti\sI_0,\ti\sI_0)_0\dual \cong
f\sta L_\triangle,\quad L_\triangle\defeq L_{\Hilb_D^{\delta_0}/\Hilb_D^{\delta_0}\times \Hilb_D^{\delta_0}},
\eeq
and via this isomorphism the last vertical arrow in the above diagram is identical to the canonical arrow
\beq\label{Lf}
\LMdelt \lra f\sta L_\triangle\dual.
\eeq
Indeed, since $\Hilb_D^{\delta_0}$ is smooth, and the conormal bundle of $\Hilb_D^{\delta_0}$
in $\Hilb_D^{\delta_0}\times \Hilb_D^{\delta_0}$ via the diagonal $\Del$ is
isomorphism to the cotangent sheaf
$\Omega_{\Hilb_D^{\delta_0}}$, we have $L_\triangle\cong\Omega_{\Hilb_D^{\delta_0}}[1]$.

Next, we let $\pi_H: D\times \Hilb_D^{\delta_0}\to \Hilb_D^{\delta_0}$ be the projection.
Then by the deformation of ideal sheaves of smooth surfaces, the derived objects
$$R\ti\pi_{H\ast}\RHom(\ti\sI_{\cZ_D},\ti\sI_{\cZ_D})_0\dual\cong \Omega_{\Hilb_D^{\delta_0}}[1].
$$
By the isomorphism \eqref{Hilb}, we have canonical isomorphism
$$R\ti\pi_{0\ast}\RHom(\ti\sI_0,\ti\sI_0)_0 \cong f\sta R\ti\pi_{H\ast}\RHom(\ti\sI_{\cZ_D},\ti\sI_{\cZ_D})_0.
$$
Combined, we have \eqref{Rf}, and that the last vertical arrow is identical to the
\eqref{Lf}.
\vsp

Applying \cite{BF}, we have

\begin{prop}\lab{decompose}
The perfect relative obstruction theories of $\Idel$  and of \eqref{arrow-2}
are compatible with respect to the fiber diagram \eqref{diag} (using \eqref{arrow-1}). Consequently, we have the identity
\beq\lab{comp}
[\Idel]\virt =\triangle^!\bl [\cMm]\virt\times
[\cMp]\virt\br.
\eeq
\end{prop}

We state the cycle version of the degeneration of Donaldson-Thomas invariants.

\begin{theo}\label{cycle}
Let $X/C$ be a simple degeneration of projective threefolds such that $X_0=Y_-\cup Y_+$ is a
union of two smooth irreducible components. Let $[\Ixp]\virt\in A_*\Ixp$ be the virtual class of the
good degeneration, and let $\triangle$ be the diagonal morphism in \eqref{diag}. Then
$i_c^![\Ixp]\virt=[\fI_{X_c}^P]\virt$ for $c\ne 0\in C$, and
\beq\label{formula}i_0^! [\Ixp]\virt=\sum _{\delta \in\Lam_P^{\mathrm{spl}} }
\triangle^!\bl [\cMm]\virt\times
[\cMp]\virt\br.
\eeq
\end{theo}

\begin{coro}
Let the situation be as in Theorem \ref{cycle}. Suppose $X_c$ are Calabi-Yau threefolds for $c\ne 0$.
Then
\beq\lab{formula1}\deg \, [\fI^P_{X_c}]\virt=\sum _{\delta \in\Lam_P^{\mathrm{spl}} }\deg
\bl \ev_{-\ast}[\cMm]\virt\bullet
\ev_{+\ast}[\cMp]\virt\br,
\eeq
where $\ev_\pm\colon \fI_{\fY_\pm/\fA_\diamond}^{\delta_\pm,\delta_0}=\cM^\delta_\pm\to \Hilb^{\delta_0}_D$ is the restriction morphism, and $\bullet$ is the intersection pairing in $A\lsta \Hilb^{\delta_0}_D$.
\end{coro}

\begin{proof}
The Theorem follows from Proposition \ref{prop-virt-deco} and \ref{decompose}.
The Corollary follows from the Theorem and that
$\deg\, i_c^![\Ixp]\virt= \deg\, i_0^![\Ixp]\virt$.
\end{proof}

\subsection{The degeneration formula}

We prove Theorem \ref{degen} in the Introduction, whose formulation is due to \cite{MNOP2}.

Let the situation be as in Theorem \ref{degen} and \ref{cycle}.
We define descendant invariants, following \cite{MNOP2}. We continue to denote by
$$\pi:\cX\lra\Ixp,\ \pi_X: \cX\lra X,\and \sI_{\cZ}\sub \sO_{\cX}
$$
be the universal family on $\Ixp$.
Since locally $\sI_{\cZ}$ admits locally free resolutions of finite length, the Chern character
$$\ch(\sI_{\cZ}): A\lsta \cX\lra A\lsta \cX
$$
is well defined.

For any $\gamma\in H^l(X,\ZZ)$, we define
\beq \label{ch}\ch_{k+2}(\gamma):H^{BM}_*(\Ixp,\QQ)\lra H^{BM}_{*-2k+2-l} (\Ixp,\QQ)
\eeq
($H^{BM}_i$ is the Borel-Moore homology) via
\[ \ch_{k+2}(\gamma)(\xi)=\pi\lsta(\ch_{k+2}(\sI_{\cZ})\cdot \pi_X^*(\gamma)\cap\pi^*(\xi)),
\]
where $\pi^*$ is the flat pullback.

For cohomology classes $\gamma_{i}\in H^{l_i}(X,\ZZ)$ of pure degree $l_i$, we define
\[\Bigl\langle\prod_{i=1}^r\tilde\tau_{k_i}(\gamma_{i})\Bigr\rangle_{\X}^P=
\Bigl[\prod_{i=1}^r(-1)^{k_i+1}\ch_{k_i+2}(\gamma_{i})\cdot [\Ixp]\virt\Bigr]_2\in H_2^{BM}(\Ixp,\QQ),
\]
where the term inside the bracket is a homology class of dimension $2\dim [\Ixp]\virt - \sum_{i=1}^r (2k_i-2+l_i)$, and
the $[\cdot]_2$ is taking the dimension two part of the term inside the bracket. This is the family version of the descendent Donaldson-Thomas invariants given in \cite{MNOP2}:
\[\Bigl\langle\prod_{i=1}^r\tilde\tau_{k_i}(\gamma_{i})\Bigr\rangle_{X_c}^P=
\Bigl[ \prod_{i=1}^r(-1)^{k_i+1}\ch_{k_i+2}(\gamma_{i})\cdot [\fI_{X_c}^P]\virt\Bigr]_0\in H_0^{BM}(\fI_{X_c}^P,\QQ).
\]

Since $P$ has degree one, we let $P(v)=d\cdot v+n$. We form the partition function of descendent
Donaldson-Thomas invariants of $X_c$
$$Z_{d}\bigg(X_c;q\Big |\prod_{i=1}^r \tilde \tau_{k_i}(\gamma_{i})
\bigg)=\sum_{n\in \ZZ}\deg\Bigl\langle\prod_{i=1}^r\tilde\tau_{k_i}(\gamma_{i})\Bigr\rangle_{X_c}^{d\cdot v+n}q^n.$$

Accordingly, for the relative Hilbert schemes $\fI_{\fY_\pm/\fA_\diamond}^{\delta_\pm,\delta_0}$, we define
$\ch_{k+2}(\gamma)$ similarly, and
\[
\Bigl\langle\prod_{i=1}^r\tilde\tau_{k_i}(\gamma_{i})\Bigr\rangle_{\Y_\pm}^{\delta_\pm}=
\ev_{\pm\ast}\bl \prod_{i=1}^r(-1)^{k_i+1}\ch_{k_i+2}
(\gamma_{i})\cdot [\fI_{\fY_\pm/\fA_\diamond}^{\delta_\pm,\delta_0}]\virt\br\in H\lsta(\Hilb_D^{\delta_0},\QQ).
\]
Let $\beta_1,\cdots,\beta_m$ be a basis of $H^*(D,\QQ)$. Let $\{C_\eta\}_{|\eta|=k}$ be a Nakajima basis of the cohomology of
$\Hilb_D^k$, where $\eta$ is a cohomology weighted partition w.r.t. $\beta_i$.
The relative DT-invariants with descendent insertions \cite{MNOP2} are
\[\Bigl\langle\prod_{i=1}^r\tilde\tau_{k_i}(\gamma_{i}) \Big | \eta\Bigr\rangle_{\Y_\pm}^{\delta_\pm}=
\Bigl[ \prod_{i=1}^r(-1)^{k_i+1}\ch_{k_i+2}(\gamma_{i})\cap\ev_\pm^*(C_\eta)\cdot
[\fI_{\fY_\pm/\fA_\diamond}^{\delta_\pm,\delta_0}]\virt\Bigl]_0,
\]
which form a partition function
\[ Z_{d_\pm,\eta}\bigg(Y_\pm,D_\pm;q\Big|\prod_{i=1}^r\tilde \tau_{k_i}
(\gamma_{i})\bigg)=\sum_{n\in \ZZ}\deg\Bigl\langle\prod_{i=1}^r\tilde\tau_{k_i}(\gamma_{i})
\Big | \eta\Bigr\rangle_{\Y_\pm}^{d_\pm\cdot v+n}q^n.
\]

\begin{theo}[Theorem \ref{degen}]\lab{degen2}
Fix a basis $\beta_1,\cdots,\beta_m$ of $H^*(D,\QQ)$. The degeneration formula of
Donaldson-Thomas invariants has the following form
\begin{align*}
\notag Z_{d}\bigg(X_c;q\Big|\prod_{i=1}^r  \tilde \tau_{k_i}(\gamma_{i})
\bigg)
= &\sum_{\substack {d_-,d_+;\,\eta\\d=d_-+d_+}} \frac{(-1)^{|\eta|-l(\eta)}\fz(\eta)}{q^{|\eta|}}\cdot\\
\cdot Z_{d_-,\eta}\bigg(Y_-,D_-;q\Big|&\prod_{i=1}^r\tilde \tau_{k_i}(\gamma_{i})
\bigg)\cdot
 Z_{d_+,\eta^\vee}
\bigg(Y_+,D_+;q\Big|\prod_{i=1}^r\tilde
\tau_{k_i}(\gamma_{i})\bigg),
\end{align*}
where $\eta$ are cohomology weighted partitions w.r.t. $\beta_i$, and $\fz(\eta)=\prod_i \eta_i|\Aut(\eta)|$.
\end{theo}

\begin{proof}
Since Gysin maps commute with proper pushforward and flat pullback, we have
$$\deg i_c^!\, \Bigl\langle\prod_{i=1}^r\tilde\tau_{k_i}(\gamma_{i})\Bigr\rangle_{\X}^{P}=
\deg\, i_0^!\Bigl\langle\prod_{i=1}^r\tilde\tau_{k_i}(\gamma_{i})\Bigr\rangle_{\X}^{P}.
$$
By $i_c^![\Ixp]\virt=[\fI_{X_c}^P]\virt$,
the left hand side term equals to $\deg\,\Bigl\langle\prod_{i=1}^r\tilde\tau_{k_i}(\gamma_{i})\Bigr\rangle_{X_c}^P$,
which is the Donaldson-Thomas invariants of $X_c$.
\vsp
For the other term, we will decompose it into relative invariants using \eqref{formula}.
Applying the operation $\prod_{i=1}^r(-1)^{k_i+1}\ch_{k_i+2}(\gamma_{i})$ to both sides of \eqref{formula},
and using the restriction morphism $\ev_\pm\colon \fI_{\fY_\pm/\fA_\diamond}^{\delta_\pm,\delta_0}\to \Hilb^{\delta_0}_D$, and
\[
\Bigl\langle\prod_{i=1}^r\tilde\tau_{k_i}(\gamma_{i})\Bigr\rangle_{\Y_\pm}^{\delta_\pm}=
\ev_{\pm\ast}\bl \prod_{i=1}^r(-1)^{k_i+1}\ch_{k_i+2}
(\gamma_{i})\cdot [\fI_{\fY_\pm/\fA_\diamond}^{\delta_\pm,\delta_0}]\virt\br\in H\lsta(\Hilb_D^{\delta_0},\QQ),
\]
we obtain
\begin{align}\label{lid}
\deg i_0^!\,\Bigl\langle\prod_{i=1}^r\tilde\tau_{k_i}(\gamma_{i})\Bigr\rangle_{\X}^{P}=\sum _{\delta \in\Lam_P^{\text{spl}} }
\deg\Big( \Bigl\langle\prod_{i=1}^r\tilde\tau_{k_i}(\gamma_{i})\Bigr\rangle_{\Y_-}^{\delta_-}  \bullet
\Bigl\langle\prod_{i=1}^r\tilde\tau_{k_i}(\gamma_{i})\Bigr\rangle_{\Y_+}^{\delta_+} \Big).
\end{align}

Let $\beta_1,\cdots,\beta_m$ be a basis of $H^*(D,\QQ)$, and let $\eta$ be a cohomology weighted partition with
respect to $\beta_i$. Following the notation in \cite{MNOP2,Nak}, we denote
$$C_\eta=\frac{1}{\fz(\eta)}P_{\delta_1}[\eta_1]\cdots P_{\delta_s}[\eta_s]\cdot \mathbf{1}\in H^*(\Hilb_D^{|\eta|},\QQ)$$
with $\fz(\eta)=\prod_i\eta_i|\Aut(\eta)|$. Then $\{C_\eta\}_{|\eta|=k}$ is the Nakajima basis of the cohomology of $\Hilb_D^{k}$,
and the Kunneth decomposition of the diagonal class
$[\triangle]\in H^*(\Hilb_D^k\times\Hilb_D^k,\QQ)$ takes the form
\[[\triangle]=\sum_{|\eta|=k}(-1)^{k-l(\eta)}{\fz(\eta)}C_\eta\otimes C_{\eta^\vee}.\]
Since
\[\Bigl\langle\prod_{i=1}^r\tilde\tau_{k_i}(\gamma_{i}) \Big | \eta\Bigr\rangle_{\Y_\pm}^{\delta_\pm}=
\Bigl[ \prod_{i=1}^r(-1)^{k_i+1}\ch_{k_i+2}(\gamma_{i})\cap\ev_\pm^*(C_\eta)\cdot [\fI_{\fY_\pm/\fA_\diamond}^{\delta_\pm,\delta_0}]\virt\Bigl]_0
\]
is an element in $H_0^{BM}(\Hilb_D^{\delta_0},\QQ)$, applying to \eqref{lid}, we have
\begin{align*}
&\deg\,i_0^!\Bigl\langle\prod_{i=1}^r\tilde\tau_{k_i}(\gamma_{i})\Bigr\rangle_{\X}^{P}\\\notag
=&\sum _{\delta \in\Lam_P^{\text{spl}}; |\eta|=\delta_0 }
(-1)^{|\eta|-l(\eta)}{\fz(\eta)}\deg\,\Bigl\langle\prod_{i=1}^r\tilde\tau_{k_i}(\gamma_{i}) \Big | \eta\Bigr\rangle_{\Y_-}^{\delta_-} \cdot\deg\,
\Bigl\langle\prod_{i=1}^r\tilde\tau_{k_i}(\gamma_{i}) \Big | \eta^\vee\Bigr\rangle_{\Y_+}^{\delta_+} .
\end{align*}

Finally, we form the partition functions of these invariants. Notice that $\delta_-+\delta_+-\delta_0=P$,
which accounts for the shift of the power of $q$. This proves Theorem \ref{degen2}.
\end{proof}

\subsection{Degeneration of stable pair invariants}
We fix a simple degeneration $\pi:X\to C$ of projective threefolds with a $\pi$-ample $H$ on $X$;
we suppose that $X_0=Y_-\cup  Y_+$ is a union of two smooth irreducible components.
For reference, we state the degeneration of PT-invariants, which is proved in \cite{MPT}.

Recall that the coherent systems we considered are homomorphisms
$\varphi:\sO_X\lra \sF$
so that $\sF$ is pure of dimension one and $\varphi$ has finite cokernel. Let $P$ be a degree one polynomial.
Let $\Mxh$ be the good degeneration of the moduli of coherent systems
constructed in this paper. It is a separated and proper Deligne-Mumford stack of finite type over $C$.
We use the relative obstruction theory of $\Mxh\to\C^P$ introduced in \cite{PT} to construct its virtual class $[\Mxh]\virt$.

Let
$\pi: \cX\to\Mxh$ and $\varphi:\sO_{\cX}\lra \sF$
be the universal family of $\Mxh$, and let $\sI^\bullet\in D^b(\cX)$ be the object corresponds to the complex
$[\sO_{\cX}\to\sF]$ with $\sO_{\cX}$ in degree $0$.
We denote by $L_{\Mxh/\fC^P}$ be the truncated relative cotangent complex of $\Mxh\to\fC^P$.
In \cite[Prop 10]{MPT}, using the Atiyah classes a perfect relative obstruction theory
is constructed:
$$ E\dual\defeq R\pi_{\ast}\RHom(\sI^\bullet, \sI^\bullet)_0[1]\dual\lra L_{\Mxh/\fC^P}.
$$
Let $[\Mxh]\virt\in A_*\Mxh$ be its associated virtual cycle.
In the same paper, for any partition $\delta=(\delta_\pm,\delta_0)$, a perfect relative obstruction theory of
$\P_{\fY_\pm/\fA_\diamond}^{\delta_\pm,\delta_0}\to\fA_\diamond^{\delta_\pm,\delta_0}$ is also constructed,
which gives its virtual class
$[\P_{\fY_\pm/\fA_\diamond}^{\delta_\pm,\delta_0}]\virt\in A_*\P_{\fY_\pm/\fA_\diamond}^{\delta_\pm,\delta_0}$.

\vsp
Let $c\in C$ and $\P^P_{\fX_c/\fC_c}=\Mxh\times_C c$. Let
$$i_c^!:A_*\Mxh\to A_*\P^P_{\fX_c/\fC_c}
$$
be the Gysin map. By Theorem \ref{3.36pri}, we can decompose $\P^P_{\fX_0/\fC_0}$ as a union of
$\P_{\fX_0\shar/\fC_0\shar}^{\delta}$,
$\delta\in \Lam_P^{\mathrm{spl}}$, and obtain the isomorphism \eqref{glue-P}.
By going through the argument parallel to the proof of degeneration formula for Hilbert schemes of ideal sheaves, Maulik, Pandharipande and Thomas proved in \cite{MPT} the degeneration formula of PT stable pair invariants.

\begin{theo}[Maulik, Pandharipande and Thomas]\label{cycle2}
Let $X/C$ be a simple degeneration of projective threefolds such that $X_0=Y_-\cup Y_+$ is a
union of two smooth irreducible components. Then
$$i_c^![\Mxh]\virt=[\P_{X_c}^P]\virt\in A\lsta \P_{X_c}^P \quad \text{for } c\ne 0\in C,
$$
and
$$
i_0^! [\Mxh]\virt=\sum _{\delta \in\Lam_P^{\mathrm{spl}} }
\triangle^!\bl [\P_{\fY_-/\fA_\diamond}^{\delta_-,\delta_0}]\virt\times
[\P_{\fY_+/\fA_\diamond}^{\delta_+,\delta_0}]\virt\br,
$$
where $\triangle:\Hilb_D^{\delta_0}\to \Hilb_D^{\delta_0}\times \Hilb_D^{\delta_0}$ is the diagonal morphism.
\end{theo}

\begin{appendix}

\section{Proof of Lemmas \ref{A} and \ref{B}}
\begin{proof}[Proof of Lemma \ref{A}]
First, because $M_I\otimes_A A_0\to M_0$ is injective and its image lies in $(M_0)_{I_0}$,
$M_I\otimes_A A_0\to (M_0)_{I_0}$ is injective. We next show that it is surjective.

Since $M_I\otimes_A A_0\to (M_0)_{I_0}$ is $\Gm$-equivariant, it suffices to show that every weight $\ell$
element in $(M_0)_{I_0}$ can be lifted to a weight $\ell$ element in $M_I\otimes_A A_0$.
Let $v\in (M_0)_{I_0}$ be a weight $\ell$ element. We first lift $v$ to a weight $\ell$ element
$\bar v\in R_0$; we write
$$\bar v=\alpha_0+z_1\alpha_1+\ldots+z_1^p \alpha_p,\quad \alpha_i\in A[z_2]^{\oplus m}.
$$
Let
$$K=\ker\{\varphi: R\lra M\},\quad K_0=\ker\{ \varphi\otimes_A A_0: R_0\lra M_0\}.
$$
By the definition of $(M_0)_{I_0}$, there is a power $z_1^k$, $k>0$,  so that $z_1^k \bar v\in K_0$.
Because $M$ is $\kk[t]$-flat, tensoring the exact sequence $0\to K\to R\to M\to 0$ with $A_0$,
we obtain an exact sequence $0\to K\otimes_A A_0\to R_0\to M_0\to 0$. Therefore,
$$K\otimes_AA_0=K_0.
$$

We let $w\in K$ be a lift of $z_1^k \bar v\in K_0$. We write $w$ in the form
$$w=w_0+tw_1+\cdots +t^r w_r, \quad w_i\in R'\defeq B[z_1,z_2]/(z_1z_2)^{\oplus m}.
$$
Since $M_0$ only contains
elements of non-negative weights, $\ell\ge 0$. Thus $w$ has weight $\ell+ka$. Since $a>0$, and
since the weights of $w_i$ are $\ell+ka-bi>ka$, we have $w_i=z_i^{k} w_i'$ for $w_i'\in P'$.
For $w_0'$, we can choose it to be $w_0'=\alpha_0+\ldots z_1^p \alpha_p$.
We let
$$w'=w_0'+tw_1'+\ldots+t^rw_r'.
$$
Then $\varphi(w')\in M$ is a lift of $v\in (M_0)_{I_0}$.

We claim $\varphi(w')\in M$ is annihilated by $z_1^k$.
This is true because
$z_1^k\cdot \varphi(w')=\varphi(z_1^k w')=\varphi(w)=0$,
since $w\in K$.
We show that $\varphi(w')$ is also annihilated by a power of $z_2$. 
We distinguish two cases. The first is when $\ell>0$. In this case, the weight of $w_i'$ are $\ell-ib\ge \ell>0$,
thus $z_1| w_i'$. Hence $z_2\varphi(w')=\varphi(z_2 w')=0$.

The other case is when $\ell=0$. In this case, we still have $z_2w_i'=0$ for $i> 0$.
We claim that for some $h>0$, $z_2^h\varphi(w'_0)=0$.
We pick a $z_2^{h}$ so that $z_2^{h} \bar v\in K_0$, (this is possible since $v\in (M_0)_{I_0}$), and then lift $z_2^{h} \bar v$ to
a weight $0$ element $\ti w
\in K$. We write
$$\ti w=\ti w_0+t\ti w_1+\ldots+t^s \ti w_s, \quad \ti w_i\in R'.
$$
Then $\ti w_{i>0}$ has positive weight in $R'$, thus are annihilated by $z_2$,
and $z_2\ti w=z_2\ti w_0$.
Therefore by replacing $h$ by $h+1$, we can assume $\ti w_{i>0}=0$, and $\ti w_0$ is expressed as
an element in $B[z_2]^{\oplus m}$.

Since $\ti w_0$ is a lift of $z_2^h \bar v=z_2^h \alpha_0$, and since both are expressed as elements in
$A[z_2]^{\oplus m}$, we have $\ti w_0=z_2^h \alpha_0$. Therefore, since $\ti w=\ti w_0\in K$,
$$z_2^h \varphi(w')=\varphi(z_2^h w')=\varphi(z_2^h w'_0)=\varphi(z_2^h \alpha_0)=\varphi(\ti w_0)=\varphi(\ti w)=0.
$$
This proves that $\varphi(w')$ lies in $M_I$ and is a lift of $v\in (M_0)_{I_0}$. This proves the lemma.
\end{proof}

\begin{proof}[Propf of Lemma \ref{B}]
Let $\beta\in C_{\text{gen}}$. Then there are $x\in K^-$, $t^k$ and $z_1^h\in A$ such that
$$x=t^kz_1^h\beta\mod (t^{k+1}, z_1^{h+1}).
$$
Since the modules involved are $\Gm$-equivariant, we can assume that $x$ has weight $ah+bk$.
Thus after expressing $x$ as
$$x=t^kz_1^h\beta+t^{k+1}\beta_1+z_1^{h+1} \beta_2,\quad
\beta_1,\beta_2\in B[z_1,t],
$$
and plus the weight consideration, we conclude
$\beta_2=t^{k+1}\beta_2'$ for a $\beta_2'\in B[z_1,t]^{\oplus m}$.
Therefore,
$z_1x=t^k(z_1^{h+1}\beta+t\beta_3)\in K$, where $\beta_3\in B[z_1,t]^{\oplus m}$.
Since $K\sub R$, we conclude $z_1^{h+1}\beta+t\beta_3\in K$. In particular, $z_1^{h+1}\beta\in K_0$
and $\beta\in (K_0^-)_{(z_1)}\cap R_0^-$. This proves $C_{\text{gen}}\sub C_0$.

For the other direction, we let $\gamma\in C_0$. For the same reason, for a positive $h$ and a
weight $ah$, $y\in K_0^-$, $y=z_1^h\gamma_1+z_1^{h+1}\gamma_2$, $\gamma_i\in A[z_1]^{\oplus m}$.
Since $z_1z_2=0$ in $B$, $y\in K_0$. We let $\ti y\in K$ be a weight $ah$ lifting of $y$,
expressed in the form
$$\ti y=(z_1^h\gamma_1+z_1^{h+1}\gamma_2)+tf_1+\ldots+t^q f_q,\quad
f_i\in R'.
$$
Since $\ti y$ has weight $ah$, we conclude that $f_i=z_1^{h+1}f_i'$, for some $f_i'\in B[z_1]^{\oplus m}$.
Therefore, $\ti y=z_1^h(\gamma+z_1\gamma_3)$, for a $\gamma_3\in B[z_1,t]^{\oplus m}$,
and hence
$$\gamma+z_1\gamma_3\in (K^-\llt)_{(z_1)}\cap R^-\llt.
$$
This implies that $\gamma$ lies in \eqref{Cgen}, and thus lies in $C_0$. This proves the Lemma.
\end{proof}

\end{appendix}

\bibliographystyle{alpha}

\end{document}